\documentclass[12pt,letterpaper]{amsart}
\usepackage[svgnames]{xcolor}
\usepackage{fullpage}
\usepackage{hyperref}
\usepackage{amsmath,amsthm,amssymb}
\usepackage{mathtools} 
\usepackage{float}
\usepackage{tikz,tikz-cd}
\usetikzlibrary{positioning}
\usetikzlibrary{matrix}
\usepackage[english]{babel}
\usepackage{calligra,mathrsfs}
\usepackage[backend=bibtex,sorting=anyt]{biblatex}
\usepackage[capitalize,noabbrev]{cleveref}
\usepackage[inline]{enumitem}

\newtheorem{theorem}{Theorem}[section]
\newtheorem{corollary}[theorem]{Corollary}
\newtheorem{proposition}[theorem]{Proposition}
\newtheorem{lemma}[theorem]{Lemma}

\theoremstyle{definition}
\newtheorem{definition}[theorem]{Definition}
\newtheorem{example}[theorem]{Example}
\theoremstyle{remark}
\newtheorem{remark}[theorem]{Remark}

\newcommand{\scHom}{\mathscr{H}\text{\kern -5pt {\calligra\large om}}\,}
\newcommand{\shtensor}{\otimes_{\mathbf{sh}}}
\newcommand{\grtensor}{\otimes_{\mathbf{gr}}}

\newcommand{\R}{\mathbb{R}}

\DeclareMathOperator{\Hom}{Hom}
\DeclareMathOperator{\op}{op}
\DeclareMathOperator{\Ran}{Ran}

\DeclareMathOperator{\res}{res}
\DeclareMathOperator{\stalk}{stalk}
\DeclareMathOperator*{\colim}{colim}

\newcommand{\uModRP}{\underline{\mathbf{Mod}}_R^{\mathbf{P}}}

\newcommand{\isom}{\cong}
\newcommand{\isomto}{\xrightarrow{\isom}}

\newcommand{\Vectk}{\mathbf{Vect}_{\mathbf{k}}}
\newcommand{\tensor}{\otimes}
\newcommand{\uHom}{\underline{\Hom}}

\newcommand{\cat}{\mathbf}
\newcommand{\Open}{\cat{Open}}
\newcommand{\xto}{\xrightarrow}
\newcommand{\Shv}{\cat{Shv}}
\newcommand{\Coshv}{\cat{Coshv}}

\newcommand{\Z}{\mathbb{Z}}
\newcommand{\Q}{\mathbb{Q}}
\newcommand{\ModR}{\mathbf{Mod}_{R}}
\newcommand{\RMod}{\prescript{}{R}{\mathbf{Mod}}}
\newcommand{\shRMod}[1]{{#1}\text{-}\mathbf{Mod}}
\newcommand{\shModR}[1]{\mathbf{Mod}\text{-}{#1}}
\renewcommand{\varprojlim}{\lim}

\title{Homological algebra for persistence modules}
\author{Peter Bubenik, Nikola Mili\'cevi\'c}

\tikzset{%
    symbol/.style={%
        draw=none,
        every to/.append style={%
            edge node={node [sloped, allow upside down, auto=false]{$#1$}}}
    }
}

\begin{document}

\begin{abstract}
We develop some aspects of the homological algebra of persistence modules, in both the one-parameter and multi-parameter settings, considered as either sheaves or graded modules. The two theories are different. We consider the graded module and sheaf tensor product and Hom bifunctors as well as their derived functors, Tor and Ext, and give explicit computations for interval modules. We give a classification of injective, projective, and flat interval modules. We state K\"unneth {T}heorems and {U}niversal {C}oefficient Theorems for the homology and cohomology of chain complexes of persistence modules in both the sheaf and graded module settings and show how these theorems can be applied to persistence modules arising from filtered cell complexes. We also give a Gabriel-Popescu {T}heorem for persistence modules. Finally, we examine categories enriched over persistence modules. We show that the graded module point of view produces a closed symmetric monoidal category that is enriched over itself.
\end{abstract}

\maketitle

\section{Introduction}
\label{section:introduction}

 In topological data analysis, one often starts with application data that has been preprocessed to obtain a digital image or a finite subset of a metric space, which is then turned into a diagram of topological spaces, such as cubical complexes or simplicial complexes. Then one applies an appropriate homology functor with coefficients in a field to obtain a diagram of vector spaces. In cases where the data is parametrized by a number of real variables, this diagram of vector spaces is indexed by $\R^n$ or a subset of $\R^n$. Such a diagram is called a (multi-parameter) persistence module.
  These persistence modules have a rich algebraic structure.
  The indexing set $\R^n$ is an abelian group under addition and has a compatible coordinate-wise partial order.
  The sub-poset generated by the origin is the positive orthant which is a commutative monoid under addition which acts on { persistence modules.}
  The category of vector spaces has kernels and cokernels with desirable properties; it is a particularly nice abelian category. 
  This algebraic structure underlies the power of topological data analysis.  We will exploit this algebraic structure and apply some of the tools of homological algebra to study persistence modules. 
                       
To facilitate a broad class of present and future applications, we generalize the above setting somewhat. We replace $\R^n$ with a {preordered} set with a compatible abelian group structure. We replace the category of 
vector spaces with any Grothendieck category. This is an abelian category satisfying additional properties useful for homological algebra. We call the resulting diagrams persistence modules. It follows that the category of persistence modules is also a Grothendieck category.
Thus, for example, persistence modules satisfy the Krull-Remak-Schmidt-Azumaya Theorem~\cite{bubenik2018wasserstein}.

The algebraic structure of persistence modules has been studied from a number of points of view, for example,  as graded modules~\cite{MR2121296,MR2506738,MR3348168}, as functors~\cite{MR3201246,MR3413628}, and as sheaves~\cite{MR3259939}.
Here we develop some aspects of the homological algebra of persistence modules, with an emphasis on the graded module and sheaf-theoretic points of view.
%
From both sheaf theory and graded module theory, we define tensor product and Hom bifunctors for persistence modules as well as their derived functors Tor and Ext (Section~\ref{section:tensors}, Section~\ref{section:homs}, Section~\ref{section:derived_functors_of_persistence_modules}). We provide explicit formulas for the interval modules arising from the persistent homology of sublevel sets of functions.
In computational settings, single-parameter persistence modules decompose into direct sums of finitely many such interval modules. So, in the computational setting, since these four functors preserve finite direct sums, the general case reduces to that of interval modules.
For example, we have the following.
\begin{proposition}
 Suppose $\mathbf{k}[a,b)$ and $\mathbf{k}[c,d)$ are interval modules. Then:
\begin{itemize}
\item $\mathbf{k}[a,b)\grtensor \mathbf{k}[c,d)=\mathbf{k}[a+c,\min\{a+d,b+c\})$
\item $\uHom(\mathbf{k}[a,b),\mathbf{k}[c,d))=\mathbf{k}[\max\{c-a,d-b\},d-a)$
\item $\mathbf{Tor}^{\mathbf{gr}}_1(\mathbf{k}[a,b),\mathbf{k}[c,d))=\mathbf{k}[\max\{a+d,b+c\},b+d)$
\item $\mathbf{Ext}_{\mathbf{gr}}^1(\mathbf{k}[a,b),\mathbf{k}[c,d))=\mathbf{k}[c-b,\min\{c-a,d-b\})$
\end{itemize}
\end{proposition}
\noindent
The sheaf theoretic Tor bifunctor is trivial (Theorem~\ref{theorem:sheaf_tensor_is_exact}), but the Ext bifunctor is not (Example~\ref{example:sheaf_ext_of_interval_modules}).

A necessary step for computations in homological algebra is understanding projective, injective, and flat modules. We give a classification of these for single-parameter interval modules in Section~\ref{subsection:classification_of_interval_modules} and also extend the result somewhat to the multi-parameter setting.

\begin{theorem}
\label{theorem:classification_into_flats_and_injectives}
Let $a\in \mathbb{R}$. Then:
\begin{itemize}
\item The interval modules $\mathbf{k}(-\infty,a)$ and $\mathbf{k}(-\infty,a]$ are injective. They are not flat and thus not projective.
\item The interval modules $\mathbf{k}[a,\infty)$ are projective (free) and the interval modules $\mathbf{k}(a,\infty)$ are flat but not projective. Both are not injective.
\item The interval module $\mathbf{k}[\mathbb{R}]$ is both injective and flat, but not projective.
\item If $I\subset \mathbb{R}$ is a bounded interval, then $\mathbf{k}[I]$ is neither flat (hence not projective) nor injective.
\end{itemize}
\end{theorem}

In both the graded module and sheaf settings, we have K\"unneth Theorems and Universal Coefficient Theorems for homology and cohomology (Section~\ref{section:kunneth}). {There is evidence to suggest these theorems can be used to give faster algorithms for computing persistent homology \cite{gakhar2019knneth}.} We  compute a number of examples for these theorems in Section~\ref{section:kunneth}.
In addition to our main results, we discuss (Matlis) duality (Section~\ref{section:duality}),
persistence modules indexed by finite posets (Section~\ref{section:pers_modules_over_finite_posets}) and we state {the} Gabriel-Popescu {Theorem} for persistence modules (Corollary~\ref{cor:gp}). Matlis duality helps us identify injective and flat modules and is used extensively in proving Theorem~\ref{theorem:classification_into_flats_and_injectives}. The Gabriel-Popescu {T}heorem characterizes Grothendieck categories as quotients of module categories. {For persistence modules over finite posets we show a stronger result is true; persistence modules are isomorphic to modules over the ring $\text{End}(U)$ where $U$ is a generator of the Grothendieck category of persistence modules. This allows one to study persistence modules as modules over a non-graded ring.} 
In Section~\ref{section:enriched}, we consider persistence modules from the point of view of enriched category theory. We show that by viewing persistence modules as graded modules we obtain a closed symmetric monoidal category that is enriched over itself. 
Many of these results are adaptations or consequences of well-known results in graded module theory, sheaf theory, and enriched category theory. However, we hope that by carefully stating our results for persistence modules and providing numerous examples we will facilitate new computational approaches to topological data analysis. {
  For example, the magnitude of persistence modules \cite{govc2020persistent} is a new numerical invariant that respects the monoidal structure of the graded module tensor product of persistence modules.}

\subsection*{Related work}

Some of the versions of the K\"unneth {T}heorems that appear here were independently discovered by Polterovich, Shelukhin, and Stojisavljevic~\cite{polterovich2017persistence}, and Gakhar and Perea~\cite{gakhar2019knneth}.
Recent papers on persistence modules as graded modules include \cite{miller2017data,harrington2019stratifying,lesnick2015interactive} where they are considered from the perspective of commutative algebra. Recent papers from the sheaf theory point of view include \cite{kashiwara2018persistent,berkouk2019stable,
berkouk2019ephemeral}.
Results akin to Theorem~\ref{theorem:classification_into_flats_and_injectives} also appear in \cite{botnan2018decomposition,Hoppner1983}.
In the final stages of preparing this paper a preprint of Carlsson and Fillipenko appeared~\cite{carlsonfilippenko2019}, which covers some of the same material considered here, in particular graded module K\"unneth {T}heorems, but from a complementary point of view.
Grothendieck categories have been used to define algebraic Wasserstein distances for persistence modules~\cite{bubenik2018wasserstein}.
{Enriched categories over monoidal categories have been used in other recent work in applied topology \cite{leinster2017magnitude,cho2019quantales}.

\section{Persistence modules}
\label{section:Equivalent_categories}

In this section we consider persistence modules from several points of view {and provide background for the rest of the paper. In particular, we consider persistence modules as functors, sheaves and graded modules. We show that these points of view are equivalent. Thus, the reader may read the paper from their preferred viewpoint. What the different perspectives bring to the table are canonical operations from their respective well developed mathematical theories.}

\subsection{Persistence modules as functors}
\label{sec:pm-as-functors}

Given a {preordered} set $(P,\leq)$ there is a corresponding category $\cat{P}$ whose objects are the elements of $P$ and whose morphisms consist of the inequalities $x \leq y$, where $x,y\in P$.
An \emph{up-set} in a {preordered} set $(P,\leq)$ is a subset $U \subset P$ such that if $x \in U$ and $x \leq y$ then $y \in U$.
For $a \in P$ denote by $U_{a}\subset P$ the \emph{principal up-set} at ${a}$, i.e., $U_{a}:=\{x\in P\,|\, a\le x\}$.
A \emph{down-set} in a {preordered} set $(P,\leq)$ is a subset $D \subset P$ such that if $y \in D$ and $x \leq y$ then $x \in D$.
For $a \in P$ denote by $D_{a}\subset P$ the \emph{principal down-set} at ${a}$, i.e., $D_{a}:=\{x\in P\,|\, x\le a\}$.
Let $\mathbf{R}^n$ denote the category corresponding to the poset $(\mathbb{R}^n,\le)$, where $\leq$ denotes the product partial order. That is, $(x_1,\ldots,x_n) \leq (y_1,\ldots,y_n)$ if and only if $x_i \leq y_i$ for all $i$.

Let $(P,\leq)$ be a {preordered} set and let $\cat{P}$ be the corresponding category. 
Let $\cat{A}$ be a Grothendieck category -- an Abelian category with additional useful properties (see Appendix~\ref{sec:category}).

\begin{definition}
  A \emph{persistence module} is a functor
  $M: \cat{P} \to \cat{A}$.
  The category of persistence modules is the functor category
  $\cat{A}^{\cat{P}}$, 
  where the objects are persistence modules and morphisms are natural transformations. Of greatest interest to us is a special case of this, when $\mathbf{P}=\mathbf{R}^n$.
\end{definition}

The assumption that $\cat{A}$ is a Grothendieck category contains most examples of interest and ensures that the category of persistence modules has a number of useful properties {(Proposition~\ref{prop:GeneratorsCogenerators})}.
For example, the category $\cat{A}$ may be the category $\ModR$ of right $R$-modules over a unital ring $R$ and $R$-module homomorphisms.
$R$ will always denote a unital ring in what follows and we will always assume that our rings are unital.
We could also consider $\RMod$ the category of left $R$-modules over a unital ring $R$ and $R$-module homomorphisms.
  Of greatest interest to us is a special case of this, the category $\Vectk$, of $\mathbf{k}$-vector spaces for some field $\mathbf{k}$ and $\mathbf{k}$-linear maps.
  
\begin{definition}
\label{def:convect_and_connected}
Let $(P,\le)$ be a {preordered} set. Say $U\subset P$ is \emph{convex} if $a\le c\le b$ with $a,b\in U$ implies that $c\in U$. Say $U \subset P$ is \emph{connected} if for any two $a,b\in U$ there exists a sequence $a=p_0\le q_1\ge p_1\le q_2\ge \cdots \ge p_n \le q_n=b$ for some $n\in \mathbb{N}$ such that all $p_i,q_i\in U$ for $0\le i\le n$.
A connected convex subset of a {preordered} set is called an \emph{interval}. 
\end{definition}

Let $A\subset P$ be a convex subset.
The \emph{indicator persistence module} on $A$ is the persistence module 
$R[A]: \mathbf{P} \to \ModR$ given by
$R[A]_a$ equals $R$ if $a\in A$ and is $0$ otherwise and all the maps $R[A]_{a\le b}$, where $a,b \in A$, are identity maps.
If $A$ is an interval, then $R[A]$ is called {an} \emph{interval persistence module}.
If $A$ is an interval on the real line, say $A=[a,b)$, and $R=\mathbf{k}$ is a field, we will write $\mathbf{k}[a,b)$ instead of $\mathbf{k}[[a,b)]$ for brevity.

\subsection{Persistence modules as sheaves and cosheaves}
\label{sec:sheaves-cosheaves}

For more details see~\cite{MR3259939,Curry:2019}.

\begin{definition} \label{def:alexandrov}
  Let $(P,\le)$ be a {preordered} set. Define the \emph{Alexandrov topology} on $P$ to be the topology whose open sets are the up-sets in $P$.
  Let $\Open(P)$ denote the category whose objects are the open sets in $P$ and whose morphisms are given by inclusions.
\end{definition}
 
\begin{lemma}
Let $(P,\leq)$ and $(Q,\leq)$ be {preordered} sets and consider $P$ and $Q$ together with their corresponding Alexandrov topologies. Let $f: P \to Q$ be a map of sets. Then $f$ is order-preserving  if and only if $f$ is continuous.
\end{lemma}

%


\begin{example}
Consider $\mathbf{R}$ with the Alexandrov topology. Then the open sets are $\emptyset$, $\R$, and the intervals $(a,\infty)$ and $[a,\infty)$, where $a \in \R$.
\end{example}

Let $(P,\leq)$ be a {preordered} set, and let $U$ be an up-set in $P$.
Then the {preorder} on $P$ restricts to a {preorder} on $U$, and $\cat{U}$ is a full subcategory of $\cat{P}$.
Furthermore any functor $F:\cat{P} \to \cat{C}$ restricts to a functor $F|_U: \cat{U} \to \cat{C}$.

\begin{lemma}\cite[Remark 4.2.7]{MR3259939} \label{lem:sheaf}
  Let $(P,\leq)$ be a {preordered} set with the Alexandrov topology and let $\cat{P}$ be the corresponding category. Let $\cat{C}$ be a complete category. Then any functor $F:\cat{P} \to \cat{C}$ has a canonical extension $\hat{F}: \Open(P)^{\op} \to \cat{C}$ given by
\begin{equation*}
 \hat{F}(U) = \lim F|_U = \lim_{p \in U} F(p),
\end{equation*}
and $\hat{F}(U \supset V)$ is given by a canonical map.
\end{lemma}

\begin{proof}
  First we define $\iota: \cat{P} \to \Open(P)^{\op}$ given by $\iota(p) = U_p$ (the principal up-set at $p$) for $p \in P$, and $\iota(p \leq q): U_p \supset U_q$.

  Next for an up-set $U$ in $P$ we have the comma category $U \downarrow \iota$, whose objects are elements $p \in P$ such that $U \supset U_p$, that is, $p \in U$, and whose morphisms are given by $p \leq q \in U$.
  Notice that this category is isomorphic to the category $\cat{U}$.
  Consider the projection $\pi: U \downarrow \iota \to \cat{P}$. Then $\pi$ is just the inclusion of $\cat{U}$ in $\cat{P}$ and $F \circ \pi = F|_U$.

  Now let $\hat{F}:\Open(P)^{\op} \to \cat{C}$ be the right Kan extension, $\Ran_{\iota}F$. By definition, $\hat{F}(U) = \Ran_{\iota}F(U) = \lim (U \downarrow \iota \xto{F\pi} C) = \lim F|_U = \lim_{p \in U} F(p)$. That is, $\hat{F}(U)$ is the universal (i.e. terminal) cone over the diagram $F|_U:\cat{U} \to \cat{C}$.
  For $U \supset V \in \Open(P)^{\op}$, $\hat{F}(U)$ is a cone over $F|_V$.
  By the universal property of $\hat{F}(V)$, there is a canonical map $\res_{V,U}:\hat{F}(U) \to \hat{F}(V)$. Let $\hat{F}(U \supset V) = \res_{V,U}$.
  The universal property of the limit shows that this defines a functor.

  Finally for $p \in P$, $\Ran_{\iota}F \iota(p) = \Ran_{\iota}F(U_p) = \lim_{q \in U_p} F(q) = \lim_{p \leq q} F(q) = F(p)$. So this Kan extension is actually an extension.
\end{proof}

\begin{proposition} \label{prop:sheaf}
  The functor $\hat{F} = \Ran_{\iota}F: \Open(P)^{\op} \to \cat{C}$ is a sheaf.
  That is, for any open cover $\{U_i\}$ of an open set $U$ in $P$,
  \begin{equation} \label{eq:equalizer}
    \begin{tikzcd}[column sep = 7em]
      \hat{F}(U)
\ar[r,"\prod_i \hat{F}(U \supset U_i)"]
&
\displaystyle\prod_i \hat{F}(U_i)
      \ar[r,shift left=.75ex,"\prod_{i,j} \hat{F}(U_i \supset U_i \cap U_j)"]
      \ar[r,shift right=.75ex,swap,"\prod_{i,j} \hat{F}(U_j \supset U_i \cap U_j)"]
&
\displaystyle\prod_{i,j} \hat{F}(U_i \cap U_j)
\end{tikzcd}
\end{equation}
is an equalizer.
\end{proposition}

\begin{proof}
  Let $c$ be the limit of the diagram $\prod_i \lim F|_{U_i} \rightrightarrows \prod_{i,j} \lim F|_{U_i \cap U_j}$ where the arrows are those in \eqref{eq:equalizer}.
  By Lemma~\ref{lem:sheaf}, we want to show that $\lim F|_U \isom c$.

  By the universal property of the limit, for all $i,j$ we have the following commutative diagram of canonical 
  maps.
  \begin{equation*}
    \begin{tikzcd}[row sep = tiny]
      & \lim F|_{U_i} \ar[dr,"\hat{F}(U_i \supset U_i\cap U_j)"] & \\
      \lim F|_U \ar[ur,"\hat{F}(U \supset U_i)"] \ar[dr,"\hat{F}(U \supset U_j)"'] \ar[rr,"\hat{F}(U \supset U_i\cap U_j)"] & & \lim F|_{U_i \cap U_j}\\
      & \lim F|_{U_j} \ar[ur,"\hat{F}(U_j \supset U_i \cap U_j)"']
    \end{tikzcd}
  \end{equation*}
  Therefore there is a canonical map $\lim F|_U \to c$.

  For all $p \in U$, $p \in U_i$ for some $i$. So $U_i \supset U_p$ and hence we have a canonical map $c \to \lim F|_{U_i} \to \lim F|_{U_p} = F(U_p) = F(p)$.
  By the definition of $c$, this map does not depend on the choice of $i$.

  For $p \leq q$, if $p \in U_i$ then $q \in U_i$. So we have the following commutative diagram.
  \begin{equation*}
    \begin{tikzcd}[row sep=tiny]
      & & F(p) \ar[dd]\\
      c \ar[r] & \lim F|_{U_i} \ar[ur] \ar[dr] & \\
      & & F(q)
    \end{tikzcd}
  \end{equation*}
  Thus for all $p,q \in U$ with $p \leq q$, we have canonical maps $c \to F(p)$ and $c \to F(q)$ which commute with $F(p \leq q): F(p) \to F(q)$.
  Therefore there is a canonical map $c \to \lim F|_U$.

  By the universal property of the limit, both composites are the identity map.
\end{proof}

\begin{theorem}\cite[Theorem 4.2.10]{MR3259939} \label{thm:sheaf}
  Let $(P,\leq)$ be a {preordered} set and let $\cat{C}$ be a complete category.
  Then there is an isomorphism of categories between the functor category $\cat{C}^{\cat{P}}$ and the category $\Shv(P;\cat{C})$ of sheaves on $P$ with the Alexandrov topology.
\end{theorem}

\begin{proof}
  Right Kan extension gives a functor $\Ran_{\iota}: \cat{C}^{\cat{P}} \to \cat{C}^{\Open(P)^{\op}}$ (see~\cite[Prop. 6.1.5]{riehl2017category} for example).
  By Proposition~\ref{prop:sheaf}, $\Ran_{\iota}: \cat{C}^{\cat{P}} \to \Shv(P;\cat{C})$.

  Define a functor $\stalk(F): \Shv(P;\cat{C}) \to \cat{C}^{\cat{P}}$ as follows.
  For $p \in P$, let $\stalk(F)(p) = F(U_p)$, and $\stalk(F)(p\leq q) = F(U_p \supset U_q)$.

  We claim these functors are mutually inverse.
  Let $p \in P$ and $F \in \cat{C}^{\cat{P}}$. Then $(\stalk \Ran_{\iota} F)(p) = \Ran_{\iota}F(U_p) = F(p)$.
  Let $U$ be an up-set of $P$ and $F \in \Shv(P;\cat{C})$.
  Then $(\Ran_{\iota} \stalk F)(U) = \lim_{p \in U}(\stalk F)(p) = \lim_{p \in U}F(U_p)$. Since $U = \cup_{p \in U}U_p$ and $F$ is a sheaf, this equals $F(U)$.
\end{proof}

We can dualize the above construction. For a {preordered} set $(P,\leq)$, let $P^{\op}$ denote the {preordered} set with the opposite order. The Alexandrov topology on $P^{\op}$ has as open sets the down-sets $D$ of $P$.
Instead of right Kan extensions and limits, we use left Kan extensions and colimits.

\begin{lemma}\cite[Example 4.5]{Curry:2019} \label{lem:cosheaf}
  Let $(P,\leq)$ be a {preordered} set and let $\cat{P}$ be the corresponding category. Let $\cat{C}$ be a cocomplete category. Then any functor $F:\cat{P} \to \cat{C}$ has a canonical extension $\hat{F}: \Open(P^{\op}) \to \cat{C}$ given by
\begin{equation*}
  \hat{F}(D) = \colim F|_D = \colim_{p \in D} F(p),
\end{equation*}
and $\hat{F}(D \subset E)$ is given by a canonical map.
\end{lemma}

\begin{proposition}\cite[Theorem 4.8]{Curry:2019} \label{prop:cosheaf}
  The functor above $\hat{F}: \Open(P^{\op}) \to \cat{C}$ is a cosheaf.
\end{proposition}

\begin{theorem}\cite[Theorem 4.2.10]{MR3259939} \label{thm:cosheaf}
  Let $(P,\leq)$ be a {preordered} set and let $\cat{C}$ be a cocomplete category.
  Then there is an isomorphism of categories between the functor category $\cat{C}^{\cat{P}}$ and the category $\Coshv(P^{\op};\cat{C})$ of cosheaves on $P^{\op}$ with the Alexandrov topology.
\end{theorem}

\begin{corollary} \label{cor:sheaf-cosheaf}
    Let $(P,\leq)$ be a {preordered} set and let $\cat{A}$ be a Grothendieck category. Then $\cat{A}^{\cat{P}} \isom \Shv(P;\cat{A}) \isom \Coshv(P^{\op};\cat{A})$, where $P$ and $P^{\op}$ have the Alexandrov topology.
\end{corollary}


\begin{example}
  We may consider the persistence module $\mathbf{k}[a,b)$ as a sheaf. For an up-set $U\subseteq \R$,
  $\mathbf{k}[a,b)(U) = \varprojlim_{x \in U} [a,b)_x =
  \begin{cases}
    \mathbf{k} & \text{if } \inf U \in [a,b)\\
    0 & \text{otherwise.}
  \end{cases}
$
\end{example}

Let $X$ be a topological space. If $\mathcal{R}$ is a sheaf of rings on $X$, we can define left (or right) $\mathcal{R}$-modules (which themselves are sheaves of abelian groups). These form a category $\shRMod{\mathcal{R}}$ (or $\shModR{\mathcal{R}}$). If $M$ and $N$ are two such $\mathcal{R}$-modules we denote their set of morphisms by $\text{Hom}_{\mathcal{R}}(M,N)$. If $R$ is a ring, define $R_X$ to be the sheaf associated to the constant presheaf $U\mapsto R$ for every open $U\subset X$. If $R=\mathbf{k}$ is a field and $X=\mathbb{R}^n$ we have the constant sheaf $\mathbf{k}_{\R^n}$ (where $X=\mathbb{R}^n$ has for example the Alexandrov topology). See Appendix~\ref{sec:sheaves} for more details.

\begin{example}
\label{example:persistence_module_sheaf_ring}
We may consider the persistence module $R[P]$ as the constant sheaf of rings $R_P$ on $P$ with the Alexandrov topology (see Appendix~\ref{sec:sheaves}).
By Corollary~\ref{cor:sheaf-cosheaf}, we can view persistence modules $M \in \ModR^{\mathbf{P}}$, or $M\in \RMod^{\mathbf{P}}$ as sheaves on $P$ valued in $\ModR$ or $\RMod$ respectively, where $P$ is given the Alexandrov topology obtained from $(P,\le)$.
Furthermore, we have isomorphisms of categories:
\[\ModR^{\mathbf{P}}\cong \shModR{R_P} \quad \text{and} \quad \RMod^{\mathbf{P}}\cong \shRMod{R_P}\]  (see Appendix~ \ref{sec:sheaves}).
\end{example}

Using the sheaf viewpoint, we have the six Grothendieck operations which we can apply to persistence modules (see \cite[Chapters 2 and 3]{MR1074006}). {In particular we have a tensor product of sheaves $M\otimes_{R_P}N$ and an internal hom of sheaves $\scHom_{R_P}(M,N)$. These six Grothendieck operations are usually only left or right exact functors and in order to preserve cohomological information we need the derived perspective (Appendix~\ref{sec:homological-algebra}). Thus, it is crucial to be able to construct injective and projective resolutions of complexes of sheaves. Proposition~\ref{prop:ks} gives us a way of determining if a given sheaf is injective or not, by checking a smaller class of diagrams rather than the one usually given in the definition of an injective object.

\begin{proposition}{\cite[Exercise 2.10]{MR1074006}} \label{prop:ks}
Let $\mathcal{R}$ be a sheaf of rings on a topological space $X$ and let $M\in \text{Ob}(\shModR{\mathcal{R}})$. Then:
\begin{itemize}
\item[1)] $M$ is injective if and only if for any sub-$\mathcal{R}$-module $\mathcal{S}$ of $\mathcal{R}$ (also called an ideal of $\mathcal{R}$), the natural homomorphism:
\[ \emph{Hom}_{\mathcal{R}}(\mathcal{R},M)\to \emph{Hom}_{\mathcal{R}}(\mathcal{S},M)\]
is surjective.
\item[2)] Let $\mathbf{k}$ be a field. Then any ideal of $\mathbf{k}_{X}$ is isomorphic to a sheaf $\mathbf{k}_U$, where $U$ is open in $X$.
\item[3)] From 1) and 2) it follows that a $\mathbf{k}_{X}$-module $M$ is injective if and only if the sheaf $M$ is flabby (Appendix~\ref{sec:sheaves}). 
\end{itemize}
\label{proposition:BaerCriterionSheaves}
\end{proposition}

Part 1) of Proposition~\ref{prop:ks} is analogous to Theorem~\ref{theorem:BaerCriterionGradedModules}, the Baer criterion for graded modules. It can be used to identify injective persistence modules by looking at a smaller class of diagrams. Part 3) tells us that a vector-space-valued persistence module is injective if and only if it is flabby as a sheaf. In other words, we only need to check if the restriction morphism $M(P)={\lim_{x\in P}}M_x\to M(U)={\lim_{x\in U}}M_x$ is surjective, for all up-sets $U$ in $P$.

\begin{example} \label{ex:ezra}
{Let $a,b\in \mathbb{R}^2$  be incomparable with respect to $\le$ and let $U=U_a\cup U_b$ and $D=D_a\cup D_b$ (Figure~\ref{fig:3})}. Consider the 
interval persistence module on $D$, $\mathbf{k}[D]$.
Observe that $\mathbf{k}[D](\mathbb{R}^2)={\lim_{x\in \mathbb{R}^2}}\mathbf{k}[D]_x=\mathbf{k}$. On the other hand we have that $\mathbf{k}[D](U)={\lim_{x\in U}}\mathbf{k}[D]_x=\mathbf{k}^2$. Hence the restriction morphism induced by the inclusion $U\subset \mathbb{R}^2$ cannot be surjective. 
Thus $\mathbf{k}[D]$ is not flabby as a sheaf, and therefore it is not injective, by Proposition~\ref{proposition:BaerCriterionSheaves}.
\end{example}

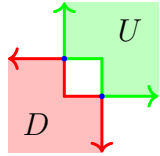
\begin{figure}[ht]
\centering
\begin{tikzpicture}[line cap=round,line join=round,x=1.0cm,y=1.0cm,scale=0.25]
\fill[fill=red!25] (5,0) rectangle (2,-5);
\fill[fill=red!25]  (7,-2) rectangle (5,-5);
\fill[fill=green!25] (5,0) rectangle (7,3);
\fill[fill=green!25] (7,3) rectangle (10,-2);
\draw[color=red,->,line width=0.4mm] (5,0)--(2,0);
\draw[color=red,line width=0.4mm] (5,0)--(5,-2);
\draw[color=red,line width=0.4mm] (5,-2)--(7,-2);
\draw[color=red,->,line width=0.4mm] (7,-2)--(7,-5);
\draw[color=green,->,line width=0.4mm] (5,0)--(5,3);
\draw[color=green,line width=0.4mm] (5,0)--(7,0);
\draw[color=green,line width=0.4mm] (7,0)--(7,-2);
\draw[color=green,->,line width=0.4mm] (7,-2)--(10,-2);
\draw[color=black]  (8.5,1.5) node {$U$};
\draw[color=black]  (3.5,-3.5)  node {$D$};
\fill[color=blue]  (5,0)  circle (1.5mm);
\fill[color=blue]  (7,-2)    circle (1.5mm);
\end{tikzpicture}
\caption {\rmfamily An up-set $U$ and a down-set $D$. {The interval module $\mathbf{k}[D]$ is not injective.} {See Example~\ref{ex:ezra}.}}
\label{fig:3}
\end{figure}

\subsection{Persistence modules as graded modules}
\label{sec:graded-modules}

Throughout this section we assume $R$ is a unital ring 
and $(P,\leq,0,+)$ {is} a {preordered} set together with an abelian group structure.
We assume that the addition operation in the abelian group structure is \emph{compatible}, meaning that  for $a, b, c \in P$, $a \leq b$ implies that $a+c \leq b+c$.
As an example consider $(\R^i \times \Q^j \times \Z^{\ell},\leq,0,+)$, with $i,j,\ell \geq 0$ and $n := i+j+\ell \geq 1$, and where the right hand side has the product partial order. Recall that $U_0$ is the principal up-set at $0 \in P$. For example, if $(P,\leq) = (\R^n,\leq)$, then $U_{0}\subset \mathbb{R}^n$ is the non-negative orthant of $\mathbb{R}^n$. 

\begin{example}
\label{example:GradedRing}
Let $(P,\le)=(\mathbb{R}^n,\le)$. Consider the monoid with addition, $(U_{{0}},+,{0})$, which we will also denote by $U_0$. Let $R$ be a unital ring. Let $R[U_0]$ be the monoid ring, whose definition is analogous to that of a group ring $R[G]$ for a ring $R$ and a group $G$. For example, elements of $R[U_0]$ can be $x_1^{\pi}$, $1+r_1x_1^{e}+r_2x_3^{5}$ , for $r_1,r_2\in R$, etc. This ring, $R[U_0]$, is an $\mathbb{R}^n$-graded ring and is commutative whenever $R$ is. Indeed, we can give it a grading in the following way: $R[U_0]={\bigoplus_{a\in P}}R[U_0]_{a}$, where $R[U_0]_{a}$ is the set of  homogeneous elements in $R[U_0]$ of degree $a$ if $a\ge 0$, and is $0$ otherwise. Observe that $R[U_0]_{a}\cong R$, for all $a\ge 0$.
\end{example}


For a {preordered} set $P$ with a compatible abelian structure, let $\cat{P}$ be the corresponding category.
Let $\mathbf{A}$ be the category of left $R$ modules, $\RMod$.  
Consider a persistence module $M: \cat{P} \to \cat{A}$. Then $M$ can be viewed as an $P$-graded left $R[U_0]$-module and vice versa. Indeed, 
we can write $M={\bigoplus_{a\in P}}M_{a}$ with left $R[U_0]$-action given by $x^s\cdot m:=M_{a\le a+s}(m)$ and extending linearly for a given $m\in M_a$ and $s\in U_0$ and $x^s$ the generator of $R[U_0]_s$. In the other direction, given a left action of $R[U_0]$ we can construct an $R$-module homomorphisms $M_{a}\to M_{a+s}$ by defining them to be given by the left action by the generator $x^s$ of $R[U_0]_s$. Furthermore, every natural transformation corresponds to a graded module homomorphism; see Figure~\ref{fig:2}. This is an isomorphism of categories. This has been observed by different authors, in \cite{lesnick2015interactive} in the $P=\mathbb{R}^n$-graded case, in \cite{MR2121296} in the $P=\mathbb{Z}^n$-graded case and in \cite[Lemma 3.4]{miller2017data} and \cite{miller2019modules} where $P$ is a partially ordered abelian group. What is new in this paper is the generalization to {preordered} sets.
The corresponding statements also hold for functors $M:\mathbf{P}\to \ModR$ and $P$-graded right $R[U_0]$-modules.
 
\begin{figure}[ht]
\centering
\begin{tikzcd}[row sep = scriptsize]
M_a\arrow[r, "M_{a\le b}", "x^{b-a}" below]\arrow[d,"\alpha_a"'] & M_b\arrow{d}{\alpha_b}\\
N_a\arrow[r, "N_{a\le b}", "x^{b-a}" below] & N_b
\end{tikzcd}
\caption{\rmfamily 
Consider maps $\alpha_a :M_a \to N_a$ for $a \in P$.
Viewing $M_{a\le b}$ and $N_{a\le b}$ as actions by $x^{b-a}$, the equality
$\alpha_{b}(M_{a\le b}(m))=N_{a\le b}(\alpha_{a}(m))$
corresponds to the equality $\alpha(x^{b-a}\cdot m)=x^{b-a}\cdot \alpha(m)$.
The first equality is the condition for $\alpha$ to be natural transformation.
The second equality is the condition for $\alpha$ to be a graded module homomorphism.}
\label{fig:2}
\end{figure}

{In Section~\ref{section:kunneth} we will state K\"unneth Theorems for persistence modules. The splitting of the short exact sequences in those theorems will depend on the properties of the graded ring $R[U_0]$.
}
Observe that when the ring $R$ is commutative, the ring $R[U_0]$ is an associative $R$-algebra. Furthermore, the ring $R[U_0]$ is commutative, thus it is a commutative $R$-algebra.
%
Now suppose $R=\mathbf{k}$ is a field and $(P,\le)=(\mathbb{R}^n,\le)$. 
We make the following observations on the ring $\mathbf{k}[U_0]$ and its ideals.
\begin{itemize}
\item[i)] $\mathbf{k}[U_0]$ is not a principal ideal domain. In particular, the ideal $\mathbf{k}[U_0\setminus \{0\}]$ is not generated by a single element. 
\item[ii)] $\mathbf{k}[U_0]$ is not even a unique factorization domain. Otherwise, it would satisfy the ascending chain condition for principal ideals (see \cite[Section 0.2]{MR1322960}). However, for $m = 1,2,3,\ldots$, the increasing sequence of principal graded ideals $\mathbf{k}[U_{(\frac{1}{m},\dots, \frac{1}{m})}]$ does not stabilize. 
\item[iii)] The only graded (homogeneous) ideals are the 
interval
persistence modules of up-sets that are contained in the first orthant, namely $\mathbf{k}[U]$ for up-sets $U\subset U_0$. See also \cite[Remark 8.12]{miller2017data} and \cite{ingebretson2013decompositions}.
\item[iv)] We have that $\mathbf{k}[U_0\setminus \{0\}]$ is the unique nonzero graded maximal ideal of $\mathbf{k}[U_{0}]$, consisting of homogeneous non-invertible elements of $\mathbf{k}[U_0]$. Hence $\mathbf{k}[U_0]$ is a graded-local ring. Note that $\mathbf{k}[U_0]$ is not a local ring. Indeed, if $\mathbf{k}[U_0]$ were local then $x_1$ or $1-x_1$ would be a unit. This is not the case, as these elements are not invertible.
\end{itemize}

Recall that we have assumed that
$P$ be a {preordered} set together with an abelian group structure. Let $M,N:\mathbf{P}\to \mathbf{A}$ be persistence modules, where $\mathbf{A}$ is either $\ModR$ or $\RMod$. Let $\text{Hom}_{R[U_0]}(M,N)$ denote the set of module homomorphisms from a persistence module $M$ to $N$, forgetting the grading. For a module $M$, let $M(s)$ be the translation of $M$ by $s$, i.e., $M(s)_{a}:=M_{s+a}$. 
Recall that a graded module is finitely generated if it is finitely generated as a module (Appendix~\ref{Graded Module Theory}). {The following proposition
  suggests how to construct sets of morphisms between persistence modules that are themselves persistence modules. This will eventually allow us to consider a chain complex of persistence modules with coefficients in another persistence module
  (Section~\ref{section:kunneth}).}

\begin{proposition}{\cite[Theorem 1.2.6]{hazrat2016graded}}
\label{prop:UnderlineHom}
Suppose $M$ is a finitely generated persistence module. Then the abelian group of module homomorphisms from $M$ to $N$, $\emph{Hom}_{R[U_0]}(M,N)$, has a direct sum decomposition $\emph{Hom}_{R[U_0]}(M,N)\cong \bigoplus\limits_{s\in P}\emph{Hom}(M,N(s))$, where $\emph{Hom}(M,N(s))$ is the set of natural transformations (graded module homomorphisms) from $M$ to $N(s)$.
\end{proposition}

Hence sets of (ungraded) module homomorphisms of persistence modules have the structure of a graded abelian group when the domain module M is a finitely generated module. 

 The following is a graded version of Nakayama's Lemma in homological algebra.

\begin{proposition}{\cite[Theorem 4.6]{li2012monoid}}
\label{prop:finitely_generated_projective_graded_module}
Let $\Gamma$ be a monoid. Let $\mathcal{S}$ be a $\Gamma$-graded ring. Suppose $\mathcal{S}$ is a graded-local ring. Then if $P$ is a  finitely generated graded projective $\mathcal{S}$-module, $P$ is a graded free $\mathcal{S}$-module.
\end{proposition}

Since every group is a monoid, we can apply Proposition~\ref{prop:finitely_generated_projective_graded_module} to rings and modules graded over a group. Thus,
we have the following corollary,
where $\Vectk$ is defined in Section~\ref{sec:pm-as-functors}.

\begin{corollary}
\label{corollary:fin_gen_projectives_are_free}
A finitely generated persistence module $M:\mathbf{P}\to \Vectk$ is projective if and only if $M$ is graded free.
\end{corollary}

Let us now summarize some of the results from this section and the previous two sections.

\begin{theorem} \label{thm:isomorphism}
Let $(P,\leq,+,0)$ be a {preordered} set with a compatible abelian group structure. Let $R$ be a unital ring. Then we have the  
  isomorphisms of categories
\begin{equation*} 
\RMod^{\mathbf{P}}\cong \shRMod{R_P}\cong \text{Gr}^P\!\text{-}_{R[U_0]}\mathbf{Mod}
\quad \text{and} \quad
\ModR^{\mathbf{P}}\cong \shModR{R_P}\cong \text{Gr}^P\!\text{-}\mathbf{Mod}_{R[U_0]}.
\end{equation*}
where $Gr^P\text{-}_{R[U_0]}\mathbf{Mod}$ and $Gr^P\text{-}\mathbf{Mod}_{R[U_0]}$ are the categories of $P$-graded left and right $R[U_0]$-modules, respectively.
In particular, for each $a \in P$, $M_a \isom M(U_a)$. Also, $M(U) = \lim_{a \in U} M(U_a)$, and the graded module structure is given by $M \isom \bigoplus_{a \in P} M_a$.
\end{theorem}

\begin{definition}
\label{def:left_and_right_persistence_modules}
We say $M$ is a \emph{left persistence module} if $M:\mathbf{P}\to \RMod$. We say $M$ is a \emph{right persistence module} if $M:\mathbf{P}\to \ModR$. Due to the above isomorphisms, we will also use these terms when $M$ is a left $R_P$-module or right $R_P$-module, respectively, and when $P$ is a preordered set with a compatible abelian group operation, when $M$ is a $P$-graded left $R[U_0]$-module or a $P$-graded right $R[U_0]$-module, respectively.
\end{definition}

\subsection{A Grothendieck category of persistence modules}
\label{sec:grothendieck}
 
In this section we observe that the category of persistence modules is a Grothendieck category and remark that the Gabriel-Popescu Theorem can be applied. {This allows us to potentially consider persistence modules as modules over a new (non-graded) ring.}
Let $(P,\leq)$ be a preordered set.
Recall that for $a\in P$, {$U_a = \{b \in P\ | \ a \leq b \}$.}
 Let $\mathbf{A}$ be a Grothendieck category. {Recall that a \emph{family of generators in a category} is a collection of objects $\{U\}_{i}$ such that for every two distinct morphism $f,g:X\to Y$ in the category, there exists an $i$ and $h:U_i\to X$ such that $fh\neq gh$ (Appendix~\ref{sec:category}). If the family is a singleton, we simply say \emph{generator}.}

\begin{proposition}
\label{prop:GeneratorsCogenerators}
The category $\mathbf{A}^{\mathbf{P}}$ is a Grothendieck category with a generator. In particular, the category has enough projectives and injectives. 
\end{proposition}

\begin{proof}
Since $\mathbf{A}$ is a Grothendieck category, so is the functor category $\mathbf{A}^{\mathbf{P}}$, by  Proposition~\ref{proposition:grothendieck_functor_categories}. Let $G$ be a generator of $\mathbf{A}$. For $a\in P$, define $G[U_a]$ to be the persistence module given by $G[U_a]_b=G$ if $b\in U_a$ and $0$ otherwise and let $G[U_a]_{b\le c}=\mathbf{1}_G$ if $b,c \in U_a$ and $0$ otherwise. The collection $\{G[U_a]\}_{a\in P}$ is a family of generators. Indeed, suppose $f,g:M\to N$ are natural transformations between persistence modules $M$ and $N$ such that $f\neq g$. Then by definition, there exists an $a\in P$ such that $f_{a}\neq g_{a}$. In particular, as $G$ is a generator of $\mathbf{A}$, there exists an $h_a:G_a\to M_a$ such that $f_ah_a\neq g_ah_a$. Define $h:G[U_a] \to M$ by setting $h_b=0$ for $b\not\in U_a$  and setting $h_b=M_{a\le b}h_a$ for $b\in U_a$. Since all of the maps in $G[U_{a}]$ are the identity or are zero, the collection of maps $\{h_b\}_{b\in P}$ are the components of a natural transformation $h$. Then, by construction it is clear that $fh\neq gh$, hence $\{G[U_{a}]\}_{a\in P}$ is a family of generators. 
By Proposition~\ref{prop:Generator} we have that  $U:={\bigoplus_{a\in P}}G[U_a]$ is a generator (which is also free and hence projective). 
By Theorem~\ref{theorem:injectives} and Proposition~\ref{prop:Generator}, the category has enough injectives and projectives.
\end{proof}

We will show later in Section~\ref{subsection:classification_of_interval_modules} that the 
interval
modules $\mathbf{k}[D_{a}]$ are injective and that the 
interval
modules $\mathbf{k}[U_{a}]$ are projective, when $\mathbf{k}$ is a field.

\begin{theorem}{\cite[Theorem 14.2, Chapter 4]{popescu1973abelian}}[Gabriel-Popescu Theorem]
\label{theorem:Gabriel_Popescu}
Let $\cat{C}$ be a Grothendieck category and let $U$ be an object in $\mathcal{C}$. Consider the endomorphism ring ${S:=\emph{\text{End}}_{\cat{C}}(U)}$. Then the following are equivalent:
\begin{itemize}
\item[1)] U is a generator.
\item[2)] The functor $\emph{\text{Hom}}(U,\cdot):\cat{C}\to \mathbf{Mod}_S$ is full and faithful and its left adjoint $\cdot \otimes_S U:\mathbf{Mod}_S\to \cat{C}$ is exact.
\end{itemize}
\end{theorem}

From this we have the following Gabriel-Popescu theorem for persistence modules:

\begin{corollary} \label{cor:gp}
Let $U=\bigoplus_{a\in P} G[U_{a}]$ and let {$S=\emph{\text{End}}(U)$}. Then:
\begin{itemize}
\item $\emph{\text{Hom}}(U,\cdot):\mathbf{A}^{\mathbf{P}}\to \mathbf{Mod}_S$ is full and faithful; and its left adjoint
\item $\cdot \otimes_SU:\mathbf{Mod}_S\to \mathbf{A}^{\mathbf{P}}$ is exact.
\end{itemize}
\end{corollary}

We will use this result in Section~\ref{section:pers_modules_over_finite_posets} when we consider persistence modules over finite {preordered sets.}

\subsection{Chain complexes of persistence modules}
\label{sec:chain-complexes}

In Section~\ref{section:kunneth} we will investigate how changing the coefficients of a chain complex of persistence modules changes its homology.
In order to compute {examples} that come from applications, we consider a chain complex of persistence modules obtained from a filtered cellular complex, such as a filtered simplicial complex or a filtered cubical complex. 

A \emph{filtration} on a CW complex $X$ is a function
$f:X\to \mathbb{R}$ that is constant on the cells of $X$ and such that
$f(\partial \sigma) \leq f(\sigma)$ for all cells $\sigma$ of $X$.
For $a\in \mathbb{R}$, let $X_a$ be the subcomplex of $X$ defined by $X_a:=f^{-1}(-\infty,a]$. The collection of CW complexes $\{X_a\}_{a \in \mathbb{R}}$
with the inclusion maps $X_a\hookrightarrow X_b$ whenever $a\le b$
is a \emph{filtered CW complex}.
The inclusion maps induce
$\mathbf{k}$-linear maps on cellular homology with coefficients in a field $\mathbf{k}$, $H_n(X_a;\mathbf{k})\to H_n(X_b;\mathbf{k})$}.
Let $\mathcal{H}_n(X)$ denote the resulting persistence module.

Let $X$ be a CW complex with filtration $f$.
Let $X_{(m)}$ denote the set of $m$-cells of $X$.
For $m \geq 0$, define
$\mathcal{C}_m(X) = {\bigoplus}_{\sigma \in X_{(m)}} \mathbf{k}[f(\sigma),\infty)$.
For $\sigma \in X_{(m)}$, also let $\sigma$ denote the generator of $\mathbf{k}[f(\sigma),\infty)$.
For $a \geq f(\sigma)$, let $\sigma_a$ denote $\mathbf{k}[f(\sigma),\infty)_{f(\sigma) \leq a} \sigma$.
Similarly define $\alpha_a \in \mathcal{C}_m(X)$ for an $m$-chain $\alpha$ in the cellular chain complex on $X$.
Define $d_m: \mathcal{C}_m(X) \to \mathcal{C}_{m-1}(X)$ to be the natural transformation obtained by extending the definition $(d_m)_a(\sigma_a) := (\partial \sigma)_a$ linearly.
Let $H_n(\mathcal{C}(X))$ be the homology of the chain complex $(C_*(X),d_*)$.

\begin{lemma}
\label{lemma:homology_of_filtered_simplicial_complex}
Let $X$ be a CW complex with a filtration as above. For all $n\in \mathbb{N}$,
$H_n(\mathcal{C}(X)) \isom \mathcal{H}_n(X)$.
\end{lemma}
\begin{proof}
{For all $a\in \mathbb{R}$, by definition, $\mathcal{H}_n(X)_a$ is the $n$-th cellular homology of $f^{-1}(-\infty,a]\subseteq X$, $H_n(f^{-1}(-\infty,a];\mathbf{k})$. By construction, $\mathcal{C}_{m}(X)_a$ has as generators the $m$ cells $\sigma$ of $X$ such that $f(\sigma)\le a$. By the definition of $(d_m)_a$ it follows that $H_n(\mathcal{C}(X))_a$ is isomorphic to $H_n(f^{-1}(-\infty,a];\mathbf{k})$.}
\end{proof}

\section{Tensor products of persistence modules}
\label{section:tensors}

In this section, we consider two functors of persistence modules. In Section~\ref{section:enriched} we show that they are both monoidal products on the category of persistence modules. These are $\grtensor$ and $\shtensor$, the tensor products from graded module theory and sheaf theory respectively. {We give formulas for calculating these functors applied to one-parameter interval modules. These formulas will be useful in computations in Section~\ref{section:kunneth}.}

\subsection{Tensor product of sheaves} \label{sec:shtensor}

Let $(P,\le)$ be a preordered set and let $R$ be a unital ring.
For more details, see \cite[Chapter 1]{Bredon:SheafTheory} and \cite[Chapter 2]{MR1074006}.

\begin{definition}
\label{def:sheaf_tensor_product}
Let $M$ be a right $R_{P}$-module and let $N$ be a left $R_{P}$-module, where $P$ is given the up-set topology. The sheaf tensor product $M\otimes_{R_{P}}N$ is the sheaf of abelian groups on $P$ which is associated to the presheaf given by the assignment $U\mapsto M(U)\otimes_R N(U)$, for an up-set $U\subset P$. The stalk of this presheaf at $a\in P$ is $M_a\otimes_R N_a$. As sheafification preserves the values on stalks, we have $(M\otimes_{R_{P}}N)_a= M_a\otimes_R N_a$. However, as discussed in the proof of Lemma~\ref{lem:sheaf}, we have $(M\otimes_{R_{P}}N)(U_a)=M(U_a)\otimes_R N(U_a)=M_a\otimes_R N_a$. By the result of Theorem~\ref{thm:sheaf}, we can also take the functor $\stalk (M\otimes_{R_{P}}N):\mathbf{P}\to \mathbf{Ab}$, defined by $\stalk(M\otimes_{R_{P}}N)_a:=M(U_a)\otimes_R N(U_a)$, as the definition of $M\otimes_{R_{P}}N$. To simplify notation, we will denote $\otimes_{R_{P}}$ by $\shtensor$ throughout this paper (the ring $R$ will always be clear from context). When $N$ is an $R_{P}$-bimodule, $M\shtensor N$ is in fact a right $R_{P}$-module. When $R$ is commutative, $M\shtensor N$ is an $R_{P}$-module.
\end{definition}

\begin{example}
\label{example:sheaf_tensor_of_interval_modules}
Assume that $(P,\le)=(\mathbb{R}^n,\le)$ and $R=\mathbf{k}$ is a field.
 Let $U,V\subset \mathbb{R}^n$ be 
intervals
and let $\mathbf{k}[U]$ and $\mathbf{k}[V]$ be the corresponding 
interval persistence modules. 
For $a \in \R^n$, $(\mathbf{k}[U] \shtensor \mathbf{k}[V])_a = \mathbf{k}[U]_a \tensor \mathbf{k}[V]_a$ which equals $\mathbf{k}$ if $a \in U \cap V$ and is otherwise zero. 
If $U\cap V$ is connected then, $\mathbf{k}[U]\shtensor \mathbf{k}[V]=\mathbf{k}[U\cap V]$.  
As a special case, if $n=1$, we have that $\mathbf{k}[a,\infty) \shtensor \mathbf{k}[b,\infty) = \mathbf{k}[{\max\{a,b\}},\infty).$
\end{example}

\subsection{Tensor product of graded modules}

Let $(P,+,0)$ be an abelian group. 
There exists a tensor product operation on $Gr^{P}$-$\mathcal{S}$, the category of $P$-graded modules over a $P$-graded ring $\mathcal{S}$; for example see \cite{hazrat2016graded}. Hence we have a tensor product of persistence modules, $M\otimes_{R[U_{0}]}N$.
For the one-parameter case, see for example \cite{polterovich2017persistence,carlsonfilippenko2019}. 
For simplicity 
and to differentiate from the sheaf tensor product
we will write $M\grtensor N$ throughout, as the ring $R$ and the abelian group $P$ will be clear from the context.

\begin{definition}
\label{def:alg_def_gr_tensor_prod}
Let $M$ be a $P$-graded right $R[U_0]$-module and let $N$ be $P$-graded left $R[U_0]$-module. 
Let $M\otimes _R N:={\bigoplus_{r\in P}}(M\otimes_R)_r$, where 
\begin{equation*}
(M\otimes_R N)_r:= R \Bigl\langle \Big\{\sum_i m_i\otimes_R n_i\,|\,m_i\in M^h,n_i\in N^h, \text{deg}(m_i)+\text{deg}(n_i)=r\Big\} \Bigr\rangle.
\end{equation*}
 Define the \emph{graded module tensor product} of $M$ and $N$, written $M \grtensor N$, to be the $P$-graded abelian group given by 
\begin{equation*}
  M \grtensor N:=  (M \tensor_R N) / J,
\end{equation*}
where $J$ is the subgroup of $M\otimes_R N$  generated by the homogeneous elements
\begin{equation*}
\{ m\cdot x\otimes_R n-m\otimes_R x\cdot n\,|\,m\in M^h,n\in N^h,x\in R[U_0]^h\}.
\end{equation*}
where $M^h,N^h$ and $R[U_0]^h$ are the sets of homogeneous elements of $M,N$ and $R[U_0]$ respectively.
\end{definition}

Now assume $(P,\le,+,0)$ is a {preorder} with a group structure compatible with the {preorder}, namely $a\le b$ implies $a+c\le b+c$. Then, there is an equivalent categorical definition of $\grtensor$, as observed in \cite{polterovich2017persistence}.
Let $X_r={\bigoplus_{s+t=r}}(M_s\otimes_R N_t)$. The abelian group $(M\grtensor N)_{r}$ is the quotient of $X_r$ given by the colimit of the diagram of abelian groups $(M_{s}\otimes_RN_{t})_{s+t\le r}$.
See Figure~\ref{fig:8} for the case $(P,\le)=(\mathbb{R},\le)$.
\begin{figure}[ht]
\centering
\begin{tikzpicture}[line cap=round,line join=round,x=1.0cm,y=1.0cm,scale=0.5]
\draw[->,color=black] (-3,0) -- (9,0);
\draw[->,color=black] (0,-0.5) -- (0,10.5);
\draw[color=black](10,0) node {M};
\draw[color=black](0,11) node {N};
\draw[color=black] (1.5,0.4) node[scale=0.7] {$M_a$};
\draw[color=black] (7.5,0.4) node[scale=0.7] {$M_b$};
\draw[color=black] (-0.5,3) node[scale=0.7]{$N_c$};
\draw[color=black] (-0.5,9) node[scale=0.7]{$N_d$};
\draw[color=black] (8.1,3.4) node[scale=0.7] {$M_{b}\otimes_R N_c$};
\draw[color=black] (2.2,9.4) node[scale=0.7] {$M_{a}\otimes_R N_{d}$};
\draw[color=black] (2.2,3.4) node[scale=0.7] {$M_a\otimes_R N_c$};
\node[text width=7cm,color=black, anchor=west, right,scale=1] at (8,9)
   { $X_r:=\bigoplus\limits_{s+t=r}(M_s\otimes_R N_t)$ };
\draw[color=black] (9.5,1) node[scale=1]{$X_r $};
\node[text width=7cm,color=black,anchor=west,right,scale=1] at (8,6)
 {$(M\grtensor N)_r:={\colim\limits_{s+t\le r}}(M_s\otimes_R N_t)$};
\draw[color=black,thick] (-1,11) -- (9,1);
\fill [color=blue] (0,9) circle (1.5mm);
\fill [color=blue] (0,3) circle (1.5mm);
\fill [color=blue] (7,3) circle (1.5mm);
\fill [color=blue] (7,0) circle (1.5mm);
\fill [color=blue] (1,9) circle (1.5mm);
\fill [color=blue] (1,0) circle (1.5mm);
\fill [color=blue] (1,3) circle (1.5mm);
\draw[->,color=red,shorten >=5pt,dashed,thick] (1,3) -- (7,3);
\draw[->,color=red,,shorten >=5pt,dashed,thick] (1,3) -- (1,9);
\draw[color=black] (4.5,2.5) node[scale=0.7] {$M_{a\le b}\otimes_R1_{N_c}$};
\node[label={[label distance=0.5cm,text depth=-1ex,rotate=90,scale=0.7]right: $1_{M_a}\otimes_R N_{c\le d}$}] at (1.1,3.7) {};
\end{tikzpicture}
\caption {\rmfamily The tensor product of one-parameter persistence modules $M$ and $N$. Each abelian group $(M\grtensor N)_r$  is assigned to be the colimit of the diagram of abelian groups $(M_s\otimes_R N_t)_{s+t\le r}$.}
\label{fig:8}
\end{figure}
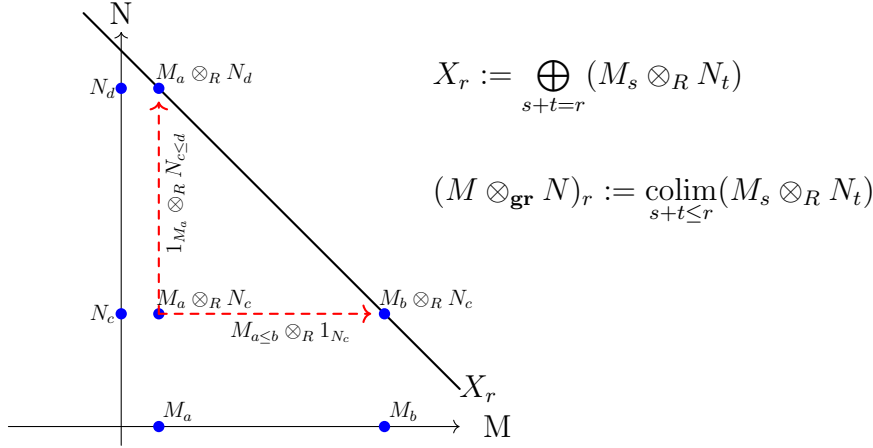 

\begin{definition}
\label{def:graded_tensor_product}
Let $M$ be a $P$-graded right $R[U_0]$-module and let $N$ be a $P$-graded left $R[U_0]$-module. Define the $P$ graded abelian group  $M\grtensor N$ by setting $(M\grtensor N)_r:=\colim_{s+t\le r}(M_{s}\otimes_RN_{t})$.
\end{definition}

Observe that Definition~\ref{def:alg_def_gr_tensor_prod} and  Definition~\ref{def:graded_tensor_product} are equivalent. Indeed this follows 
from {Section~\ref{sec:graded-modules}} and the way the  {$\mathbb{Z}[U_0]$ action is defined in the quotient} in Definition~\ref{def:alg_def_gr_tensor_prod}.
If $N$ is a $P$-graded $R[U_0]$-bimodule, then $M\grtensor N$ is a $P$-graded right $R[U_0]$-module. Furthermore 
$M\grtensor N(s)=M(s)\grtensor N=(M\grtensor N)(s)$ for all $s\in \mathbb{R}^n$, and
$M\grtensor \mathbf{k}[U_{s}]=M(-s)$.

\begin{example}
\label{example:TensorOfIntervalModules} 
Let $M=\mathbf{k}[a,b)$ and $N=\mathbf{k}[c,d)$.
Assume $b+c\le a+d$ (see Figure~\ref{fig:9}). 
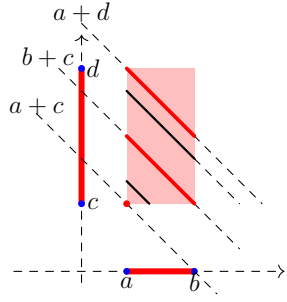
\begin{figure}[ht]
\centering
\begin{tikzpicture}[line cap=round,line join=round,,x=1.0cm,y=1.0cm,scale=0.3]
\draw[->,color=black,dashed] (-3,0) -- (9,0);
\draw[->,color=black,dashed] (0,-0.5) -- (0,10.5);
\draw[color=red,line width=0.8mm] (2,0) -- (5,0);
\draw[color=red,line width=0.8mm] (0,3) -- (0,9);
\fill [color=blue] (2,0) circle (1.5mm);
\fill [color=blue] (5,0) circle (1.5mm);
\fill [color=blue] (0,3) circle (1.5mm);
\fill [color=blue] (0,9) circle (1.5mm);

\fill[fill=red!25] (5,3) rectangle (2,9);
\draw[color=black,dashed] (-2,7) -- (6,-1);
\draw[color=black,dashed] (-1,9) -- (7,1);
\draw[color=black,dashed] (0,11)--(8,3);
\draw[color=red, line width=0.5mm] (2,6)--(5,3);
\draw[color=red, line width=0.5mm] (2,9) -- (5,6);
\fill [color=red] (2,3) circle (1.5mm);
\draw[color=black,line width=0.3mm] (2,4)--(3,3);
\draw[color=black,line width=0.3mm](2,8)--(5,5);
\draw[color=black,dashed] (5,5)--(7,3);
\draw[color=black] (2,-0.5) node[scale=0.8] {$a$};
\draw[color=black] (5,-0.5) node[scale=0.8] {$b$};
\draw[color=black] (0.5,3) node[scale=0.8] {$c$};
\draw[color=black] (0.5,9) node[scale=0.8] {$d$};
\draw[color=black] (-2,7.3) node[scale=0.8]{$a+c$};
\draw[color=black] (-1.5,9.5) node[scale=0.8] {$b+c$};
\draw[color=black] (0,11.5) node[scale=0.8] {$a+d$};
\end{tikzpicture}
\caption {\rmfamily Tensor product of interval modules : $\mathbf{k}[a,b)\grtensor \mathbf{k}[c,d) =\mathbf{k}[a+c,$min$\{a+d,b+c\})$ }
\label{fig:9}
\end{figure} 
Let $r\in \mathbb{R}$ and let $X_r:={\bigoplus_{s+t=r}}(M_s\otimes_{\mathbf{k}}N_t)$. 
For $a+c\le r<b+c$, every summand of $X_r$ is in the image of $M_a\otimes_{\mathbf{k}}N_c\cong \mathbf{k}$,
and hence $(M\grtensor N)_r\cong \mathbf{k}$ and for $a+c\le r\le r'<b+c$,  $(M\grtensor N)_{r\le r'}$ is the identity map on $\mathbf{k}$. 
For $b+c\le r$, each {non-zero} summand $M_s\otimes_{\mathbf{k}}N_t$ of $X_r$ {has $t>c$ and thus} lies in the image of $M_s\otimes_{\mathbf{k}}N_c\cong \mathbf{k}$. However, the map $M_s\otimes_{\mathbf{k}}N_c \to M_l\otimes_{\mathbf{k}}N_c$ where $l$ is such that $l+c=r$ has to be the zero map as $r\ge b+c$, thus $l\ge b$ and thus $M_l=0$. Hence $M\grtensor N\cong \mathbf{k}[a+c,b+c)$. 

If we had $a+d\le b+c$, then the same argument shows that $M\grtensor N\cong \mathbf{k}[a+c,a+d)$. 
Combining these two results we have the following.
\[
  \mathbf{k}[a,b)\grtensor \mathbf{k}[c,d)=\mathbf{k}[a+c,\min\{a+d,b+c\})
\]
Note that the persistence of this interval module (i.e. the length of the corresponding interval) is the minimum of the persistences of the interval modules $M$ and $N$.

Alternatively, note that $\mathbf{k}[a,b)$ and $\mathbf{k}[c,d)$ are graded modules with one generator in degrees $a$ and $c$, respectively. Label these generators as $y^a$ and $z^c$ respectively. Then note that by the action of the graded ring $\mathbf{k}[0,\infty)$, we have $x^t\cdot y^a\neq 0$ if and only if $t< b-a$. Similarly $x^t\cdot z^c\neq 0$ if and only if $t\le d-c$.
From the point of view of graded module theory, $\mathbf{k}[a,b)\grtensor \mathbf{k}[c,d)$ will be a graded module with a single generator in degree $a+c$, namely $y^a\grtensor z^c$ and $x^t\cdot (y^a\grtensor z^c)\neq 0$ if and only if $t<\min \{b-a,d-c\}$. 

Similarly one obtains the following equalities.
\begin{gather*}
  \mathbf{k}[a,\infty) \grtensor \mathbf{k}[c,d) = \mathbf{k}[a+c,a+d) \quad \quad
  \mathbf{k}[a,\infty) \grtensor \mathbf{k}[c,\infty) = \mathbf{k}[a+c,\infty)\\  
  \mathbf{k}[a,b) \grtensor \mathbf{k}(-\infty,d) = 0 \quad \quad
  \mathbf{k}[a,\infty) \grtensor \mathbf{k}(-\infty,d) = \mathbf{k}(-\infty,a+d)\\  
  \mathbf{k}[a,b) \grtensor \mathbf{k}(-\infty,\infty) = 0 \quad \quad
  \mathbf{k}[a,\infty) \grtensor \mathbf{k}(-\infty,\infty) = \mathbf{k}(-\infty,\infty)
\end{gather*}
\end{example}

Note that $\grtensor$ is different from $\shtensor$. Indeed, the tensor unit of $\grtensor$ is $R[U_{0}]$ while the tensor unit of $\shtensor$ is $R[P]$. 
We will focus more on $\grtensor$ over $\shtensor$ in this paper because $\grtensor$ is right exact in general unlike $\shtensor$ which is exact when $R=\mathbf{k}$ is a field. We thus need to spend more time carefully constructing projective resolutions and calculating the derived functor of $\grtensor$. 
However, as we will see in Proposition~\ref{prop:inverse_image_and_sh_tensor} and {Remark}~\ref{prop:inverse_image_and_gr_tensor}, unlike $\shtensor$, $\grtensor$ does not interact nicely with the other Grothendieck operations obtained from sheaf theory. 
\begin{definition}
\label{def:inverse_image_sheaf}
Let $X$ and $Y$ be topological spaces and $f:Y\to X$ a continuous map. Let $F$ be a sheaf on $X$. The inverse image of $F$ by $f$, denoted $f^{-1}F$ is the sheaf on $Y$ associated to the presheaf given by the following assignment:
\begin{equation*}
f^{-1}F(U):={\colim_{f(U)\subset V}}F(V),
\end{equation*} 
for all open $U\subset Y$, where $V$ ranges over all open subsets of $X$ containing $f(U)$.
\end{definition}

\begin{proposition}
\label{prop:inverse_image_and_sh_tensor}
Let $f:P\to P$ be a continuous map 
(with respect to the Alexandrov topology on $(P,\leq)$). 
Then for a right persistence module $M$ and a left persistence module $N$ we have a canonical isomorphism.
\begin{equation} \label{eq:inverse-image-iso}
f^{-1}(M\shtensor N)\cong f^{-1}M\shtensor f^{-1}N
\end{equation}
\end{proposition}

\begin{proof}
  Let $f:X\to Y$ be a map of topological spaces, and let $\mathcal{R}$ be a sheaf of rings on $Y$. Let $M$ be a right $\mathcal{R}$ module and let $N$ be a left $\mathcal{R}$ module. Then there is a canonical isomorphism $f^{-1}(M\otimes_{\mathcal{R}}N)\cong f^{-1}M\otimes_{f^{-1}\mathcal{R}}f^{-1}N$, see for example \cite[Proposition 2.3.5]{MR1074006}. Now let $X=Y=P$ (with the Alexandrov topology) and let $\mathcal{R}=R_P$. Suppose $f:P\to P$ is continuous, with respect to the Alexandrov topology on $P$. Then, $f^{-1}R_P=R_P$. Indeed let $U\subset P$ be an up-set. Then by Definition~\ref{def:inverse_image_sheaf}, $f^{-1}R_P$ is the sheaf associated to the presheaf $f^{-1}R_P(U):= \colim_{f(U)\subset V}R_P(V)=R$, which means $f^{-1}R_P$ is the constant sheaf on $P$.
Thus we have $f^{-1}(M\shtensor N)\cong f^{-1}M\shtensor f^{-1}N$.
\end{proof}

\begin{remark}
\label{prop:inverse_image_and_gr_tensor}
Let $f:P\to P$ be continuous (with respect to the Alexandrov topology on $(P,\leq)$). It is not necessarily true that  $f^{-1}(M\grtensor N)$ is isomorphic to $f^{-1}M\grtensor f^{-1}N$. Indeed consider the following counter example. Let $(P,\le)=(\mathbb{R},\le)$ and let $R=\mathbf{k}$. Let $f:\mathbb{R}\to \mathbb{R}$ be given by $f(x)=x+5$. Observe that $f$ is continuous and that for an interval module $\mathbf{k}[a,b)$ we have:
\[f^{-1}(\mathbf{k}[a,b)\grtensor \mathbf{k}[a,b))=f^{-1}\mathbf{k}[2a,a+b)=\mathbf{k}[2a-5,a+b-5)\, .\]
On the other hand:
\[f^{-1}\mathbf{k}[a,b)\grtensor f^{-1}\mathbf{k}[a,b)=\mathbf{k}[a-5,b-5)\grtensor \mathbf{k}[a-5,b-5)=\mathbf{k}[2a-10,a+b-10)\, .
\]
\end{remark}

For the remainder of this section we assume that $(P,\le)=(\mathbb{R}^n,\le)$ and that $R=\mathbf{k}$ is a field.

\begin{example}
\label{example:multiparameter_tensor_of_rect_modules}
Consider persistence modules $M=\mathbf{k}[[a_1,b_1)\times \dots \times [a_n,b_n)]$ and $N=\mathbf{k}[[c_1,d_1)\times \dots \times [c_n,d_n)]$. 
Then $M\grtensor N=\mathbf{k}[[a_1+c_1,\min\{b_1+c_1,a_1+d_1\})\times\dots\times [a_n+c_n,\min\{b_n+c_n,a_n+d_n)]$. To see this, observe that $M$ and $N$ are graded modules with a single generator, in degrees $(a_1,\dots ,a_n)$ and $(c_1,\dots ,c_n)$ respectively. Hence $M\grtensor N$ will be a persistence module with a single generator in degree $(a_1+c_1,\dots ,a_n+c_n)$, say $y^{a+c}$, and all that is left is to determine for which $t\in \mathbb{R}^n$ is $x^t\cdot y^{a+c}$ zero. We examine this coordinatewise as in Example~\ref{example:TensorOfIntervalModules} to obtain the answer above. 
\end{example}

\section{Homomorphisms of persistence modules}
\label{section:homs}

In this section we consider two bifunctors of persistence modules: the two internal homs, $\uHom$ and $\scHom$, coming from graded module theory and sheaf theory, respectively. {These functors are well known in their respective domains but examples in the persistence module literature seem to be lacking.} {In order to do computations with interval modules we first need to understand the sets of natural transformations between them. The following examples serve that purpose.}

\begin{example}[{\cite[Appendix A.2]{bubenik2018topological}}]
\label{example:natural_transformations_form_a_vector_space}
Suppose $\mathbf{k}[a,b)$ and $\mathbf{k}[c,d)$ are interval modules. Then, due to the constraints of commutative squares for natural transformations, we have:
\[\text{Hom}(\mathbf{k}[a,b),\mathbf{k}[c,d)) \cong
  \begin{cases}
                                   \mathbf{k} &\text{if } c \le a < d \le b\\
                                    0 &\text{otherwise}
  \end{cases}
\]
\end{example}

\begin{example}{\cite[Proposition 3.10]{miller2019modules}}
\label{example:hom_examples}
  Let $U$ be an up-set and $D$ a down-set in {a} poset $({P,\le})$. Then $\text{Hom}(\mathbf{k}[U],\mathbf{k}[D])\cong\mathbf{k}^{\pi_0(U\cap D)}$ where $\pi_0 A$ is the set of equivalence classes of connected components of a set $A$, with respect to  the poset structure, as in Definition~\ref{def:convect_and_connected}. For upsets $U$ and $U'$, $\Hom(\mathbf{k}[U'],\mathbf{k}[U])=\mathbf{k}^{\{S\in \pi_0U'\,|\,S\subseteq U\}}$.
\end{example}

\subsection{Sheaf internal hom}

Let $(P,\le)$ be a {preorder} with the Alexandrov topology and let $R$ be a unital ring. Given two left/right persistence modules $M$ and $N$, thought of as sheaves, there is a sheaf of abelian groups given by
$\scHom_{R_P}(M,N)(U):=\text{Hom}_{{R_P}|_{U}}(M|_U,N|_U)$, for any up-set $U$ (see Definition~\ref{def:sheaf_hom}). We will write $\scHom(M,N)$ instead of $\scHom_{R_P}(M,N)$ as the ring $R$ and {preorder} $P$ will always be clear from context. Furthermore, when the ring $R$ is commutative, $\scHom(M,N)$ also has the structure of an $R_P$-module (i.e. a persistence module).
For any persistence module $X$, the functor $- \shtensor X$ is left adjoint to the functor $\scHom(X,-)$ (see Proposition~\ref{prop:sheaf-tensor-hom-adjunction}).

\begin{example}
\label{example:sheaf_hom_of_interval_modules}
For interval modules $\mathbf{k}[a,b)$ and $\mathbf{k}[c,d)$ we have the following. 
\[\scHom(\mathbf{k}[a,b),\mathbf{k}[c,d))=\begin{cases}
0 & \text{if } a < b\le c < d\\
0 & \text{if } a < c\le b < d\\
\mathbf{k}[c,d) & \text{if } a < c < d\le b\\
0 & \text{if } c\le a < b < d\\
\mathbf{k}(-\infty,d) & \text{if } c\le a < d\le b\\
0 & \text{if } c < d\le a < b
\end{cases}\]
To see this, note that by definition we have the following.
\begin{gather*}
\scHom(\mathbf{k}[a,b),\mathbf{k}[c,d))_x=\scHom(\mathbf{k}[a,b),\mathbf{k}[c,d))([x,\infty))=\\
=\text{Hom}_{\mathbf{k}[\mathbb{R}]|_{[x,\infty)}}(\mathbf{k}[a,b)|_{[x,\infty)},\mathbf{k}[c,d)|_{[x,\infty)})
\end{gather*}
Thus, we need to compute the set of natural transformations between the functors \[\mathbf{k}[a,b)|_{[x,\infty)},\mathbf{k}[c,d)|_{[x,\infty)}:[x,\infty)\to \mathbf{Vect}_{\mathbf{k}},
\]
where $[x,\infty)$ is given the total linear order induced from $\mathbb{R}$. 
Note that $\mathbf{k}[a,b)|_{[x,\infty)}$ is nonzero if and only if $x \in (-\infty,b)$.
Consider the case $c\le a < d\le b$. 
As in Example~\ref{example:natural_transformations_form_a_vector_space}, we see that $\text{Hom}_{\mathbf{k}[\mathbb{R}]|_{[x,\infty)}}(\mathbf{k}[a,b)_{[x,\infty)},\mathbf{k}[c,d)_ {[x,\infty)})\cong \mathbf{k}$ if $x\in (-\infty,d)$ and is zero otherwise. The other cases may be computed similarly.
The same argument also shows that
\begin{equation*}
  \scHom(\mathbf{k}[a,\infty),\mathbf{k}[c,d)) =
  \begin{cases}
    \mathbf{k}[c,d) &\text{if } a < c\\
    \mathbf{k}(-\infty,d) &\text{if } c \leq a < d \\
    0 &\text{if } d \leq a
  \end{cases}
\end{equation*}
and that
\[
  \scHom(\mathbf{k}[a,b), \mathbf{k}[\R]) = 0 \quad \text{and} \quad
  \scHom(\mathbf{k}[a,\infty), \mathbf{k}[\R]) = \mathbf{k}[\R].
\]
\end{example}

\subsection{Graded module internal hom}
Now assume that $(P,\le,+,0)$ is a {preordered} set with a compatible abelian group structure. Then we can consider the graded module internal hom, the right adjoint of $\grtensor$.

\begin{definition}
\label{def:translation_functor}
Let $M$ be a persistence module (either left or right). For $s\in P$, let $\mathcal{T}_{s}:\mathbf{P}\to \mathbf{P}$ be the translation functor by $s$, i.e., $\mathcal{T}_{s}(x)=x+s$. Define $M(s):=M\circ \mathcal{T}_{s}$.
\end{definition}

Observe that for every $s\in U_{0}$, there is a natural transformation $\eta_{s}:1_{\mathbf{P}}\to \mathcal{T}_{s}$ whose components $(\eta_{s})_{a}:1_{\mathbf{P}}(a)\to \mathcal{T}_{s}(a)$ are given by $a\le a+s$. Then $\eta_{s}$ is a natural transformation since $a\le b$ implies $a+s\le b+s$ for all $a,b\in P$. Furthermore, for any $s\in U_0$, given a persistence module $M$,  we have a natural transformation $1_M* \eta_{s}:M\to M(s)$, where $*$ denotes horizontal composition.

\begin{definition}
Let $M$ and $N$ be two persistence modules (both left or both right). Define $\uHom(M,N):=\bigoplus_{s\in P}\Hom(M,N(s))$. Then $\uHom(M,N)$ is a $P$-graded abelian group. This follows from Proposition~\ref{prop:UnderlineHom}. When the ring $R$ is commutative, $\uHom(M,N)$ is a persistence module. Given $s\in {P}$, we have the $R$-module $\uHom(M,N)_s:=\Hom(M,N(s))$ and for each $s\le t$ we have {an} $R$-module homomorphisms $\uHom(M,N)_{s\le t}$ defined by $(\{\alpha_x:M_x\to N_{x+s}\}_{x\in P})\mapsto (\{N_{x+s\le x+t}\alpha_x:M_x\to N_{x+t}\}_{x\in P})$ or equivalently, by the naturality of $\alpha$, $(\{\alpha_x:M_x\to N_{x+s}\}_{x\in P})\mapsto (\{\alpha_{x+t}M_{x\le x+t}:M_x\to N_{x+s+t}\}_{x\in P})$.
\end{definition}

There is a canonical isomorphism $\text{Hom}(M,N(s))\cong \text{Hom}(M(-s),N)$ for all $s\in P$. Hence shifting the first argument in $\text{Hom}$ or the second one to construct $\uHom$ gives us the same definition. 

\begin{proposition}
\label{prop:limit_characterization_of_underline_hom}
(Limit characterization of $\uHom$)
Let $M,N$ be persistence modules (both left or both right). Then $\uHom(M,N)_{r}$ is the limit of the diagram $\{\emph{Hom}_R(M_{-s},N_{t})\}_{s+t\ge r}$.
\end{proposition}

\begin{proof}
Define $X_{r}=\prod_{s+t=r}\Hom_R(M_{-s},N_{t})$.
\begin{figure}[ht]
\centering
\begin{tikzpicture}[line cap=round,line join=round,x=1.0cm,y=1.0cm,scale=0.5]
\draw[->,color=black] (-3,0) -- (9,0);
\draw[->,color=black] (0,-0.5) -- (0,10.5);
\draw[color=black](10,0) node {M};
\draw[color=black](0,11) node {N};
\draw[color=black] (7.5,0.4) node[scale=0.9] {$b$};
\draw[color=black] (1.5,0.4) node[scale=0.9] {$a$};
\draw[color=black] (-0.5,3) node[scale=0.9]{$c$};
\draw[color=black] (-0.5,9) node[scale=0.9]{$d$};
\draw[color=black] (8,9.5) node[scale=0.7] {$\Hom_R(M_{-b}, N_{d})$};
\draw[color=black] (9.5,3) node[scale=0.7] {$\Hom_R(M_{-b}, N_{c})$};
\draw[color=black] (2.7,9.6) node[scale=0.7] {$\Hom_R(M_{-a}, N_{d})$};
\draw[color=black,thick] (-1,11) -- (9,1);
\fill [color=blue] (0,9) circle (1.5mm);
\fill [color=blue] (0,3) circle (1.5mm);
\fill [color=blue] (7,0) circle (1.5mm);
\fill [color=blue] (7,9) circle (1.5mm);
\fill [color=blue] (1,0) circle (1.5mm);
\fill [color=blue] (7,3) circle (1.5mm);
\fill [color=blue] (1,9) circle (1.5mm);

\draw[->,color=red,dashed,thick,shorten >=5pt] (7,3) -- (7,9);
\draw[->,color=red,dashed,thick,shorten >=5pt] (1,9) -- (7,9);
\draw[color=black] (4.5,8.5) node[scale=0.7] {$\rule{0.3cm}{0.15mm}\circ M_{-b\le -a}$};
\node[label={[label distance=0.5cm,text depth=-1ex,rotate=90,scale=0.7]right: $ N_{c\le d}\circ \rule{0.3cm}{0.15mm}$}] at (7,4) {};
\node[text width=7cm, color=black,anchor=west, right,scale=1] at (8.5,7)
  { $X_{r}=\displaystyle\prod\limits_{s+t=r}\Hom_R(M_{-s}, N_{t})$ };
\draw[color=black] (9.5,1) node[scale=1]{$X_{r} $};

\node[text width=7cm,color=black, anchor=west,right,scale=1] at (8.5,5)
 {$\uHom(M, N)_r=\varprojlim\limits_{s+t\ge r}\Hom_R(M_{-s},N_{t})$};

\end{tikzpicture}
\caption{Limit characterization of $\uHom$.}
\label{fig:16}
\end{figure} 
We claim that $\uHom(M,N)_{r}$ is the abelian subgroup of $X_r$ that is the limit of the diagram of abelian groups given by $\Hom_R(M_{-s},N_{t})$ with $s+t\ge r$ and maps as in Figure~\ref{fig:16}. Note that Figure~\ref{fig:16} illustrates the case in which $(P,\le)=(\mathbb{R},\le)$, but the algebra holds for the general case. To see this observe the following: Let $f\in \Hom_R(M_{-b}, N_{c})$ and $g\in \Hom_R(M_{-a},N_{d})$, where $a+d=b+c=r$. The canonical maps $\Hom_R(M_{-b},N_{c})\to \Hom_R(M_{-b},N_{d})$ and $\Hom_R(M_{-a},N_{d})\to \Hom_R(M_{-b},N_{d})$ that are induced by $M_{-b\le -a}$ and $N_{c\le d}$ are just postcomposition and precomposition by $N_{c\le d}$ and $M_{-b\le -a}$ respectively. If $f$ and $g$ are components of a natural transformation in $\Hom(M,N(r))$ then the parallelogram in Figure~\ref{fig:17} commutes. Equivalently, $f$ and $g$ are mapped to the same morphism under the above maps (see Figure~\ref{fig:17}). 
\begin{figure}[ht]
\centering

\begin{tikzcd}
M_{-b}\arrow[dr,"f"]\arrow[rr,"M_{-b\le -a}"]&&M_{-a}\arrow[dr,"g"]\\
&N_c\arrow[rr,"N_{c\le d}"]&&N_d
\end{tikzcd}
\caption{Commutativity of natural transformations is equivalent to a limit characterization of the appropriate hom sets. }
\label{fig:17}
\end{figure}
\end{proof}

In the remainder of this section we assume that $(P,\le)=(\mathbb{R}^n,\le)$ and $R=\mathbf{k}$ is a field.

\begin{example}
\label{example:underline_hom_of_interval_modules}
Consider two interval modules, say $\mathbf{k}[a,b)$ and $\mathbf{k}[c,d)$. Note that in the definition of $\uHom(M,N)$, we compute the direct sum of {abelian groups of} natural transformations between the persistence module $M$ and all translations of the persistence module $N$ on the real line. Thus, by using the same arguments as in Example~\ref{example:natural_transformations_form_a_vector_space}, $\uHom(\mathbf{k}[a,b),\mathbf{k}[c,d))$ is the interval module $\mathbf{k}[I]$ such that for all $t\in I$, $\uHom(\mathbf{k}[a,b),\mathbf{k}[c,d))_t=\mathbf{k}$ and $0$ otherwise. Depending on the lengths of the intervals $[a,b)$ and $[c,d)$ there are two cases to consider, namely $b-a\le d-c$ and $d-c\le b-a$. We can  calculate, accounting for both cases, that 
\[\uHom(\mathbf{k}[a,b),\mathbf{k}[c,d))= \mathbf{k}[\max\{c-a,d-b\},d-a)\,,\]

Alternatively, using Proposition~\ref{prop:limit_characterization_of_underline_hom} and reasoning similar to that used in Example~\ref{example:TensorOfIntervalModules} we can do the same calculation in terms of limits of diagrams of vector spaces, see Figure~\ref{fig:18}.
\begin{figure}[ht]
\centering
\begin{tikzpicture}[line cap=round,line join=round,x=1.0cm,y=1.0cm,scale=0.25]
\draw[color=red,line width=0.8mm] (-5,0)--(-2,0);
\draw[color=red,line width=0.8mm] (0,3) -- (0,9);
\draw[->,color=black,dashed] (-6,0) -- (2,0);
\draw[->,color=black,dashed] (0,-0.5) -- (0,10.5);
\fill[fill=red!25] (-2,3) rectangle (-5,9);
\draw[color=black,dashed] (-5,9) -- (-2,6);
\draw[color=black,thick] (-5,8)--(-2,5);
\draw[color=black,thick](-4,9)--(-2,7);
\draw[color=black,dashed](-5,8)--(-6,9);
\draw[color=black,dashed] (-5,6) -- (-2,3);
\draw[color=black] (-2.2,-0.5) node[scale=0.7] {$-a$};
\draw[color=black] (-5.2,-0.5) node[scale=0.7] {$-b$};
\draw[color=black] (-0.5,3) node[scale=0.7] {$c$};
\draw[color=black] (-0.5,9) node[scale=0.7] {$d$};
\fill [color=blue] (0,3) circle (1.5mm);
\fill [color=blue] (0,9) circle (1.5mm);
\fill [color=blue] (-5,0) circle (1.5mm);
\fill [color=blue] (-2,0) circle (1.5mm);
\end{tikzpicture}
\caption {\rmfamily $\uHom$ of interval modules: $\uHom(\mathbf{k}[a,b),\mathbf{k}[c,d)=\mathbf{k}[\text{max}\{c-a,d-b\},d-a)$.}
\label{fig:18}
\end{figure}
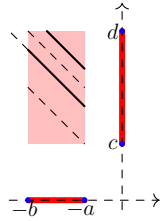 
Other formulas such as $\uHom(\mathbf{k}[a,b),\mathbf{k}[c,\infty)=0$, $\uHom(\mathbf{k}[a,\infty),\mathbf{k}[b,c))=\mathbf{k}[b-a,c-a)$ can be computed using the same arguments.
\end{example}

\begin{example}
\label{example:multiparameter_hom_underline_of_rect_modules}
Suppose that $M=\mathbf{k}[[a_1,b_1)\times\dots\times [a_n,b_n)]$ and $N=\mathbf{k}[[c_1,d_1)\times\dots\times [c_n,d_n)]$ are two rectangle modules. Then $\uHom(M,N)=\mathbf{k}[[\max\{c_1-a_1,d_1-b_1\},d_1-a_1)\times\dots\times [\max\{c_n-a_n,d_n-b_n\},d_n-a_n)]$.
\end{example}

\section{Duality}
\label{section:duality}

Let $(P,\le)$ be a {preorder}. Let $R$ be a commutative unital ring. 
For a persistence module we have duals from sheaf theory and from graded module theory. For the former, see \cite[Corollary 2.2.10.]{MR1074006} and for the latter see
\cite{miller2017data,lesnick2015interactive}.
{The graded module dual will be useful in determining which interval modules are flat and injective, which will be used in homological algebra computations to come.}

\begin{definition}
\label{def:sheaf_dual}
The \emph{sheaf dual} of a persistence module $M$ is the persistence module given by 
\[
M^*_{\mathbf{sh}}:=\scHom(M,R_P).
\]
 \end{definition}
 
 \begin{example}
 \label{example:sheaf_dual}
Let $(P,\le)=(\mathbb{R},\le)$ and let $R=\mathbf{k}$ be a field. Let $a\le b$ and consider the interval module $\mathbf{k}[a,b)$. Then, by Example~\ref{example:sheaf_hom_of_interval_modules}, we find $\mathbf{k}[a,b)^*_{\mathbf{sh}}:=\scHom(\mathbf{k}[a,b),\mathbf{k}[\mathbb{R}])=0$. Similarly, if $a\in \mathbb{R}$, we have $\mathbf{k}[a,\infty)^*_{\mathbf{sh}}:=\scHom(\mathbf{k}[a,\infty),\mathbf{k}[\mathbb{R}])=\mathbf{k}[\mathbb{R}]$.
 \end{example}

Now suppose that $(P,\le,+,0)$ is a {preorder} with a compatible abelian group action. The ring $R$ is still assumed commutative and unital.
\begin{definition} \label{def:Matlis-dual}
  The \emph{Matlis dual} of a persistence module $M$ is the persistence module given by
  \[
    M^*_{\mathbf{gr}}:=\uHom(M,R[D_{0}]]).
  \]
\end{definition}

\begin{lemma}
For $a \in P$, $\left(M^*_{\mathbf{gr}}\right)_a \isom \Hom_R(M_{-a},R)$.
\end{lemma}

\begin{proof}
Note that $\uHom(M,R[D_0])_a = \Hom(M,R[D_0](a)) = \Hom(M,R[D_{-a}])$. It remains to show that
\[\Hom(M,R[D_{-a}]) \isom \Hom_R(M_{-a},R).\]
For $\varphi: M \to R[D_{-a}]$, we have the component $\varphi_{-a}:M_{-a} \to R$.
For $f: M_{-a} \to R$, define $\varphi: M \to R[D_{-a}]$ by
$\varphi_{-a} = f$, for $x \leq -a$, $\varphi_x = fM_{x \leq -a}$, and let $\varphi_x$ be the zero map otherwise.
These two mappings provide the desired isomorphism.
Under this isomorphism, $\left(M^*_{\mathbf{gr}}\right)_{a \leq b}$ is given by the mapping $f \mapsto f \circ M_{-b\leq -a}$.
\end{proof}

Using the $\grtensor-\uHom$ adjunction
(see Theorem~\ref{theorem:graded_module_adjunction})
we have the following canonical isomorphism:
\[(M\grtensor N)^{*}_{\mathbf{gr}}=\uHom(M\grtensor N,R[D_{0}]])\cong\uHom(M,\uHom(N,R[D_{0}]))=\uHom(M,N^*_{\mathbf{gr}})\,.\]
Similarly, using the $\shtensor-\scHom$ adjunction, \cite[Proposition 2.2.9]{MR1074006}, we have the following canonical isomorphism:
\[(M\shtensor N)^{*}_{\mathbf{sh}}=\scHom(M\shtensor N,R_P)\cong \scHom(M,\scHom(N,R_P))=\scHom(M,N^*_{\mathbf{sh}})\,.\]

In the remainder of this section $(P,\le)=(\mathbb{R}^n,\le)$ and $R=\mathbf{k}$ is a field.

If a persistence module $M$ is pointwise finite dimensional we have $(M^*_{\mathbf{gr}})^*_{\mathbf{gr}}\cong M$. This is true since for finite dimensional vector spaces the same formula holds for vector space duals. In particular for a pointwise finite dimensional persistence module $M$, the module $M^*_{\mathbf{gr}}$ is in some sense the dilation of $M$ about the origin of scale factor $-1$.
 
\begin{example}
Consider an interval
module $\mathbf{k}[A]$. Then $\mathbf{k}[A]^*_{\mathbf{gr}}=\mathbf{k}[-A]$.
\end{example}

\begin{definition}
A graded module $M$ is \emph{$\grtensor$-flat} if $-\grtensor M$ is an exact functor.
\end{definition}

\begin{proposition}{\cite[Remark 4.20]{miller2017data}}
A persistence module $M$ is $\grtensor$-flat if and only if its Matlis dual $M^*$ is injective, and vice versa. In particular, $\mathbf{k}[A]$ is injective if and only if $\mathbf{k}[-A]$ is $\grtensor$-flat.
\label{prop:matlis}
\end{proposition}

\begin{remark}
\label{remark:adjoint_to_injective_is_flat}
Observe that we can use Matlis duality and the fact that injectivity of persistence modules is equivalent to their flabbiness (Appendix~\ref{sec:sheaves} and Proposition~\ref{proposition:BaerCriterionSheaves}) to classify {interval} modules into injectives and flats, see Figure~\ref{fig:3}, or use the Baer criterion (Proposition~\ref{prop:BaerCriterionForCategories}, Theorem~\ref{theorem:BaerCriterionGradedModules}, and Proposition~\ref{proposition:BaerCriterionSheaves}) if one prefers it over the flabbiness condition.
\end{remark}

\section{Classification of projective, injective and flat interval modules}
\label{subsection:classification_of_interval_modules}

In this section we assume that $(P,\le)=(\mathbb{R}^n,\le)$ and that $R=\mathbf{k}$ is a field. We classify interval modules (in the one-parameter case) into injectives and projectives and extend the results somewhat to the multi-parameter setting. {This is a necessary step for the homological algebra computations that are to come involving interval modules.}

\begin{proposition}
\label{prop:interval_module_that_is_flat_and_not_projective}
Let $a\in \mathbb{R}$. The interval module $\mathbf{k}(a,\infty)$ is not graded projective.
\end{proposition}

\begin{proof}
For simplicity we will prove the claim for $\mathbf{k}(0,\infty)$. Consider the following diagram
\[
  \begin{tikzcd}[row sep=scriptsize]
&&\mathbf{k}(0,\infty)\arrow[d,"\text{Id}"]\arrow[dll, "\beta" above, dashed]\\
\bigoplus\limits_{a>0}\mathbf{k}[a,\infty)\arrow[rr, "p" above]&&\mathbf{k}(0,\infty)\arrow[r] &0
\end{tikzcd}
\]
where $p$ is induced by the inclusions $\mathbf{k}[a,\infty) \hookrightarrow \mathbf{k}[0,\infty)$, $a > 0$. However, by Example~\ref{example:hom_examples}, $\mathbf{k}(0,\infty)$ has no nonzero maps to $\mathbf{k}[a,\infty)$, when $a>0$, because $(0,\infty)$ is not a subset of $[a,\infty)$. Thus $\beta=0$ but then the diagram cannot commute, and thus $\mathbf{k}(0,\infty)$ is not projective.
\end{proof}

In particular, submodules of free modules are not necessarily free, which is expected as the graded ring we are working with is not a principal ideal domain.

\vskip 0.1in
The following is an observation due to Parker Edwards.

\begin{lemma}
    If $a<c \in \R \cup \{\infty\}$ then
    $\colim_{a<b<c} \mathbf{k}[b,c) = \mathbf{k}(a,c)$.
    Dually, if $c<a \in \R \cup \{-\infty\}$, then
    ${\lim_{c<b<a}} \mathbf{k}(c,b] = \mathbf{k}(c,a)$.
\end{lemma}

\begin{lemma}
\label{lemma:colimits_of_gr_proj_are_flat}
Colimits of graded projective modules are $\grtensor$-flat.
\end{lemma}


\begin{corollary}
\label{corollary:gr_flat_interval_modules}
Let $a\in \mathbb{R}$. The interval module $\mathbf{k}(a,\infty)$ is $\grtensor$-flat.
\end{corollary}

We now prove the classification of projective, injective and $\grtensor$-flat interval modules stated in Theorem~\ref{theorem:classification_into_flats_and_injectives}.

\begin{theorem}[Theorem 1.2]
Let $a\in \mathbb{R}$. Then:
\begin{itemize}
\item The interval modules $\mathbf{k}(-\infty,a)$ and $\mathbf{k}(-\infty,a]$ are injective. They are not flat and thus not projective.
\item The interval modules $\mathbf{k}[a,\infty)$ are projective (free) and the interval modules $\mathbf{k}(a,\infty)$ are flat but not projective. Both are not injective.
\item The interval module $\mathbf{k}[\mathbb{R}]$ is both injective and flat, but not projective.
\item If $I\subset \mathbb{R}$ is a bounded interval, then $\mathbf{k}[I]$ is neither flat (hence not projective) nor injective.
\end{itemize}
\end{theorem}

\begin{proof}
  First, let us show that $\mathbf{k}(-\infty,a)$ is injective.
  By Corollary~\ref{corollary:gr_flat_interval_modules}, we know the interval module $\mathbf{k}(-a,\infty)$ is $\grtensor$-flat. By Proposition~\ref{prop:matlis} it follows that the interval module $\mathbf{k}(-\infty,a)$ is injective. 
  To see that $\mathbf{k}[a,\infty)$ is projective, note that it is a graded free module and is thus graded projective (hence $\grtensor$-flat). The statement that $\mathbf{k}(a,\infty)$ is $\grtensor$-flat and not projective is Proposition~\ref{prop:interval_module_that_is_flat_and_not_projective} and Corollary~\ref{corollary:gr_flat_interval_modules}.
Note that $\mathbf{k}[a,\infty)([a,\infty) \subset \R): \mathbf{k}[a,\infty)(\R) = 0 \to \mathbf{k}[a,\infty)[a,\infty) = \mathbf{k}$ is not surjective. Thus the sheaf $\mathbf{k}[a,\infty)$ is not flabby and hence by Proposition~\ref{proposition:BaerCriterionSheaves} is not injective.
  The same argument shows $\mathbf{k}(a,\infty)$ is not injective. By Proposition~\ref{prop:matlis}, $\mathbf{k}(-\infty,a)$ and $\mathbf{k}(-\infty,a]$ are not $\grtensor$-flat thus not projective.

The same argument used for $\mathbf{k}(-\infty,a)$ shows that $\mathbf{k}[\mathbb{R}]$ is injective. By Proposition~\ref{prop:matlis} $\mathbf{k}[-\mathbb{R}]=\mathbf{k}[\mathbb{R}]$ is $\grtensor$-flat.

  For a bounded interval $I\subset \mathbb{R}$ note that
 for $a \in I$,  $\mathbf{k}[I]([a,\infty) \subset \R)$ is not surjective.
  Thus the sheaf $\mathbf{k}[I]$ is not flabby thus not injective by Proposition~\ref{proposition:BaerCriterionSheaves}. Its Matlis dual $\mathbf{k}[-I]$ is thus not $\grtensor$-flat. By the same arguments $\mathbf{k}[-I]$, as $-I$ is a bounded interval, $\mathbf{k}[-I]$ is not injective thus by Proposition~\ref{prop:matlis} $\mathbf{k}[I]$ is not $\grtensor$-flat thus not projective. 
\end{proof}

For the multi-parameter case note that for $a\in \mathbb{R}^n$, the persistence module $\mathbf{k}[U_{a}]$ is graded free, thus graded projective, and hence $\grtensor$-flat. By Proposition~\ref{prop:matlis}, the persistence module $\mathbf{k}[D_{a}]$ is injective.

\begin{definition}
\label{def:join_and_meet_in_a_poset}
Let $(Q,\le)$ be a poset. Let $p,q\in Q$. The \emph{join of} $p$ and $q$ denoted $p\vee q$ is the smallest $r\in Q$ such that $p\le r$ and $q\le r$, if it exists. The \emph{meet of} $p$ and $q$ denoted $p\wedge q$ is the largest $t\in Q$ such that $t\le p$ and $t\le q$, if it exists. A poset where every join exists is called a \emph{join semilattice}. A poset where every meet exists is called a \emph{meet semilattice}. A poset where every join and every meet exists is called a \emph{lattice}. {Note that every up-set $U$ (and every down-set $D$) in a lattice is an interval.} 
\end{definition}

\begin{example}
\label{example:lattice}
The poset $(\mathbb{R}^n,\leq)$ is a lattice.
\end{example}

\begin{proposition} \label{prop:updown}
  Let $(P,\le)$ be a lattice. Consider a down-set $D\subset P$
  such that for all $a,b \in D$, the join $a\vee b$ is in $D$. Then the {interval} module $\mathbf{k}[D]$ is injective.
  Dually, for an up-set $U \subset \R^n$  with $a \wedge b \in U$ for all $a,b \in U$,  the {interval} module $\mathbf{k}[U]$ is $\grtensor$-flat.
\end{proposition}

\begin{proof}
  Let $D$ be as in the statement of the proposition. Observe that, viewing $\mathbf{k}[D]$ as a sheaf, we have $\mathbf{k}[D](P)=\lim_{x\in P}\mathbf{k}[D]_{x}\cong \mathbf{k}$. Let $U$ be an up-set of $P$ with $D \cap U \neq \emptyset$. Then, since the join for all $a,b\in D\cap U$ exists in $D\cap U$, it follows that $\mathbf{k}[D](U)=\lim_{x\in U}\mathbf{k}[D]_{x}\cong \mathbf{k}$. Since the nonzero maps in the module $\mathbf{k}[D]$ are identities, the induced map between the limits is an isomorphism. Hence $\mathbf{k}[D]$ is a flabby sheaf, hence an injective persistence module by Proposition~\ref{proposition:BaerCriterionSheaves}.
  The remainder of the statement follows from Proposition~\ref{prop:matlis}.
\end{proof}

\section{Derived functors for persistence modules}
\label{section:derived_functors_of_persistence_modules}

In this section we consider the derived functors of the following functors of persistence modules: $\grtensor, \uHom,\shtensor$ and $\scHom$. 
Throughout this section we assume $(P,\le)=(\mathbb{R}^n,\le)$ and that $R=\mathbf{k}$ is a field.
For one-parameter interval decomposable persistence modules we will use the classification of projective, injective, and flat interval modules (Theorem~\ref{theorem:classification_into_flats_and_injectives})
{in order to construct projective and injective resolutions and calculate these derived functors. The resulting formulas will be used in Section~\ref{section:kunneth}.}

\subsection{Graded module Tor and Ext}
Here we consider the derived functors $\mathbf{Tor}^{\mathbf{gr}}$ and $\mathbf{Ext}_{\mathbf{gr}}$ of the graded module tensor product $\otimes_{\mathbf{gr}}$ and its adjoint, $\uHom$.

\begin{example}[\cite{polterovich2017persistence}]
\label{example:tor_interval_modules}
Consider the interval modules $\mathbf{k}[a,b)$ and $\mathbf{k}[c,d)$.
We have the following augmented projective resolution of $\mathbf{k}[a,b)$.
\[0\to \mathbf{k}[b,\infty)\to \mathbf{k}[a,\infty)\to \mathbf{k}[a,b)\to 0
\]
Apply the functor $-\grtensor \mathbf{k}[c,d)$ to the projective resolution to get the following (no longer exact) sequence.
\[0\to \mathbf{k}[b+c,b+d)\to \mathbf{k}[a+c,a+d)\to  0
\]
Calculating homology (i.e. taking the kernel of the middle map) we get the following.
\[\mathbf{Tor}^{\mathbf{gr}}_1(\mathbf{k}[a,b),\mathbf{k}[c,d))=\mathbf{k}[\max\{a+d,b+c\},b+d)\,.\]
Similarly
$\mathbf{Tor}_1^{\mathbf{gr}}(\mathbf{k}[a,b),(-\infty,d)) = \mathbf{k}[a+d,b+d)$
and 
$\mathbf{Tor}_1^{\mathbf{gr}}(\mathbf{k}[a,\infty),\mathbf{k}[c,d)) = 0$. 
\end{example}

\begin{example}
\label{example:ext_interval_modules}
Consider the interval modules $\mathbf{k}[a,b)$ and $\mathbf{k}[c,d)$.
We have the following augmented injective resolution of $\mathbf{k}[c,d)$.
\[0\to \mathbf{k}[c,d)\to \mathbf{k}(-\infty,d)\to \mathbf{k}(-\infty,c)\to 0\]
Apply the functor $\uHom(\mathbf{k}[a,b),-)$ to the injective resolution. 
\[0\to \mathbf{k}[d-b,d-a)\to\mathbf{k}[c-b,c-a)\to 0\]
Calculating homology (i.e. taking the cokernel of the middle map) we get the following.
\[\mathbf{Ext}^1_{\mathbf{gr}}(\mathbf{k}[a,b),\mathbf{k}[c,d))=\mathbf{k}[c-b,\min\{c-a,d-b\})\]
\end{example}

In Example~\ref{example:tor_interval_modules} and Example~\ref{example:ext_interval_modules}, where $U_0\subseteq \mathbb{R}$, the given one-parameter persistence modules had projective, respectively injective resolutions, of length one. It is an open question whether all one-parameter persistence modules have projective resolutions of length one. More generally, for $U_0\subseteq \mathbb{R}^n$, it is unknown if the ring $\mathbf{k}[U_0]$ has a finite global dimension.

\subsection{Sheaf Tor and Ext}

Here we consider the derived functors $\mathbf{Tor}^{\mathbf{sh}}$ and $\mathbf{Ext}_{\mathbf{sh}}$ of the sheaf tensor product $\otimes_{\mathbf{sh}}$ and its adjoint $\scHom$.

\begin{theorem}
\label{theorem:sheaf_tensor_is_exact}
Let $M$ be a persistence module. Then $-\shtensor M$ and $M\shtensor -$ are exact functors. In particular, $\mathbf{Tor}_i^{\mathbf{sh}}(M,N)=0$ for any persistence modules $M$ and $N$ and any $i\ge 1$.
\end{theorem}

\begin{proof}
We will show that $-\shtensor M$ is exact. The other case is symmetric.

 Suppose $0\to A\to B\to C\to 0$ is a short exact sequence of persistence modules. A classical result in sheaf theory is that a sequence of morphisms $A\to B\to C$ of sheaves is short exact if and only if the induced maps on all the stalks are short exact. Thus for all $x\in \mathbb{R}^n$,  $0\to A_{x}\to B_{x}\to C_{x}\to 0$ is a short exact sequence of vector spaces. Now observe that applying the functor $-\shtensor M$, we get a sequence $A\shtensor M\to B\shtensor M\to C\shtensor M$ which gives us a sequence on stalks  $(A\shtensor M)_{x} \to (B\shtensor M)_{x}\to (C\shtensor M)_{x}$ which is equal to  $A_{x}\otimes_{\mathbf{k}}M_{x}\to B_{x}\otimes_{\mathbf{k}}M_{x}\to C_{x}\otimes_{\mathbf{k}}M_{x}$. Since every $\mathbf{k}$-vector space is a flat $\mathbf{k}$-module, and the sequence $A_{x}\to B_{x}\to C_{x}$ is short exact, the sequence $(A\shtensor M)_{x} \to (B\shtensor M)_{x}\to (C\shtensor M)_{x}$ is also exact, for all $x\in \mathbb{R}^n$. Thus  the sequence $A\shtensor M\to B\shtensor M\to C\shtensor M$ is also exact. Thus $-\shtensor M$ is an exact functor.
\end{proof}

It is not true in general that for any persistence module $M$ the functors $\scHom(-,M)$ and $\scHom(M,-)$ are exact. Thus we do have non-trivial $\mathbf{Ext}_{\mathbf{sh}}^i(M,N)$ groups for certain persistence modules $M$ and $N$, see Example~\ref{example:sheaf_ext_of_interval_modules}.

\begin{example}
\label{example:sheaf_ext_of_interval_modules}
Consider two interval modules $\mathbf{k}[a,b)$ and $\mathbf{k}[c,d)$.
We have the following augmented projective resolution.
\[0\to \mathbf{k}[b,\infty)\to \mathbf{k}[a,\infty)\to \mathbf{k}[a,b)\to 0\]
Apply the functor $\scHom(-,\mathbf{k}[c,d))$ to the projective resolution to get the (no longer exact) sequence:
\[0\to \scHom(\mathbf{k}[a,\infty),\mathbf{k}[c,d))\to \scHom(\mathbf{k}[b,\infty),\mathbf{k}[c,d))\to 0\]
Using Example~\ref{example:sheaf_hom_of_interval_modules} this sequence falls in one of the following cases:
\[\begin{cases}
0\to \mathbf{k}[c,d)\to \mathbf{k}[c,d)\to 0 & \text{if } a < b < c < d\\
0\to \mathbf{k}[c,d)\to \mathbf{k}(-\infty,d) \to 0& \text{if } a < c\le b < d\\
0\to \mathbf{k}[c,d)\to 0\to 0 & \text{if } a < c < d\le b\\
0\to \mathbf{k}(-\infty,d)\to \mathbf{k}(-\infty,d) \to 0 & \text{if } c\le a < b < d\\
0\to \mathbf{k}(-\infty,d)\to 0\to 0& \text{if } c\le a < d\le b\\
0\to 0\to 0\to 0 & \text{if } c < d\le a < b
\end{cases}
\]
By definition, $\mathbf{Ext}_{\mathbf{sh}}^1(\mathbf{k}[a,b),\mathbf{k}[c,d))$ is the cokernel of the middle morphisms. Thus we have the following.
\[\mathbf{Ext}_{\mathbf{sh}}^1(\mathbf{k}[a,b),\mathbf{k}[c,d))=\begin{cases}
\mathbf{k}(-\infty,c) & \text{if } a < c\le b < d\\
0 & \text{otherwise}
\end{cases}\] 
\end{example}
  
\section{K\"unneth theorems and universal coefficient theorems}
\label{section:kunneth}

In this section we state K\"unneth and {U}niversal {C}oefficient {T}heorems for chain complexes of persistence modules. We apply these theorems to products of filtered CW complexes. 
We will see that the K\"unneth formula for the additive product-filtration comes from graded module theory and that the K\"unneth formula for the maximum product-filtration comes from sheaf theory.

 Theorem~\ref{theorem:Kunneth_gr_1} and Theorem~\ref{theorem:Kunneth_gr_2} below imply the existence of certain natural short exact sequences, however additional assumptions are needed for these sequences to split. One of these is the assumption that the ring $R[U_{0}]$ is \emph{hereditary} (submodules of projective modules are projective). For example, the $G=\mathbb{R}^n$-graded ring $\mathbf{k}[U_0]$ is not hereditary. 
Indeed, for $a\in \mathbb{R}$ the interval module $\mathbf{k}[a,\infty)$ is projective, however its submodule $\mathbf{k}(a,\infty)$ is not
 (Theorem~\ref{theorem:classification_into_flats_and_injectives}).
However, the $\mathbb{Z}$-graded ring $\mathbf{k}[U_0]$ where $U_0$ is the principal up-set at $0$ of the poset $\mathbb{Z}$, is hereditary, since it is a principal ideal domain.
We can obtain splittings for more general persistence modules if they are left Kan extension of persistence modules indexed over $\Z$.
That is, if $M:\mathbf{Z}\to \mathbf{Vect}_{\mathbf{k}}$ is a persistence module and $i:\mathbf{Z}\to G$ is an inclusion of posets, then the left Kan extension of $M$ along $i$, is the persistence module given by $\overline{M}_{a}={\lim_{i(k)\le a}}M_{k}$.
In particular, when $n=1$ and $M$ is a real parameter persistence module isomorphic to a direct sum of interval modules, $M={\bigoplus_{i=1}^N}\mathbf{k}[a_i,b_i)$ then $M$ is obtained by such a Kan extension.
In these cases, by the functoriality of the left Kan extension,
the splitting of persistence modules indexed by $\Z$ provides a splitting of persistence modules indexed by $G$.

Below, let $(P,\le,+,0)$ be a {preorder} with a compatible abelian group structure. Let $*$ denote either $\mathbf{sh}$ or $\mathbf{gr}$.
Recall that given a chain complex $(K,d^K)$, valued in some abelian category, the subcomplex of boundaries is the chain complex $(L,d^L)$ where $L_n= d^K_n(K_n)$ and $d^L_n$ is the restriction of $d^K_n$, for all $n\in \mathbb{Z}$.

\begin{theorem}[K\"unneth Homology Theorem for Persistence Modules]
\label{theorem:Kunneth_gr_1} 
Let  $(K,d^K)$ be a chain complex of  $\otimes_*$-flat right persistence modules whose subcomplex of boundaries $B$ also has all terms $\otimes_*$-flat. Let $(L,d^L)$ be a chain complex of left persistence modules. Then:
\begin{itemize}
\item[1)] For every $n\in \mathbb{Z}$ there is a natural short exact sequence 
\begin{gather*}
0\to\bigoplus\limits_{p+q=n}(H_p(K)\otimes_* H_q(L))\to H_n(K\otimes_* L)\to\bigoplus\limits_{p+q=n-1}(
\emph{\textbf{Tor}}^{*}_1(H_p(K),H_q(L)))\to 0\,.
\end{gather*}
\item[2)] Suppose now that $R[U_{0}]$ is right hereditary and all terms in $(K,d^K)$ are projective, then the above sequence splits (the splitting need not be natural).
\end{itemize}
\end{theorem}

\begin{proof}
For part 1), adapt the proof of Theorem 3.6.3 in \cite{MR1269324}. Part 2) follows from Exercise 3.6.2 in \cite{MR1269324}.
\end{proof}

Recall that by Theorem~\ref{theorem:sheaf_tensor_is_exact} persistence modules with coefficients in a field are $\shtensor$-flat, hence there will be no $\mathbf{Tor}^{\mathbf{sh}}_1$ term present in the sequence above, if we work over a field $\mathbf{k}$.

\begin{theorem}[K\"unneth Cohomology Theorem for Persistence Modules]
\label{theorem:Kunneth_gr_2}
Let $(K,d^K)$ be a complex of left persistence modules such that all terms of $K$ and its subcomplex of boundaries $B$ are projective.
\begin{itemize}
\item[1)] For all $n\ge 0$ and every complex $(L, d^L)$ of left persistence modules, there is a natural short exact sequence
\[0\to \prod\limits_{p-q=n-1}\mathbf{Ext}_{*}^1(H_p(K),H_{-q}(L))\to H^n(\mathbf{Hom}^{*}(K,L))\to \prod\limits_{p-q=n}\Hom^*(H_p(K),H_{-q}(L)) \to 0\,.\]
where $\Hom^{*}$ is $\uHom$ if $*=\mathbf{gr}$ and $\scHom$ otherwise. 
\item[2)] If $R[U_{0}]$ is graded left hereditary, then the exact sequence splits for all $n\ge 0$.
\end{itemize}
\end{theorem}

\begin{proof}
In the ungraded module case, this is Exercise 3.6.1 in \cite{MR1269324} and Theorem 10.85 in \cite{rotman2008introduction}. The proof can be adapted to the graded case.
\end{proof}
 
We apply the K\"unneth theorems above to some simple filtered simplicial complexes.

\begin{example}
\label{example:kunneth_plus}
See Figure~\ref{fig:20}. Let $(K,d^K)$ and $(L, d^L)$ be chain complexes of persistence modules determined by filtrations of the $1$-simplex. In particular, let
\begin{gather*}
  K_0=\mathbf{k}[a_1,\infty)\oplus \mathbf{k}[b_1,\infty), \quad
  K_1=\mathbf{k}[c_1,\infty),\\
  L_0=\mathbf{k}[a_2,\infty)\oplus \mathbf{k}[b_2,\infty), \quad
  L_1=\mathbf{k}[c_2,\infty)
\end{gather*}
where $a_1\le b_1\le c_1$ and $a_2\le b_2\le c_2$, and let $d_K$ and $d_L$ be the induced boundary maps by the boundary maps of the $1$-simplex, as discussed in Section~\ref{sec:chain-complexes}. 
Note that the boundary subcomplexes of $K$ and $L$ are $\grtensor$-flat. Indeed, the only nontrivial boundary map are $d^K_1$ and $d^L_1$ and by construction $d^K_1(K_1)\cong\mathbf{k}[c_1,\infty)$ and $d^L_1(L_1)\cong \mathbf{k}[c_2,\infty)$ and we know these are in fact projective by Theorem~\ref{theorem:classification_into_flats_and_injectives}. Thus the hypotheses in Theorem~\ref{theorem:Kunneth_gr_1} are satisfied for both $(K,d^K)$ and $(L,d^L)$.
Now consider the product complex $K \grtensor L$ (Definition~\ref{def:monoidal_prod_of_chain_complexes}). Note that it is the chain complex of persistence modules corresponding to the filtered cubical given by the square in Figure~\ref{fig:20}, which assigns each cell the sum of filtration values of corresponding cells in the two 1-simplices.
\begin{figure}[ht]
\centering
\begin{tikzpicture}[line cap=round,line join=round,x=1.0cm,y=1.0cm,scale=1]
\draw[fill=green!20] (6,0) rectangle (8,2);
\draw[color=blue] (0,0)--(2,0);
\filldraw[blue] (1,0) circle (0pt) node[anchor=south,color=black] {$c_1$};
\draw[color=blue] (4,0)--(4,2);
\filldraw[blue] (4,1) circle (0pt) node[anchor=east,color=black] {$c_2$};
\draw[color=blue] (6,0)--(8,0);
\filldraw[blue] (7,0) circle (0pt) node[anchor=south,color=black,scale=0.8] {$c_1+a_2$};
\draw[color=blue] (8,0)--(8,2);
\filldraw[blue] (7,2) circle (0pt) node[anchor=north,color=black,scale=0.8] {$c_1+b_2$};
\draw[color=blue] (8,2)--(6,2);
\filldraw[blue] (8,1) circle (0pt) node[anchor=north,color=black,scale=0.8,rotate=90] {$b_1+c_2$};
\draw[color=blue] (6,2)--(6,0);
\filldraw[blue] (6,1) circle (0pt) node[anchor=south,color=black,scale=0.8,rotate=90] {$a_1+c_2$};
\draw[color=black] (3,1) node {$\times$};
\draw[color=black,->] (4.75,1)--(5.25,1);
\draw[color=black](7,1) node[scale=0.8] {$c_1+ c_2$};
\fill [color=red] (0,0) circle (0.5mm) node[anchor=south,color=black] {$a_1$};
\fill [color=red] (2,0) circle (0.5mm) node[anchor=south,color=black] {$b_1$};
\fill [color=red] (4,0) circle (0.5mm) node[anchor=east,color=black] {$a_2$};
\fill [color=red] (4,2) circle (0.5mm) node[anchor=east,color=black] {$b_2$};
\fill [color=red] (6,0) circle (0.5mm) node[anchor=north,color=black,scale=0.8] {$a_1+a_2$};
\fill [color=red] (8,0) circle (0.5mm) node[anchor=north,color=black,scale=0.8] {$b_1+a_2$};
\fill [color=red] (8,2) circle (0.5mm) node[anchor=south,color=black,scale=0.8] {$b_1+b_2$};
\fill [color=red] (6,2) circle (0.5mm) node[anchor=south,color=black,scale=0.8] {$a_1+b_2$};
\end{tikzpicture}
f\caption{A product complex, with respect to $\grtensor$, visualized.}
\label{fig:20}
\end{figure}

One can compute that the only non-trivial homology groups are:
\begin{gather*}
{H_0(K)=\mathbf{k}[a_1,\infty)\oplus\mathbf{k}[b_1,c_1), H_0(L)=\mathbf{k}[a_2,\infty)\oplus \mathbf{k}[b_2,c_2),}\\ 
H_0(K\grtensor L)=\mathbf{k}[a_1+a_2,\infty) \oplus \mathbf{k}[a_1+b_2,a_1+c_2) \oplus \\
\oplus\mathbf{k}[b_1+a_2,c_1+a_2) \oplus \mathbf{k}[b_1+b_2,\min\{b_1+c_2,c_1+b_2\}),\\
{\text{and } H_1(K\grtensor L)=\mathbf{k}[\max\{b_1+c_2,c_1+b_2\},c_1+c_2).}
\end{gather*}
Note that
$H_0(K \grtensor L) = H_0(K) \grtensor H_0(L)$ and
$H_1(K \grtensor L) = \mathbf{Tor}^{\mathbf{gr}}_1(H_0(K),H_0(L))$,
which agrees with Theorem~\ref{theorem:Kunneth_gr_1}.
\end{example}

\begin{example}
\label{example:kunneth_max}
  Let $(K,d^K)$ and $(L,d^L)$ be as in the previous example, where again $a_1\le b_1\le c_1$ and $a_2\le b_2\le c_2$. Now form the product complex
  $K \shtensor L$,
  recalling that $\mathbf{k}[a,\infty) \shtensor \mathbf{k}[b,\infty) = \mathbf{k}[\max(a,b),\infty)$ (Example~\ref{example:sheaf_tensor_of_interval_modules}).
The corresponding picture is given in Figure~\ref{fig:30}.
\begin{figure}[ht]
\centering
\begin{tikzpicture}[line cap=round,line join=round,x=1.0cm,y=1.0cm,scale=1]
\draw[fill=green!20] (6,0) rectangle (8,2);
\draw[color=blue] (0,0)--(2,0);
\filldraw[blue] (1,0) circle (0pt) node[anchor=south,color=black] {$c_1$};
\draw[color=blue] (4,0)--(4,2);
\filldraw[blue] (4,1) circle (0pt) node[anchor=east,color=black] {$c_2$};
\draw[color=blue] (6,0)--(8,0);
\filldraw[blue] (7,0) circle (0pt) node[anchor=south,color=black,scale=0.8] {$\max(c_1,a_2)$};
\draw[color=blue] (8,0)--(8,2);
\filldraw[blue] (7,2) circle (0pt) node[anchor=north,color=black,scale=0.8] {$\max(c_1,b_2)$};
\draw[color=blue] (8,2)--(6,2);
\filldraw[blue] (8,1) circle (0pt) node[anchor=north,color=black,scale=0.8,rotate=90] {$\max(b_1,c_2)$};
\draw[color=blue] (6,2)--(6,0);
\filldraw[blue] (6,1) circle (0pt) node[anchor=south,color=black,scale=0.8,rotate=90] {$\max(a_1,c_2)$};
\draw[color=black] (3,1) node {$\times$};
\draw[color=black,->] (4.75,1)--(5.25,1);
\draw[color=black](7,1) node[scale=0.8] {$\max(c_1,c_2)$};
\fill [color=red] (0,0) circle (0.5mm) node[anchor=south,color=black] {$a_1$};
\fill [color=red] (2,0) circle (0.5mm) node[anchor=south,color=black] {$b_1$};
\fill [color=red] (4,0) circle (0.5mm) node[anchor=east,color=black] {$a_2$};
\fill [color=red] (4,2) circle (0.5mm) node[anchor=east,color=black] {$b_2$};
\fill [color=red] (6,0) circle (0.5mm) node[anchor=north,color=black,scale=0.8] {${\max\{a_1,a_2\}}$};
\fill [color=red] (8,0) circle (0.5mm) node[anchor=north,color=black,scale=0.8] {${\max\{b_1,a_2\}}$};
\fill [color=red] (8,2) circle (0.5mm) node[anchor=south,color=black,scale=0.8] {${\max\{b_1,b_2\}}$};
\fill [color=red] (6,2) circle (0.5mm) node[anchor=south,color=black,scale=0.8] {${\max\{a_1,b_2\}}$};
\end{tikzpicture}
\caption{A product complex, with respect to $\shtensor$, visualized.}
\label{fig:30}
\end{figure}

In this case, recalling the discussion from Example~\ref{example:sheaf_tensor_of_interval_modules}, the only non-trivial homology groups are:
\begin{gather*}
{H_0(K)=\mathbf{k}[a_1,\infty)\oplus\mathbf{k}[b_1,c_1), H_0(L)=\mathbf{k}[a_2,\infty)\oplus \mathbf{k}[b_2,c_2),}\\
H_0(K\shtensor L) = \mathbf{k}[\max\{a_1,a_2\},\infty)\oplus
\mathbf{k}[\max\{a_1,b_2\},c_2)\oplus\\
\oplus\mathbf{k}[\max\{b_1,a_2\},c_1)\oplus
\mathbf{k}[\max\{b_1,b_2\},\min\{c_1,c_2\})=\\
{=H_0(K)\shtensor H_0(L).}
\end{gather*}
\end{example}

Example~\ref{example:kunneth_plus} and Example~\ref{example:kunneth_max} are specific instances of Theorem~\ref{theorem:kunneth_applications}.
Let $X$ and $Y$ be CW complexes with filtrations $f$ and $g$, respectively (see Section~\ref{sec:chain-complexes}). 
The CW complex $X \times Y$ has two canonical filtration given by $f+g$ and $\max(f,g)$, which we call the additive filtration and maximum filtration respectively.

\begin{theorem}
\label{theorem:kunneth_applications}
Let $(K,d_K)$ and $(L,d_L)$ be two chain complexes of persistence modules obtained from filtered CW complexes $X$ and $Y$ respectively (Section~\ref{sec:chain-complexes}). Then the additive and maximum filtrations on $X\times Y$ induce the chain complexes of persistence modules $K\grtensor L$ and $K\shtensor L$, respectively. In particular, we can calculate the persistent homology of these filtrations on $X\times Y$ by applying Theorem~\ref{theorem:Kunneth_gr_1}.
\end{theorem}

\begin{proof}
  Let $\sigma$ be an $n$-cell of $Y$ and let $\tau$ be an $m$-cell of $Y$. 
  These cells have corresponding free summands
  $\mathbf{k}[a_{\sigma},\infty)$ and $\mathbf{k}[b_{\tau},\infty)$ in $K_n$ and $L_m$ respectively (Section~\ref{sec:chain-complexes}).
Consider the additive filtration on $X \times Y$.
Then $\sigma \times \tau$ is a $(n+m)$-cell in $X \times Y$ with corresponding free summand
$\mathbf{k}[a_{\sigma},\infty)\grtensor \mathbf{k}[b_{\tau},\infty)=\mathbf{k}[a_{\sigma}+b_{\tau},\infty)$
in $(K \grtensor L)_{m+n}$
(Example~\ref{example:TensorOfIntervalModules}).
Note that this correspondence is compatible with 
the cellular boundary 
$\partial(\sigma\times \tau)=\partial(\sigma)\times \tau+(-1)^{|\sigma|}\sigma\times \partial(\tau)$
(see for example \cite[Proposition 3.B.1]{MR1867354}),
and the boundary map in $K \grtensor L$ (see Appendix~\ref{sec:homological-algebra}).
Thus $K\grtensor L$ is the chain complex of persistence modules induced by the additive filtration on $X\times Y$.

Similarly, by Example~\ref{example:sheaf_tensor_of_interval_modules}, $\mathbf{k}[a_{\sigma},\infty)\shtensor\mathbf{k}[b_{\tau},\infty)=\mathbf{k}[\max\{a_{\sigma},b_{\tau}\},\infty)$ and $K \shtensor L$ is the chain complex of persistence modules induced by the maximum filtration on $X \times Y$.
\end{proof}

The K\"unneth theorems allow us to compute homology of a tensor product of chain complexes of persistence modules (or cohomology of its adjoint). Thus, as a special case of the K\"unneth Theorem for persistence modules, we have Theorem~\ref{theorem:universal_coefficients_1} and Theorem~\ref{theorem:universal_coefficients_2}, where the second chain complex of persistence modules is assumed to be concentrated in degree $0$.
\begin{theorem}
[Universal Coefficient Homology Theorem for Persistence Modules] 
\label{theorem:universal_coefficients_1}
Let $A$ be a left persistence module and let $(K,d)$ be a chain complex  of $\otimes_*$-flat right persistence modules whose subcomplex of boundaries $B$ also has all terms $\otimes_*$-flat. Then
\begin{itemize}
\item[1)] for all $n\in \mathbb{N}$, there is a natural exact sequence
\[0\to H_n(K)\otimes_*A\to H_n(K\otimes_* A)\to \mathbf{Tor}_1^*(H_{n-1}(K,A))\to 0\]
\item[2)] Assuming the ring in question is right-hereditary (right submodules of right projective modules are projective) and $(K,d)$ has all terms projective (no assumptions on $B$ this time), the above sequence splits (it need not be a natural splitting).
\end{itemize}
\end{theorem}

\begin{theorem}
[Universal Coefficient Cohomology Theorem for Persistence Modules] 
\label{theorem:universal_coefficients_2}
Let $A$ be a left persistence module, let $(K,d)$ be a complex of projective left persistence modules whose subcomplex $B$ of boundaries has all terms projective.
\begin{itemize}
\item[1)] Then for all $n\in \mathbb{N}$ there is a natural short exact sequence
\[0\to \mathbf{Ext}_*^1(H_{n-1}(K),A)\to H^n(\mathbf{Hom}^*(K,A))\to \Hom^*(H_n(K),A)\to 0\,\]
where $\Hom^{*}$ is $\uHom$ if $*=\mathbf{gr}$ and $\scHom$ otherwise. 
\item[2)] If the ring in question is left-hereditary then the above splits (need not be a natural splitting).
\end{itemize}

\end{theorem}

We now consider some examples in the one-parameter setting and assuming that the coefficient ring is a field $\mathbf{k}$.

\begin{example}
Let $a\le b\le c\le d\le e\le f\le g$ be real numbers and consider the filtration of the $2$-simplex in Figure~\ref{fig:21}.
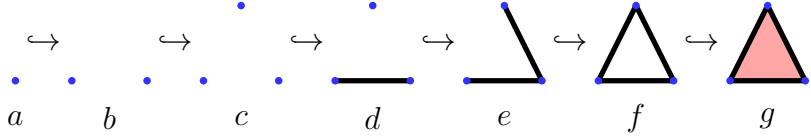
\begin{figure}[ht]
\centering
\begin{tikzpicture}[line cap=round,line join=round,x=1.0cm,y=1.0cm,scale=0.5]



\fill [color=blue!80] (-1,2) circle (1mm);

\draw[color=black] (-1,1) node {$a$};
\draw[color=black] (-0.25,3) node[scale=1] {$\hookrightarrow$};

\fill [color=blue!80] (0.5,2) circle (1mm);
\fill [color=blue!80] (2.5,2) circle (1mm);

\draw[color=black] (1.5,1) node {$b$};
\draw[color=black] (3.25,3) node[scale=1] {$\hookrightarrow$};
\fill[color=blue!80] (4,2) circle(1mm);
\fill[color=blue!80] (6,2) circle(1mm);
\fill[color=blue!80] (5,4) circle (1mm);
\draw[color=black] (6.75,3) node[scale=1] {$\hookrightarrow$};
\draw[color=black] (5,1) node {$c$};
\draw[color=black, line width=0.7mm] (7.5,2)--(9.5,2);
\fill[color=blue!80](7.5,2) circle(1mm);
\fill[color=blue!80] (9.5,2) circle (1mm);
\fill[color=blue!80] (8.5,4) circle (1mm);
\draw[color=black] (8.5,1) node {$d$};

\draw[color=black] (10.25,3) node[scale=1] {$\hookrightarrow$};

\draw[color=black, line width=0.7mm] (11,2)--(13,2);
\draw[color=black, line width=0.7mm] (13,2)--(12,4);
\fill[color=blue!80](11,2) circle(1mm);
\fill[color=blue!80] (13,2) circle (1mm);
\fill[color=blue!80] (12,4) circle (1mm);
\draw[color=black] (12,1) node {$e$};

\draw[color=black] (13.75,3) node[scale=1] {$\hookrightarrow$};

\draw[color=black, line width=0.7mm] (14.5,2)--(16.5,2);
\draw[color=black, line width=0.7mm] (16.5,2)--(15.5,4);
\draw[color=black, line width=0.7mm] (15.5,4)--(14.5,2);
\fill[color=blue!80](14.5,2) circle(1mm);
\fill[color=blue!80] (16.5,2) circle (1mm);
\fill[color=blue!80] (15.5,4) circle (1mm);

\draw[color=black] (15.5,1) node {$f$};

\draw[color=black] (17.25,3) node[scale=1] {$\hookrightarrow$};

\draw[fill=red!35, line width=0.7mm] (18,2)--(20,2)--(19,4)--cycle;
\fill[color=blue!80](18,2) circle(1mm);
\fill[color=blue!80] (20,2) circle (1mm);
\fill[color=blue!80] (19,4) circle (1mm);
\draw[color=black] (19,1) node {$g$};
\end{tikzpicture}
\caption{A filtration of a triangle.}
\label{fig:21}
\end{figure}
\label{example:triangle_filtration}

The corresponding chain complex of persistence modules
(Section~\ref{sec:chain-complexes})
is given by
\begin{gather*}
{
K_0=\mathbf{k}[a,\infty)\oplus\mathbf{k}[b,\infty)\oplus\mathbf{k}[c,\infty),}\\
{K_1=\mathbf{k}[d,\infty)\oplus\mathbf{k}[e,\infty)\oplus \mathbf{k}[f,\infty),}\\
{\text{and } K_2=\mathbf{k}[g,\infty).}
\end{gather*}
We compute 
$H_0(K)=\mathbf{k}[a,\infty)\oplus\mathbf{k}[b,d)\oplus \mathbf{k}[c,e)$, and
$H_1(K)=\mathbf{k}[f,g)$.

\end{example}

\begin{example}
Let $(K,d)$ be the chain complex of persistence modules in Example~\ref{example:triangle_filtration}. Let $A=\mathbf{k}[\alpha,\infty)$. Let us compute $H_*(K\grtensor A)$. Using Theorem~\ref{theorem:universal_coefficients_1}, since $A$ is free, $\mathbf{Tor}^{\mathbf{gr}}_i(H_{n-1}(K),A)=0$ for all $i\ge 1$, hence $H_n(K\grtensor A)\cong H_n(K)\grtensor A$.
  Thus,
  \begin{gather*}
{H_0(K)=\mathbf{k}[a+\alpha,\infty)\oplus\mathbf{k}[b+\alpha,d)\oplus \mathbf{k}[c+\alpha,e),}\\
{\text{and } H_1(K)=\mathbf{k}[f+\alpha,g).}
\end{gather*}
That is, all interval modules have shifted to the right by $\alpha$.
If $K$ is obtained from a filtration $f$, then $K \grtensor A$ is the chain complex obtained from the filtration $f + \alpha$.
\end{example}

\begin{example}
Let $(K,d)$ be the chain complex of persistence modules in Example~\ref{example:triangle_filtration}. Let $A=\mathbf{k}[\mathbb{R}]$. Let us compute $H_*(K\grtensor A)$.
By Theorem~\ref{theorem:universal_coefficients_1}, $H_n(K\grtensor A)=H_n(K)\grtensor A$ since $A$ is $\grtensor$-flat. From Example~\ref{example:TensorOfIntervalModules}, we have $\mathbf{k}[a,b)\grtensor A= 0$ for all $a \leq b\in \mathbb{R}$, and
$\mathbf{k}[a,\infty)\grtensor A=A$ for all $a\in \mathbb{R}$.
Therefore, $H_0(K) = \mathbf{k}[\R]$ and $H_1(K) = 0$.
\end{example}

\begin{example}
\label{example:UCT_example_1}
Let $(K,d)$ be the chain complex of persistence modules in Example~\ref{example:triangle_filtration}.
Let $A=\mathbf{k}(-\infty,0)$.
Applying Example~\ref{example:TensorOfIntervalModules} we have 
\begin{gather*}
{(K\grtensor A)_0:=K_0\grtensor A=\mathbf{k}(-\infty,a)\oplus\mathbf{k}(-\infty,b)\oplus\mathbf{k}(-\infty,c),}\\ 
{(K\grtensor A)_1:=K_1\grtensor A=\mathbf{k}(-\infty,d)\oplus\mathbf{k}(-\infty,e)\oplus\mathbf{k}(-\infty,f),}\\
{\text{and } (K\grtensor A)_2:=K_2\grtensor A=\mathbf{k}(-\infty,g).} 
\end{gather*}
Applying Theorem~\ref{theorem:universal_coefficients_1}
and Example~\ref{example:tor_interval_modules} we calculate the following:
\begin{gather*}
{H_0(K\grtensor A)\cong H_0(K)\grtensor A=\mathbf{k}(-\infty,a),}\\ 
{H_1(K\grtensor A)\cong H_1(K)\grtensor A\oplus \mathbf{Tor}^{\mathbf{gr}}(H_0(K),A)
= 0\oplus \mathbf{k}(b,d)\oplus\mathbf{k}(c,e),}\\ 
{\text{and } H_2(K\grtensor A)\cong \mathbf{Tor}^{\mathbf{gr}}(H_1(K),A)=\mathbf{k}(f,g).}
\end{gather*}

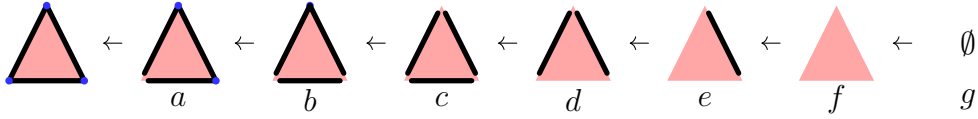
\begin{figure}[ht]
\centering
\begin{tikzpicture}[line cap=round,line join=round,x=1.0cm,y=1.0cm,scale=0.5]


\draw[fill=red!35] (-1,2)--(1,2)--(0,4)--cycle;
\draw[color=black, line width=0.7mm] (-1,2)--(1,2);
\draw[color=black, line width=0.7mm] (1,2)--(0,4);
\draw[color=black, line width=0.7mm] (0,4)--(-1,2);

\fill[color=blue!80](-1,2) circle(1mm);
\fill[color=blue!80] (1,2) circle (1mm);
\fill[color=blue!80] (0,4) circle (1mm);

\draw[<-,color=black] (1.5,3) -- (2,3);




\draw[draw=none, fill=red!35] (4.5,2)--(3.5,4)--(2.5,2)--cycle;
\draw[color=black, line width=0.7mm] (4.5,2)--(3.5,4);
\draw[color=black, line width=0.7mm] (4.5,2)--(2.7,2);
\draw[color=black, line width=0.7mm] (3.5,4)--(2.6,2.2);
\fill [color=blue!80] (4.5,2) circle (1mm);
\fill [color=blue!80] (3.5,4) circle (1mm);
\draw[color=black] (3.5,1.5) node {$a$};
\draw[<-,color=black] (5,3) -- (5.5,3);
\draw[draw=none, fill=red!35] (6,2)--(8,2)--(7,4)--cycle;
\fill[color=blue!80] (7,4) circle (1mm);

\draw[color=black, line width=0.7mm] (7,4)--(7.9,2.2);
\draw[color=black, line width=0.7mm] (7.8,2)--(6.2,2);
\draw[color=black, line width=0.7mm] (6.1,2.2)--(7,4);

\draw[<-,color=black] (8.5,3)--(9,3);
\draw[color=black] (7,1.5) node {$b$};
\draw[draw=none, fill=red!35] (9.5,2)--(11.5,2)--(10.5,4)--cycle;
\draw[color=black, line width=0.7mm] (9.7,2)--(11.3,2);
\draw[color=black, line width=0.7mm] (9.6,2.2)--(10.4,3.8);
\draw[color=black, line width=0.7mm] (11.4,2.2)--(10.6,3.8);
\draw[color=black] (10.5,1.5) node {$c$};

\draw[<-, color=black] (12,3)--(12.5,3);
\draw[draw=none, fill=red!35] (13,2)--(15,2)--(14,4)--cycle;
\draw[color=black, line width=0.7mm] (13.1,2.2)--(13.9,3.8);
\draw[color=black, line width=0.7mm] (14.9,2.2)--(14.1,3.8);
\draw[color=black] (14,1.5) node {$d$};

\draw[<-, color=black] (15.5,3)--(16,3);

\draw[draw=none, fill=red!35] (16.5,2)--(18.5,2)--(17.5,4)--cycle;

\draw[color=black, line width=0.7mm] (18.4,2.2)--(17.6,3.8);

\draw[color=black] (17.5,1.5) node {$e$};

\draw[<-, color=black] (19,3)--(19.5,3);

\draw[draw=none,fill=red!35] (20,2)--(22,2)--(21,4);

\draw[color=black] (21,1.5) node {$f$};

\draw[<-,color=black] (22.5,3)--(23,3);
\draw[color=black] (24.5,3) node {$\emptyset$};
\draw[color=black] (24.5,1.5) node {$g$};
\end{tikzpicture}
\caption{A filtration of the geometric realization of $\Delta^2$ corresponding to the chain complex $K\grtensor \mathbf{k}(-\infty,0)$.}
\label{fig:23}
\end{figure}

\end{example}

\begin{remark}
\label{remark:UCT_explaining_cofiltrations}
We thank Alexander Dranishnikov for the following observation.  The persistence barcodes in Example~\ref{example:UCT_example_1} correspond to the compactly supported cohomology groups of the filtration of topological spaces in Figure~\ref{fig:23}. {It may be that this observation can be generalized to an arbitrary filtered CW complex.
  We leave it as question for future work.}
\end{remark}

\begin{example}
\label{example:tensoring_with_finite_bar}
Let $(K,d)$ be the chain complex of persistence modules in Example~\ref{example:triangle_filtration}.
Let $A=\mathbf{k}[\alpha,\beta)$.
Using Example~\ref{example:TensorOfIntervalModules} we have
\begin{gather*}
{(K\grtensor A)_0:=K_0\grtensor A=\mathbf{k}[a+\alpha,a+\beta)\oplus \mathbf{k}[b+\alpha,b+\beta)\oplus \mathbf{k}[c+\alpha,c+\beta),}\\ 
{(K\grtensor A)_1:=K_1\grtensor A=\mathbf{k}[d+\alpha,d+\beta)\oplus \mathbf{k}[e+\alpha,e+\beta)\oplus\mathbf{k}[f+\alpha,f+\beta),}\\ 
{\text{and } (K\grtensor A)_2=K_2\otimes_{\mathbf{gr}}A=\mathbf{k}[g+\alpha,g+\beta).}
\end{gather*}
Applying Theorem~\ref{theorem:universal_coefficients_1} and Example~\ref{example:tor_interval_modules} we have the following:
\begin{gather*}
H_0(K\grtensor A)\cong H_0(K)\grtensor A=\mathbf{k}[a+\alpha,a+\beta)\oplus\\
\oplus\mathbf{k}[b+\alpha,\min\{d+\alpha,b+\beta\})
\oplus\mathbf{k}[c+\alpha,\min\{e+\alpha,c+\beta\}),\\
H_1(K\grtensor A)\cong \mathbf{Tor}^{\mathbf{gr}}(H_0(K),A)\oplus H_1(K)\grtensor A=0 \oplus \mathbf{k}[\max\{b+\beta,d+\alpha\},d+\beta)\oplus\\
\oplus \mathbf{k}[\max\{c+\beta, e+\alpha\},e+\beta)\oplus \mathbf{k}[f+\alpha,\min\{f+\beta,g+\alpha\},\\
{\text{and } H_2(K\grtensor A)\cong \mathbf{Tor}^{\mathbf{gr}}(H_1(K), A)=\mathbf{k}[\max\{f+\beta,g+\alpha\}, g+\beta).}
\end{gather*}
Once again, there is a geometric interpretation. If we examine the chain groups $K\grtensor A$ then we see that for each simplex appearing at time $t$ in the original filtration, it now appears at time $t+\alpha$ and is removed at time $t+\beta$. For example an edge generates a homology class when both its boundary points are removed. 
\end{example}

\begin{example}
Let $A=\mathbf{k}[\alpha,\infty)$. Let $(K,d)$ be the chain complex  in Example~\ref{example:triangle_filtration}. Then by using Example~\ref{example:underline_hom_of_interval_modules} we calculate that:
\begin{gather*}
{\mathbf{Hom}_{\mathbf{gr}}(K,A)^0:=\uHom(K_0,A)=\mathbf{k}[\alpha-a,\infty)\oplus\mathbf{k}[\alpha-b,\infty)\oplus \mathbf{k}[\alpha-c,\infty),}\\ 
{\mathbf{Hom}_{\mathbf{gr}}(K,A)^1:=\uHom(K_1,A)=\mathbf{k}[\alpha-d,\infty)\oplus \mathbf{k}[\alpha-e,\infty)\oplus \mathbf{k}[\alpha-f,\infty),}\\
{\text{and } \mathbf{Hom}_{\mathbf{gr}}(K,A)^2:=\uHom(K_2,A)=\mathbf{k}[\alpha-g,\infty).}
\end{gather*}
By Theorem~\ref{theorem:universal_coefficients_2}, Example~\ref{example:ext_interval_modules} and Theorem~\ref{theorem:classification_into_flats_and_injectives} we have 
\begin{gather*}
{H^0(\mathbf{Hom}_{\mathbf{gr}}(K,A))\cong\uHom(H_0(K),A)=\mathbf{k}[\alpha-a,\infty),}\\ 
H^1(\mathbf{Hom}_{\mathbf{gr}}(K,A))\cong\mathbf{Ext}_{\mathbf{gr}}(H_0(K),A)\oplus \uHom(H_1(K),A)=\\
={\mathbf{k}[\alpha-d,\alpha-b)\oplus \mathbf{k}[\alpha-e,\alpha-c)\oplus 0=\mathbf{k}[\alpha-d,\alpha-b)\oplus \mathbf{k}[\alpha-e,\alpha-c),}\\ 
{\text{and }  H^2(\mathbf{Hom}_{\mathbf{gr}}(K,A))\cong\mathbf{Ext}_{\mathbf{gr}}(H_1(K),A)=\mathbf{k}[\alpha-g,\alpha-f).}
\end{gather*}
This also has a geometric interpretation (see Figure~\ref{fig:24}).
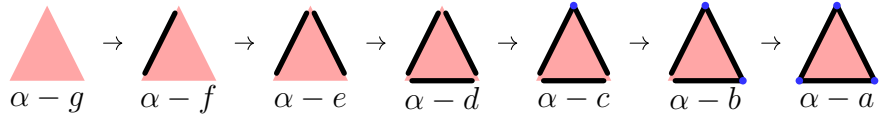
\begin{figure}[ht]
\centering
\begin{tikzpicture}
[line cap=round,line join=round,x=1.0cm,y=1.0cm,scale=0.5]


\draw[draw=none,fill=red!35] (-1,2)--(1,2)--(0,4)--cycle;
\draw[color=black] (0,1.5) node {$\alpha-g$};


\draw[->,color=black] (1.5,3) -- (2,3);




\draw[draw=none, fill=red!35] (4.5,2)--(3.5,4)--(2.5,2)--cycle;
\draw[color=black, line width=0.7mm] (3.4,3.8)--(2.6,2.2);
\draw[color=black] (3.5,1.5) node {$\alpha-f$};
\draw[->,color=black] (5,3) -- (5.5,3);
\draw[draw=none, fill=red!35] (6,2)--(8,2)--(7,4)--cycle;
\draw[color=black, line width=0.7mm] (7.1,3.8)--(7.9,2.2);
\draw[color=black, line width=0.7mm] (6.1,2.2)--(6.9,3.8);

\draw[->,color=black] (8.5,3)--(9,3);
\draw[color=black] (7,1.5) node {$\alpha-e$};
\draw[draw=none, fill=red!35] (9.5,2)--(11.5,2)--(10.5,4)--cycle;
\draw[color=black, line width=0.7mm] (9.7,2)--(11.3,2);
\draw[color=black, line width=0.7mm] (9.6,2.2)--(10.4,3.8);
\draw[color=black, line width=0.7mm] (11.4,2.2)--(10.6,3.8);
\draw[color=black] (10.5,1.5) node {$\alpha-d$};

\draw[->, color=black] (12,3)--(12.5,3);
\draw[draw=none, fill=red!35] (13,2)--(15,2)--(14,4)--cycle;
\draw[color=black, line width=0.7mm] (13.2,2)--(14.8,2);
\draw[color=black, line width=0.7mm] (13.1,2.2)--(14,4);
\draw[color=black, line width=0.7mm] (14.9,2.2)--(14,4);
\fill[color=blue!80] (14,4) circle (1mm);
\draw[color=black] (14,1.5) node {$\alpha-c$};

\draw[->, color=black] (15.5,3)--(16,3);

\draw[draw=none, fill=red!35] (16.5,2)--(18.5,2)--(17.5,4)--cycle;

\draw[color=black, line width=0.7mm] (16.7,2)--(18.5,2);
\draw[color=black, line width=0.7mm] (18.5,2)--(17.5,4);
\draw[color=black, line width=0.7mm] (16.6,2.2)--(17.5,4);
\fill[color=blue!80] (18.5,2) circle (1mm);
\fill[color=blue!80] (17.5,4) circle (1mm);

\draw[color=black] (17.5,1.5) node {$\alpha-b$};

\draw[->, color=black] (19,3)--(19.5,3);

\draw[draw=none,fill=red!35] (20,2)--(22,2)--(21,4);
\draw[color=black, line width=0.7mm] (20,2)--(22,2);
\draw[color=black, line width=0.7mm] (22,2)--(21,4);
\draw[color=black, line width=0.7mm] (21,4)--(20,2);
\fill[color=blue!80] (20,2) circle (1mm);
\fill[color=blue!80] (22,2) circle (1mm);
\fill[color=blue!80] (21,4) circle (1mm);
\draw[color=black] (21,1.5) node {$\alpha-a$};

\end{tikzpicture}
\caption{A filtration of the geometric realization of $\Delta^2$ corresponding to $\mathbf{Hom}_{\mathbf{gr}}(K,A)$.}
\label{fig:24}
\end{figure}
In particular, each cell in the original simplicial complex which appeared at time $t$, now appears at time $\alpha-t$.
\end{example}

\begin{example}
Let $(K,d)$ be the chain complex in Example~\ref{example:triangle_filtration}. Let $A=\mathbf{k}(-\infty,\alpha)$. Then by using Example~\ref{example:underline_hom_of_interval_modules} we have:
\begin{gather*}
{\mathbf{Hom}_{\mathbf{gr}}(K,A)^0:=\uHom(K_0,A)=\mathbf{k}(-\infty,\alpha-a)\oplus\mathbf{k}(-\infty,\alpha-b)\oplus\mathbf{k}(-\infty,\alpha-c),}\\
{\mathbf{Hom}_{\mathbf{gr}}(K,A)^1:=\uHom(K_1,A)=\mathbf{k}(-\infty,\alpha-d)\oplus\mathbf{k}(-\infty,\alpha-e)\oplus\mathbf{k}(-\infty,\alpha-f),}\\
{\text{and } \mathbf{Hom}_{\mathbf{gr}}(K,A)^2:=\uHom(K_2,A)=\mathbf{k}(-\infty,\alpha-g).}
\end{gather*}
Noting that $A$ is injective, by Theorem~\ref{theorem:universal_coefficients_2} we have that:
\begin{gather*}
{H^0(\mathbf{Hom}_{\mathbf{gr}}(K,A))\cong \uHom(H_0(K),A)=\mathbf{k}(-\infty ,\alpha-a)\oplus\mathbf{k}(\alpha-d,\alpha-b)\oplus \mathbf{k}(\alpha-e,\alpha-c),}\\
 \text{and } H^1(\mathbf{Hom}_{\mathbf{gr}}(K,A))\cong \mathbf{Ext}_{\mathbf{gr}}(H_0(K),A)\oplus\uHom(H_1(K),A)=\\
 =0\oplus\mathbf{k}(\alpha- g,\alpha-f)=\mathbf{k}(\alpha- g,\alpha-f).
\end{gather*}
As before, there is a filtration (see Figure~\ref{fig:25}) and the persistence module may be interpreted as arising from the cohomology of this filtration. {It is not yet clear how this generalizes to arbitrary CW complexes with a filtration. We leave this question for future work.}
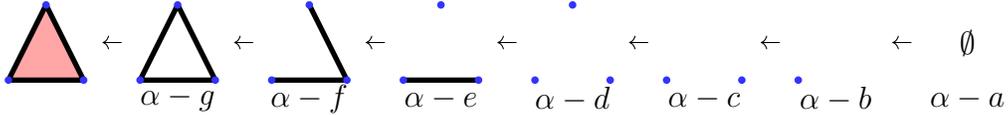
\begin{figure}[ht]
\centering
\begin{tikzpicture}
[line cap=round,line join=round,x=1.0cm,y=1.0cm,scale=0.5]

\draw[fill=red!35, line width=0.7mm] (-1,2)--(1,2)--(0,4)--cycle;
\fill[color=blue!80] (-1,2) circle(1mm);
\fill[color=blue!80] (1,2) circle(1mm);
\fill[color=blue!80] (0,4) circle(1mm);


\draw[<-,color=black] (1.5,3) -- (2,3);




\draw[color=black, line width=0.7mm] (4.5,2)--(3.5,4)--(2.5,2)--cycle;
\fill[color=blue!80] (4.5,2) circle(1mm);
\fill[color=blue!80] (2.5,2) circle(1mm);
\fill[color=blue!80] (3.5,4) circle(1mm);
\draw[color=black] (3.5,1.5) node {$\alpha-g$};
\draw[<-,color=black] (5,3) -- (5.5,3);
\draw[color=black, line width=0.7mm] (6,2)--(8,2);
\draw[color=black, line width=0.7mm] (8,2)--(7,4);
\fill[color=blue!80] (6,2) circle(1mm);
\fill[color=blue!80] (8,2) circle(1mm);
\fill[color=blue!80] (7,4) circle(1mm);

\draw[<-,color=black] (8.5,3)--(9,3);
\draw[color=black] (7,1.5) node {$\alpha-f$};
\draw[color=black, line width=0.7mm] (9.5,2)--(11.5,2);
\fill[color=blue!80] (9.5,2) circle(1mm);
\fill[color=blue!80] (11.5,2) circle(1mm);
\fill[color=blue!80] (10.5,4) circle(1mm);

\draw[color=black] (10.5,1.5) node {$\alpha-e$};

\draw[<-, color=black] (12,3)--(12.5,3);


\fill[color=blue!80] (13,2) circle(1mm);
\fill[color=blue!80] (15,2) circle(1mm);
\fill[color=blue!80] (14,4) circle(1mm);

\draw[color=black] (14,1.5) node {$\alpha-d$};

\draw[<-, color=black] (15.5,3)--(16,3);


\fill[color=blue!80] (16.5,2) circle (1mm);
\fill[color=blue!80] (18.5,2) circle (1mm);

\draw[color=black] (17.5,1.5) node {$\alpha-c$};

\draw[<-, color=black] (19,3)--(19.5,3);


\fill[color=blue!80] (20,2) circle (1mm);

\draw[color=black] (21,1.5) node {$\alpha-b$};

\draw[<-,color=black] (22.5,3)--(23,3);

\draw[color=black] (24.5,1.5) node {$\alpha-a$};
\draw[color=black] (24.5,3) node {$\emptyset$};
\end{tikzpicture}
\caption{A filtration corresponding to $\mathbf{Hom}_{\mathbf{gr}}(K,A)$}
\label{fig:25}
\end{figure}

Each simplex in the original simplicial complex which appeared at time $t$ now appears at time $\alpha-t$.
Note that if $\alpha=0$, then $\uHom(H_n(K),A)=H_n(K)^*_{\mathbf{gr}}=H^n(\uHom(K,A))=H^n(K^*_{\mathbf{gr}})$, generalizing the classical result that homology and cohomology, with coefficients in a field, are isomorphic.
\end{example}

\begin{example}
Let $A=\mathbf{k}[\alpha,\beta)$. Let $(K,d)$ be the chain complex in Example~\ref{example:triangle_filtration}. By Example~\ref{example:underline_hom_of_interval_modules} we have:
\begin{gather*}
{\mathbf{Hom}_{\mathbf{gr}}(K,A)^0:=\uHom(K_0,A)=\mathbf{k}[\alpha-a,\beta-a)\oplus \mathbf{k}[\alpha-b,\beta-b)\oplus \mathbf{k}[\alpha-c,\beta-c),}\\ 
{\mathbf{Hom}_{\mathbf{gr}}(K,A)^1:=\uHom(K_1,A)=\mathbf{k}[\alpha-d,\beta-d)\oplus \mathbf{k}[\alpha-e,\beta-e)\oplus \mathbf{k}[\alpha-f,\beta-f),}\\
{\text{and } \mathbf{Hom}_{\mathbf{gr}}(K,A)^2:=\uHom(K_2,A)=\mathbf{k}[\alpha-g,\beta-g).}
\end{gather*}
By Theorem~\ref{theorem:universal_coefficients_2} and Example~\ref{example:ext_interval_modules} we have that:
\begin{gather*}
H^0(\mathbf{Hom}_{\mathbf{gr}}(K,A))\cong \uHom(H_0(K),A)
=\mathbf{k}[\alpha-a,\beta-a)\oplus \\
\oplus\mathbf{k}[\max\{\alpha-b,\beta-d\},\beta -b)\oplus \mathbf{k}[\max\{\alpha-c,\beta-e\},\beta-c),\\ 
{H^1(\mathbf{Hom}_{\mathbf{gr}}(K,A))\cong \mathbf{Ext}_{\mathbf{gr}}(H_0(K),A)\oplus \uHom(H_1(K),A)
=}\\
{=\mathbf{k}[\alpha-d,\min\{\alpha-b,\beta-d\})\oplus \mathbf{k}[\alpha-e,\min\{\alpha-c,\beta-e\})\oplus \mathbf{k}[\max\{\alpha-f,\beta-g\},\beta-f),}\\ 
{\text{and } H^2(\mathbf{Hom}_{\mathbf{gr}}(K,A))\cong \mathbf{Ext}_{\mathbf{gr}}(H_1(K),A)=\mathbf{k}[\alpha-g,\min\{\alpha-f,\beta-g\}).}
\end{gather*}
This has a geometric interpretation, dual situation to that in Example~\ref{example:tensoring_with_finite_bar}
\end{example}

\begin{example}
Let $(K,d)$ be complex of projective persistence modules coming from a filtration of a simplicial complex and let $A$ be an arbitrary persistence module.  Since persistence modules are $\shtensor$-flat as noted in Theorem~\ref{theorem:sheaf_tensor_is_exact}, we have natural isomorphisms $H_n(K\shtensor A)\cong H_n(K)\shtensor A$, by Theorem~\ref{theorem:universal_coefficients_1}. In particular, if $A$ is an interval module, say $A=\mathbf{k}[I]$, and $H_n(K)\cong {\bigoplus_{j\in J}}\mathbf{k}[I_j]$ is the interval decomposition of $H_n(K)$, then recalling Example~\ref{example:sheaf_tensor_of_interval_modules} we have $H_n(K\shtensor A)\cong{\bigoplus_{j\in J}}\mathbf{k}[I\cap I_j]$.
\end{example}

\begin{example}
Let $(K,d)$ be as in Example~\ref{example:triangle_filtration} and let $A=\mathbf{k}[\alpha,\beta)$ with $b\le \alpha$ and $g\le \beta$. By Example~\ref{example:sheaf_hom_of_interval_modules} we have:
\begin{gather*} 
{\mathbf{Hom}_{\mathbf{sh}}(K,A)^0:=\scHom(K_0,A)=\mathbf{k}[\alpha,\beta)\oplus\mathbf{k}[\alpha,\beta)\oplus\mathbf{k}(-\infty,\beta),}\\
{\mathbf{Hom}_{\mathbf{sh}}(K,A)^1:=\scHom(K_1,A)=\mathbf{k}(-\infty,\beta)\oplus\mathbf{k}(-\infty,\beta)\oplus\mathbf{k}(-\infty,\beta),}\\
{\text{and } \mathbf{Hom}_{\mathbf{sh}}(K,A)^2:=\scHom(K_2,A)=\mathbf{k}(-\infty,\beta).}
\end{gather*}
By Theorem~\ref{theorem:universal_coefficients_2},  Example~\ref{example:sheaf_ext_of_interval_modules}, Example~\ref{example:sheaf_hom_of_interval_modules}, and Theorem~\ref{theorem:universal_coefficients_2}, we have that:
\begin{gather*}
H^0(\mathbf{Hom}_{\mathbf{sh}}(K,A))\cong \scHom(H_0(K),A)
=\scHom(\mathbf{k}[a,\infty),\mathbf{k}[\alpha,\beta))\oplus\\
\oplus\scHom(\mathbf{k}[b,d),\mathbf{k}[\alpha,\beta)\oplus\scHom(\mathbf{k}[c,e),\mathbf{k}[\alpha,\beta))
=\\
{=\mathbf{k}[\alpha,\beta)\oplus 0\oplus 0=\mathbf{k}[\alpha,\beta),}\\ 
{H^1(\mathbf{Hom}_{\mathbf{sh}}(K,A))\cong \mathbf{Ext}_{\mathbf{sh}}(H_0(K),A)\oplus\scHom(H_1(K),A)=}\\
=\mathbf{Ext}_{\mathbf{sh}}(\mathbf{k}[a,\infty),\mathbf{k}[\alpha,\beta))\oplus \mathbf{Ext}_{\mathbf{sh}}(\mathbf{k}[b,d),\mathbf{k}[\alpha,\beta))\oplus \mathbf{Ext}_{\mathbf{sh}}(\mathbf{k}[c,e),\mathbf{k}[\alpha,\beta)\oplus\\\oplus\scHom(\mathbf{k}[f,g),\mathbf{k}[\alpha,\beta))
{=0\oplus \mathbf{k}(-\infty,\alpha)\oplus 0\oplus 0=\mathbf{k}(-\infty,\alpha),}\\
{\text{and } H^2(\mathbf{Hom}_{\mathbf{sh}}(K,A))\cong \mathbf{Ext}_{\mathbf{sh}}(H_1(K),A)=\mathbf{Ext}_{\mathbf{sh}}(\mathbf{k}[f,g),\mathbf{k}[\alpha,\beta))=0.}
\end{gather*}
\end{example}

\section{Persistence modules over finite posets}
\label{section:pers_modules_over_finite_posets}

In this section we apply the Gabriel-Popescu {T}heorem (Theorem~\ref{theorem:Gabriel_Popescu}) to persistence modules over finite {preordered} sets. {It is a classical result that every abelian category is isomorphic to a full subcategory of modules over some ring. Here we do not assume an additional abelian group structure on our preorder $P$ and thus persistence modules are not graded modules over a graded ring. However, the stronger version of the Gabriel Popescu Theorem we show for persistence modules in this section allows us to explicitly construct the ring in question in the above-mentioned isomorphism of categories.}

\begin{definition}
\label{def:compact_object}
Let $\cat{C}$ be a cocomplete abelian category. Then an object $A$ in $\cat{C}$ is \emph{compact} if $\Hom_{\cat{C}}(A,\cdot)$ commutes with direct sums.
\end{definition}

\begin{example}(\cite[Satz 3]{Lenzing1969} and \cite[Introduction]{Breaz2013})
  Let $R$ be a unital ring and $A$ a left $R$-module. Then $A$ is compact if and only if $A$ is finitely presented.
\end{example}

\begin{theorem}(Strengthening of the Gabriel-Popescu Theorem)
\label{theorem:Gabriel_Popescu_compact_version}
 Let $U\in \cat{C}$ be an object in a cocomplete abelian category. Let $R=\emph{{End}}(U)$. Then the following are equivalent:
\begin{itemize}
\item[1)] U is a compact projective generator.
\item[2)] The functor
  $\Hom_{\cat{C}}(U,\cdot)$
  gives us an equivalence of categories between $\cat{C}$ and $\cat{Mod}(R)$.
\end{itemize}
\end{theorem}

\begin{proof}
See \cite[Exercise F, page 106]{freyd1964abelian}.
\end{proof} 

\begin{proposition}
\label{prop:generator_over_finite_poset}
Suppose $\cat{A}$ is a Grothendieck category, let $(P,\leq)$ be a finite {preordered} set and let $\cat{P}$ denote the corresponding 
category.  Let $G$ be a generator of $\cat{A}$. Then the set $\{G[U_a]\}_{a\in P}$ is a family of generators for $\cat{A}^{\mathbf{P}}$. In particular, $U:={\bigoplus_{a\in P}}G[U_a]$ is a generator for $\cat{A}^{\mathbf{P}}$.
\end{proposition}
\begin{proof}
  Repeat the arguments
   in the proof of Proposition~\ref{prop:GeneratorsCogenerators}.
\end{proof}

\begin{proposition}
\label{prop:generator_is_projective}
Let $(P,\leq)$ be a finite {preordered} set and let $R$ be a unital ring. For each $a\in P$, $R[U_{a}]$ is a projective right (and left) persistence module. In particular, $U = \bigoplus_{a \in P}R[U_a]$ is a projective persistence module.
\end{proposition}

\begin{proof}
We prove the statement for the case for right persistence modules. The proof for left persistence modules uses the same arguments.
Given any right exact sequence of right persistence modules $M\xrightarrow{\pi} N\to 0$ and a natural transformation $\alpha:R[U_{a}]\to N$ we need to show that there exists a natural transformation $\hat{\alpha}:R[U_a]\to M$ such that $\pi\hat{\alpha}=\alpha$. Since we have $R$-module homomorphisms $\pi_a:M_a\to N_a$ and $\alpha_a:R[U_a]_a\to N_a$ and $R$ is a projective object in $\ModR$, there is an $R$-module homomorphism$\hat{\alpha}_a:R[U_a]_a\to M_a$ such that $\pi_a\hat{\alpha_a}=\alpha_a$. Define $\hat{\alpha}_b:R[U_a]_b\to M_b$ for $a\le b$ to be the map $M_{a\le b}\hat{\alpha}_aR[U_a]_{a\le b}^{-1}$ (recall that $R[U_a]_{a\le b})$ is the identity map on $R$ so its inverse is defined). If $b\not \in U_a$, let $\hat{\alpha}_b:R[U_a]_b\to M_b$ be the zero map. By construction it follows that the collection $\{\hat{\alpha}_b\}_{b\in P}$ are components of a natural transformation $\hat{\alpha}:R[U_a]\to M$. Furthermore, observe that since all the maps for $R[U_a]_{a\le b}$ are the identity maps of $R$ and $\alpha$ is a natural transformation it follows that $\alpha_b=N_{a\le b}\alpha_aR[U_a]_{a\le b}^{-1}$. On the other hand, for $a\le b$, since $\hat{\alpha}$ and $\pi$ are natural transformations we have $\pi_b\hat{\alpha}_b=\pi_bM_{a\le b}\hat{\alpha}_aR[U_a]_{a\le b}^{-1}=N_{a\le b}\pi_a\hat{\alpha}_aR[U_a]_{a\le b}^{-1}=N_{a\le b}\alpha_aR[U_a]_{a\le b}^{-1}=\alpha_b$. Thus $\pi\hat{\alpha}=\alpha$ and therefore $R[U_a]$ is a projective persistence module.
\end{proof}

\begin{proposition}
\label{prop:generator_is_compact}
Let $(P,\leq)$ be a finite {preordered} set and let $R$ be a unital ring. Then $U = \bigoplus_{a \in P}R[U_a]$ is a compact right (and left) persistence module.
\end{proposition}

\begin{proof}
  We show that there is a canonical isomorphism $\Hom(U,\bigoplus_{i}M_i)\cong \bigoplus_{i}\Hom(U,M_i)$. Given $f:R[U_{a}]\to \bigoplus_i M_i$, since $f$ is a natural transformation and all maps in $R[U_{a}]$ are the identity or the zero map, $f$ is completely determined by its image, $f(1_a)$, where $1_a$ is the multiplicative identity in $R[U_a]_a$. Thus as the codomain is a direct sum we have $f(1_a)=\sum_{i=1}^nm^a_i$ for some $m_i^a\in (M_i)_a$. Define $f_i:R[U_{a}]\to M_i$ by setting $f_i(1_a)=m_i^a$, and extending appropriately.
  Define the map $\psi_a:\Hom(R[U_a],\bigoplus_iM_i)\to \bigoplus_i\Hom(R[U_{a}],M_i)$ by $f\mapsto (f_i)$. This is clearly well-defined and a canonical isomorphism.
The functor $\Hom$ commutes with limits. Since finite direct sums are isomorphic to finite direct products, $\Hom$ commutes with finite direct sums. 
  Thus we have canonical isomorphisms
\begin{multline*}
  \Hom(\bigoplus\limits_{a\in P}R[U_{a}],\bigoplus_iM_i)\cong \bigoplus_{a\in P}\Hom(R[U_a],\bigoplus_iM_i) \\
  \overset{\oplus\psi_a}{\cong}
  \bigoplus_{a\in P} \bigoplus_i 
  \Hom(R[U_{a}],M_i)\cong\bigoplus_i(\bigoplus_{a\in P}R[U_{a}], M_i)
\end{multline*}
\end{proof}

Combining Proposition~\ref{prop:generator_over_finite_poset}, Proposition~\ref{prop:generator_is_projective}, Proposition~\ref{prop:generator_is_compact} and the fact a unital ring $R$ is a generator for the category left/right modules over $R$ (Appendix~\ref{sec:category}) we obtain Theorem~\ref{theorem:gabriel_popescu_for_persistence_modules}.

\begin{theorem}
\label{theorem:gabriel_popescu_for_persistence_modules}
  Let $(P,\leq)$ be a finite {preordered} set with corresponding 
category $\cat{P}$
and let $R$ be a unital ring. 
Let $U = \bigoplus_{a \in P}R[U_a]$. Then 
  $\Hom(U,-): \ModR^{\mathbf{P}} \to \cat{Mod}_{\emph{{End}}(U)}$ is an equivalence of categories.
\end{theorem}

Thus two persistence modules $M$ and $N$ over a finite {preordered} set are isomorphic iff
the $\text{End}(U)$-modules
$\Hom(U,M)$ and $\Hom(U,N)$ are isomorphic.

\section{Enriched category theory and persistence modules} 
\label{section:enriched}

In this section, we assume that $R$ is a commutative unital ring and that $(P,\le,+,0)$ is a {preorder} with a compatible abelian group structure.
That is, $a \leq b$ implies that $a+c \leq b+c$. {The purpose of this section is to observe that persistence modules are symmetric monoidal categories enriched over themselves, with respect to both the graded module and sheaf tensor products. These observations follow from classical results in enriched category theory. We state these results for persistence modules in the hope that doing so will facilitate new computational approaches to topological data analysis. We remark that enriched category theory has been used in applied topology recently \cite{leinster2017magnitude,cho2019quantales,govc2020persistent}. 
}
\subsection{Enriched structure of persistence modules with the graded tensor}

\begin{theorem}
\label{thm:monoidal-cat-gr}
$\big(\ModR^{\mathbf{P}},\grtensor, R[U_0]\big)$
is a symmetric monoidal category.
\end{theorem}

\begin{proof}

Let $M,N \in \ModR^{\mathbf{P}}$ and $s,t \in P$. We have canonical morphisms $\gamma_{s,t}:M_{s}\otimes_RN_{t}\to N_{t}\otimes_R M_{s}$ since $\ModR$ is a symmetric monoidal category, {as $R$ is assumed to be commutative}, with unit $R$ and tensor product $\otimes_R$.  The collection of maps $\gamma_{s,t}$ induces an isomorphism of diagrams $\{(M_{s}\otimes_RN_{t})\}_{s+t\le r}$ and $\{N_{t}\otimes_RM_{s}\}_{s+t\le r}$ and thus a natural isomorphism between their colimits.
  Hence we get a natural isomorphism between $(M\grtensor N)_{r}$ and $(N\grtensor M)_{r}$, called the braiding.
  By the same argument, we obtain an associator and left and right unitors.

  Since it will be used later, let us explicitly define the left unitor.
  The left unitor is a natural isomorphism with components $\lambda_M:R[U_0]\grtensor M \isomto M$, for each persistence module $M$. Let $x^0$ be the generator of $R[U_0]$ and consider $\sum c_ix^{t_i}\grtensor m_i \in R[U_{0}]\grtensor M$ where $c_i\in R$, $t_i\in U_0$, and $m_i \in M_{s_i}$.  Note that by the definition of $\grtensor$ we have that  $\sum c_ix^{t_i}\grtensor m_i=\sum_i x^0\grtensor c_i x^{t_i}\cdot m_i$. Define $\lambda_M(\sum c_ix^{t_i}\grtensor m_i):=\sum_i c_ix^{t_i}\cdot m_i$.

  Since the pentagon identity, triangle identity, and hexagon identity hold in $\ModR$, it follows that they also hold here.
\end{proof}

\begin{proposition}
  There is a functor $\uHom(-,-):(\ModR^{\mathbf{P}})^{\op}\times \ModR^{\mathbf{P}}\to \ModR^{\mathbf{P}}$ given by \[\uHom(M,N)_{s}:=\Hom(M,N(s)),
  \]
  for $s\in P$ and persistence modules $M$ and $N$.
\end{proposition} 

\begin{proof}
  Let $s\in P$. Then $\Hom(M,N(s))$ is the set of natural transformations from $M$ to $N(s)$. This is an $R$-module. Indeed given a natural transformation, we define an $R$ action by an element $r\in R$ to be componentwise multiplication by $r$.
  Whenever $s\le t$
define  $\uHom(M,N)_{s\le t}:\Hom(M,N(s))\to \Hom(M,N(t))$ to be the map \[\uHom(M,N)_{s\le t}(\alpha)=\alpha*\eta_{t-s}.\] 
That is, given a natural transformation $\alpha:M\to N(s)$ compose each component $\alpha_a$ with $N_{a+s\le a+t}$ to get a new natural transformation, namely $\alpha*\eta_{t-s}$. 
Note that we could have precomposed with $M_{a-(t-s)\le a}$ to have a similar construction, however due to the naturality of $\alpha$ this choice would give us the same answer.

\begin{equation*}
\begin{tikzcd}
M_{a-(t-s)}\arrow[rr,"M_{a-(t-s)\le a}"]\arrow[drr,dashed]&&M_a\arrow[d,"\alpha_a"]\arrow[drr,dashed]\\
&&N_{a+s}\arrow[rr,"N_{a+s\le a+t}"]&&N_{a+t}
\end{tikzcd}
\end{equation*}

It remains to show is that
this definition is functorial.

Suppose $\alpha:M\to N$ is a natural transformation of persistence modules $M$ and $N$. Let ${N'}$ be a persistence module. Define $\uHom(\alpha,P):\uHom(N,{N'})\to \uHom(M,{N'})$ by pre-composing with $\alpha$. Namely, for a given 
$\beta:N\to P(s)$
define $\uHom(\alpha,{N'})(\beta)= \beta\alpha$.
Then $\uHom(\gamma\alpha,{N'})(\beta\alpha)=\uHom(\alpha,{N'})\circ \uHom(\gamma,{N'})$ and that $\uHom(-,{N'})(\mathbf{1}_M)=\mathbf{1}_{\uHom(M,{N'})}$. 
To show that $\uHom({N'},-)$ is a functor, define
$\uHom({N'},\alpha)(\beta) = \alpha\beta.$
It follows that $\uHom({N'},-)$ is a functor.
\end{proof}

 By Proposition~\ref{prop:UnderlineHom}, when $M$ is finitely generated $\uHom(M,N)$ is the abelian group of module homomorphisms when we forget the grading.

\begin{proposition}
\label{prop:enriched_category_is_a_category}
There is a category whose objects are persistence modules and whose morphisms are the sets $\uHom(M,N)$.
We denote this category $\uModRP$.
\end{proposition}

\begin{proof}
  Let $\alpha\in \uHom(M,N)$ and $\beta\in \uHom(N,{N'})$.
  Suppose $\alpha$ and $\beta$ are of homogeneous degrees, $s$ and $t$ respectively.
  That is $\alpha:M\Longrightarrow N(s)$ and $\beta:N\Longrightarrow {N'}(t)$. Define the $x$-component of $\beta\circ \alpha$ to be $\beta_{s+x}\circ \alpha_{x}$.
Extend to the general case by linearity.
  Since the composition in $\ModR^{\mathbf{P}}$ is associative, this composition is associative as well.
  
  For the more categorically minded reader, this can be stated using horizontal 
  and vertical compositions of natural transformations.
  We define $\beta\circ \alpha:=(\beta * 1_{\mathcal{T}_{s}})\bullet \alpha$ (Definition~\ref{def:translation_functor}), where $*$ signifies horizontal composition and $\bullet$ a vertical composition. 
Consider the following diagram, where  $N(s)=N\mathcal{T}_s$ and ${N'}(t)={N'}\mathcal{T}_t$.
  
\begin{equation*}
\begin{tikzcd}[row sep=2em,column sep=2em]
\mathbf{P}\arrow[d,"="]\arrow[rrr,"M",""{name=V,below}]&&& \ModR\arrow[d,"="]\\
\mathbf{P}  \arrow[rrr, bend left=25, "N(s)" below, ""{name=K}] \arrow[Rightarrow, from=V,to=K, "\alpha" right]\arrow[d,"="]  \arrow[r, "\mathcal{T}_s",""{name=U,below}]   & \mathbf{P} \arrow[rr, "N", ""{name=J,below}]\arrow[d,"="]  && \ModR\arrow[d,"="]\\
\mathbf{P} \arrow[r,"\mathcal{T}_s" below,""{name=L,above}]\arrow[Rightarrow,from=U,to=L, "1_{\mathcal{T}_s}" right] &\mathbf{P}\arrow[r,"\mathcal{T}_t" below]&\mathbf{P}\arrow[Rightarrow, from=J, "\beta", right]\arrow[r, "{N'}" below]&\ModR
\end{tikzcd}
\end{equation*}

For every persistence module $M$ we have an identity morphism, $\mathbf{1}_M$ the identity morphism in $\ModR^{\mathbf{P}}$ viewed as a morphism in the new category. The identity axiom in $\uModRP$ follows from the identity axiom in $\ModR^{\mathbf{P}}$.
\end{proof}

\begin{example}
Let $M$ be a persistence module and $s \in P$. Then the translations by $s$ and $-s$, show that  $M(s)$ and $M$ are isomorphic in $\uModRP$. That is, translations are isomorphisms.
\end{example}

\begin{proposition}
$\uModRP$ is an additive category.
\end{proposition}

\begin{proof}
  Each $\uHom(M,N)$ is an abelian group as it is a graded direct sum of abelian groups.
  Our definition of composition in $\uModRP$  is bilinear.
  The zero persistence module is the $0$ object.
  The coproduct of persistence modules $M$ and $N$ is $M\oplus N$. The product of persistence modules $M$ and $N$ is $M\times N$ and is canonically isomorphic to $M\oplus N$.  
\end{proof}

\begin{theorem}
\label{theorem:enriched_category_is_monoidal}
$\Big(\uModRP,\grtensor, R[U_{0}]\Big)$ is a symmetric monoidal category.
\end{theorem}

\begin{proof}
Let the braiding, associator, left and right unitor {morphisms} be those from the symmetric monoidal category $\left( \ModR^{\mathbf{P}}, \grtensor, R[U_0]\right)$ (Theorem~\ref{thm:monoidal-cat-gr}).
It remains to show that these commute with the larger set of morphisms in $\uModRP$.

Consider the braiding. Let $\varphi:M \to M'$ and $\psi: N \to N'$ in $\uModRP$.
Since we have the following commutative diagrams in $\ModR$,
\begin{figure}[H]
\centering
\begin{tikzcd}
  M_s \tensor_R N_t \ar[r,"\gamma"] \ar[d,"\varphi_s^a \tensor_R \psi_t^b"] &  N_t \tensor_R M_s \ar[d,"\psi_t^b \tensor_R \varphi_s^a"] \\
  M'_{s+a} \tensor_R N'_{t+b} \ar[r,"\gamma"] & N'_{t+s} \tensor_R M'_{s+a}
\end{tikzcd}
\end{figure}
it follows that the braiding is natural in $\uModRP$. Naturality of the associator and left and right unitors follows similarly.
\end{proof}

\begin{theorem}
  \label{theorem:graded_module_adjunction}
  $\Big(\uModRP,\grtensor, R[U_0]\Big)$ is a closed symmetric monoidal category. That is, for all $N \in \uModRP$, $-\grtensor N$ has right adjoint $\uHom(N,-):\uModRP \to \uModRP$.
  That is, for any persistence modules $M,N$ and ${N'}$ there exists a natural (in all arguments) isomorphism $\uHom(M\grtensor N,{N'})\cong \uHom(M,\uHom(N,{N'}))$. Furthermore this isomorphism is a morphism of degree zero, i.e., a natural transformation. In particular, $\Big(\ModR^{\mathbf{P}},\grtensor,R[U_0]\Big)$ is a closed symmetric monoidal category.
\end{theorem}

\begin{proof}
  Define $\psi_r$ to be the following composition of isomorphisms. The first two and last isomorphisms are by the categorical definitions of $\uHom(-,-)$ and $- \grtensor -$ (Definition~\ref{def:graded_tensor_product} and Proposition~\ref{prop:limit_characterization_of_underline_hom}). The third and the third last isomorphism are due to the fact that $\Hom_R(-,-)$ preserves limits in both variates. Since it is contravariant in the first {variable} this means that it sends colimits to limits in $\ModR$.
The fourth isomorphism is the tensor-hom adjunction in $\ModR$. The remaining fifth isomorphism follows from reindexing. We have $p+q \leq -s \leq t-r$ which may be rewritten as $p+q-t \leq -r$. Substituting variables, we have $-s-p-q \leq -r$ which may be rewritten as $p+q \geq t \geq -s+r$.

\begin{equation*}
\begin{tikzcd}
\uHom(M\grtensor N,{N'})_{r}  \arrow[d,"\cong"] & \uHom(M,\uHom(N,{N'}))_{r} \\
\varprojlim\limits_{s+t\ge r}\Hom_R((M\grtensor N)_{-s},{N'}_{t})\arrow[d,"\cong"] & \varprojlim\limits_{s+t\ge r}\Hom_R(M_{-s},\uHom(N,{N'})_{t})\arrow[u,"\cong"]\\
\varprojlim\limits_{s+t\ge r}\Hom_R({\colim\limits_{p+q\le -s}}M_{p}\otimes_R N_{q},{N'}_{t})\arrow[d,"\cong"] & \varprojlim\limits_{s+t\ge r}\Hom_R(M_{-s},\varprojlim\limits_{p+q\ge t}\Hom_R(N_{-p},{N'}_{q}))\arrow[u,"\cong"]\\
\varprojlim\limits_{s+t\ge r}(\varprojlim\limits_{p+q\le -s}\Hom_R(M_{p}\otimes_R N_{q},{N'}_{t}))\arrow[d,"\cong"]& \varprojlim\limits_{s+t\ge r}(\varprojlim\limits_{p+q\ge t}\Hom_R(M_{-s},\Hom_R(N_{-p},{N'}_{q})))\arrow[u,"\cong"]\\
\varprojlim\limits_{s+t\ge r}(\varprojlim\limits_{p+q\le -s}\Hom_R(M_{p},\Hom_R(N_{q},{N'}_{t})))\arrow[ru,"\cong"]
\end{tikzcd}
\end{equation*}

Note that all of these isomorphisms are natural. The last statement of the theorem follows by setting $r=0$ and thus getting the following natural isomorphism:
\[\psi_0:\uHom(M\grtensor N,{N'})_0:=\Hom(M\grtensor N,{N'})\isom \uHom(M,\uHom(N,{N'}))_0:=\Hom(M,\uHom(N,{N'})),\]
which gives us an adjunction between $\grtensor$ and $\uHom$ in $\ModR^{\mathbf{P}}$.
\end{proof}

Every closed symmetric monoidal category is enriched over itself, see for example \cite[Section 1.6]{kelly:enriched}.
Thus, we have the following corollary.
We expect that it will be useful in applications that the respective hom objects between persistence modules are themselves persistence modules, from which one may compute invariants such as persistence diagrams.

\begin{corollary}
\label{thm:enriched_over_itself}
The category $\ModR^{\mathbf{P}}$ is enriched over $\big(\ModR^{\mathbf{P}},\grtensor, R[U_0]\big)$. The category $\uModRP$ is enriched over $(\uModRP,\grtensor, R[U_0])$.
\end{corollary}


  Note that the statements in Corollary~\ref{thm:enriched_over_itself} and Theorem~\ref{theorem:graded_module_adjunction} are not true if we replace $\ModR$  by $\mathbf{mod}_R$, the category of finitely generated $R$-modules.
  The reason is that when applied to persistence modules valued in $\mathbf{mod}_R$, $\uHom$ does not always give a  persistence module valued in $\mathbf{mod}_R$, as shown in 
the following example.

\begin{example}
  \label{example:not_ptwfd}
Let $\mathbf{vect}_{\mathbf{k}}$ denote the category of finite-dimensional $\mathbf{k}$-vector spaces and $\mathbf{k}$-linear maps. Let $M:(\mathbb{R},\le)\to \mathbf{vect}_{\mathbf{k}}$ be the one-parameter persistence module given by $M_a = \mathbf{k}$ if $a \in \Z$ and $M_a = 0$ otherwise.
By functoriality, the maps $M_{a\le b}$ have to be zero except when $a=b\in \mathbb{Z}$. Hence any collection of linear maps $\{f_{a}:M_a\to M_a\}$ will give us a natural transformation $f:M\to M$ as there are no restrictions for the appropriate squares to commute. In particular let, $\alpha^n:M\to M$ be a natural transformation such that 
$\alpha_a^n = 1_{\mathbf{k}}$ if $a = n$ and $\alpha_a^n = 0$ otherwise.
Then $\uHom(M,M)_0$ is an infinite dimensional vector space as the collection $\{\alpha^n\}_{n\in \mathbb{Z}}$ is linearly independent.
\end{example}

\subsection{Enriched structure of persistence modules with the sheaf tensor}

Let $M$ be a right $R_{P}$-module and let $N$ be a left $R_{P}$-module, where $P$ is given the Alexandrov topology (Section~\ref{sec:sheaves-cosheaves}).
Recall that we have a sheaf tensor product $M \shtensor N$ (Section~\ref{sec:shtensor}).
Viewed as a graded module, $M\shtensor N \isom \bigoplus_{a \in P} M(U_a) \otimes_R N(U_a) \isom \bigoplus_{a\in P}M_a\otimes_R N_a$.

\begin{theorem}
\label{thm:monoidal-cat-sh}
$\big(\ModR^{\mathbf{P}},\shtensor, R[P]\big)$
is a symmetric monoidal category.
\end{theorem}

\begin{proof}  
Since $\shtensor$ is a pointwise tensor product of $R$-modules,
  the axioms for a symmetric monoidal category will hold pointwise and thus can be assembled to obtain the desired axioms for persistence modules.
\end{proof}

\begin{theorem}
\label{theorem:sheaf_theory_adjunction}
$\Big(\ModR^{\mathbf{P}},\shtensor,R[P]\Big)$ is a closed symmetric monoidal category.
\end{theorem}

\begin{proof}
  See Proposition~\ref{prop:sheaf-tensor-hom-adjunction}.
\end{proof}

\begin{corollary}
\label{thm:enriched_over_itself_sheaf}
The category $\ModR^{\mathbf{P}}$ is enriched over $\big(\ModR^{\mathbf{P}},\shtensor, R[P]\big)$. 
\end{corollary}

\begin{proof}
{Same arguments that justify Corollary~\ref{thm:enriched_over_itself} may be used.}
\end{proof}

\appendix

\section{Category theory}
\label{sec:category}
We review some notions from category theory.
For more details, see for example \cite{popescu1973abelian,riehl2017category}. 

Let $\mathcal{C}$ be a category.  A family $\{U_i\}_{i\in I}$ of objects from $\mathcal{C}$ is called a family of $\emph{generators}$ of $\mathcal{C}$ if for any pair $(A,B)$ of objects in $\mathcal{C}$ and for any two distinct morphisms $f,g:A\to B$, there is an index $i_0$ and a morphism $h:U_{i_0}\to A$ such that $fh\neq gh$. We say $\{U_i\}_{i\in I}$ is a set of $\emph{cogenerators}$ of $\mathcal{C}$ if the family $\{U_i^{op}\}_{i\in I}$ is a set of generators of $\mathcal{C}^{op}$. If the families in question are singleton sets, we say they are a generator (resp. cogenerator) of $\mathcal{C}$.
In the category $\mathit{Set}$, the singleton set $\{*\}$ is a generator and the 2 point set $\{*_1,*_2\}$ is a cogenerator. In the category of abelian groups $\mathbf{Ab}$ the group $\mathbb{Z}$ is a generator. More generally, whenever we have a unital ring $\mathcal{R}$ and the category of left modules over $\mathcal{R}$ , $\emph{Mod}_{\mathcal{R}}$, the ring as a module over itself is a generator of $\emph{Mod}_{\mathcal{R}}$. In particular if $\mathcal{R}$ is a field, say $\mathbf{k}$ , then $\mathbf{Vect}_{\mathbf{k}}$ the category of vector spaces over $\mathbf{k}$ is a category with generator $\mathbf{k}$.
Grothendieck categories are abelian categories with a few extra axioms that guarantee existence of injective and projective resolutions, for more details see \cite{popescu1973abelian,grothendieck1957quelques, freyd1964abelian}. 
A \emph{Grothendieck category} $\mathcal{C}$ is a category satisfying the following axioms:
\begin{enumerate*}
\item $\mathcal{C}$ is an abelian category.
\item $\mathcal{C}$ has a generator.
\item $\mathcal{C}$ contains all small colimits (colimits of diagrams indexed by a category with a set of objects).
\item Taking colimits of diagrams of short exact sequences produces a short exact sequence.
\end{enumerate*}

The following is due to Grothendieck and is presented in for example 
\cite[Proposition 1.9.1]{grothendieck1957quelques} or \cite[Chapter 2, Proposition 8.2]{popescu1973abelian}. 

\begin{proposition}
\label{prop:Generator}
Let $\mathcal{C}$ be an abelian category category with infinite direct sums and $\{U_i\}_{i\in I}$ a set of objects of $\mathcal{C}$. The following are equivalent:
\begin{enumerate}
\item The given set is a set of generators of $\mathcal{C}$.
\item The object $U:=\coprod_{i\in I}U_i$ is a generator of $\mathcal{C}$.
\item For any object $A$ in $\mathcal{C}$, there is a set $J$ and an epimorphism: $U^{(J)}\to A$.
\end{enumerate}
\end{proposition}

\begin{proposition}[{{\cite[Proposition 1.8]{grothendieck1957quelques}}}]
\label{proposition:grothendieck_functor_categories}
Let $(P,\leq)$ be a pre-ordered set and $\cat{P}$ be the corresponding
  category.
Let $\mathcal{C}$ be a  category. If $\mathcal{C}$ is an additive/abelian/Grothendieck category then  $\mathcal{C}^{(\mathbf{P},\le)}$ is also an additive/abelian/Grothendieck category.
\end{proposition}

\begin{theorem}[{{\cite[Theorem 1.10.1]{grothendieck1957quelques}}}]
If a category $\mathcal{C}$ is a Grothendieck category then any $A\in \mathcal{C}$ has a monomorphism into an injective object.
\label{theorem:injectives}
\end{theorem}

The category of abelian groups $\mathbf{Ab}$ (or more generally the category of modules over a unital ring) is a Grothendieck category.
In particular, if we have a unital ring $R$ then $\ModR$ and $\RMod$ are Grothendieck categories.  
 Note that if we consider the category $\mathbf{mod}_R$ of right $R$-modules of finite rank it is abelian but not a Grothendieck category, as a coproduct (direct sum) of an infinite family of finite rank modules is not a finite rank module.

\begin{proposition}[{{\cite[Chapter 3, Lemma 3.1]{popescu1973abelian}, \cite[Lemma 1]{grothendieck1957quelques}}}]
\label{prop:BaerCriterionForCategories}
Let $\mathcal{C}$ be a Grothendieck category, $U$  a generator and $E$ and object of $\mathcal{C}$. Then
$E$ is an injective object if and only if
for any monomorphism $\iota:U'\to U$ and for any morphism $f:U'\to E$ there exists a morphism $\bar{f}:U\to E$ such that $\bar{f}\iota=f$.
\end{proposition}

\cref{prop:BaerCriterionForCategories} provides a simpler criterion for an object to be injective in a Grothendieck category. In particular instead of checking diagrams with arbitrary monomorphisms $M\to N$, we need only check diagrams with monomorphisms into the generator
This generalizes the Baer Criterion in module theory. The Baer Criterion for graded modules is presented in \cref{theorem:BaerCriterionGradedModules}.

\section{Graded module theory} 
\label{Graded Module Theory}

The purpose of this section is to introduce the reader to the basics of graded module theory. The literature on graded modules is bountiful but our main reference is \cite{hazrat2016graded}.

Let $\Gamma$ be a group. A $\Gamma\emph{-graded ring}$ is a ring $S=\bigoplus\limits_{g\in \Gamma}S_g$, where $S_g$ is an additive subgroup of $S$ and $S_gS_h\subset S_{gh}$.

Let $S$ be a $\Gamma$-graded ring. A $\emph{graded left S-module}$ is a left $S$-module $M=\bigoplus\limits_{g\in \Gamma}M_g$, where $M_g$ is an additive subgroup of $M$ and $S_gM_h\subset M_{gh}$. Let $M$ and $N$ be $\Gamma$-graded $S$-modules.  A $\Gamma\emph{-graded S-module homomorphism}$ between $M$ and $N$ is a module homomorphism $\alpha:M\to N$, such that $\alpha(M_g)\subset N_g$.
One can analogously define graded right $\mathcal{S}$-module and corresponding graded module homomorphism. Assuming the ring $S$ is commutative we stop differentiating between left and right. The set $M^h=\bigcup_{g\in \Gamma}M_g$ is called the set of \emph{homogeneous} elements of $M$.

For graded $S$-modules $M$ and $N$, a \emph{graded} $S$-\emph{module  homomorphism of degree} $\epsilon, \epsilon\in \Gamma$, is a $S$-module homomorphism $f:M\to N$, such that $f(M_g)\subset N_{g\epsilon}$ for any $g\in \Gamma$. Let $\text{Hom}_S(M,N)_{\epsilon}$ be the subgroup of $\text{Hom}_S(M,N)$, the group of non-graded module homomorphisms between $M$ and $N$, consisting of all $S$-graded module homomorphisms of degree $\epsilon$.

A graded module $M$ is \emph{finitely generated} if it is finitely generated as a module.
A $\Gamma$-graded (left) $\mathcal{S}$-module $M$ is called a \emph{graded-free} $\mathcal{S}$-module if $M$ is a free left $\mathcal{S}$-module with a homogeneous base.
Let $M$ be a $\Gamma$-graded module over a $\Gamma$-graded ring $S$. Let $g\in \Gamma$. Define a new module $M(g)$ by setting its graded by $M(g)_h:=M_{gh}$. The action of $S$ on $M(g)$ is induced from the action of $S$ on $M$.

\begin{theorem}(Baer Criterion for Graded Modules)
\label{theorem:BaerCriterionGradedModules}
Let $E$ be a $\Gamma$-graded module over the $\Gamma$-graded ring $\mathcal{S}$. Then $E$ is injective if and only if given any monomorphism $i:I\to \mathcal{S}(g)$ and a graded module homomorphism $f:I\to E$, there exists an $\overline{f}:\mathcal{S}(g)\to E$ such that $f=\overline{f}i$.
\end{theorem}

Note that this is a specific example of \cref{prop:BaerCriterionForCategories}, as the generator of the category of $\Gamma$-graded modules over the $\Gamma$-graded ring $\mathcal{S}$, with graded module homomorphisms as the morphisms, is $\bigoplus\limits_{g\in\Gamma}\mathcal{S}(g)$. 

\section{Sheaf theory} 
\label{sec:sheaves}

We introduce some notions from sheaf theory. For more details see \cite[Chapter 1]{Bredon:SheafTheory} and
\cite[Chapter 2]{MR1074006}. Throughout this section, $X$ is a topological space. 

Given a presheaf $F$ on $X$ there exists a sheaf $F^+$ and a morphism $\theta:F\to F^+$ such that for any sheaf $G$ the homomorphism given by $\theta$:
 \[\emph{Hom}_{\mathbf{Sh}(X)}(F^+,G)\to \emph{Hom}_{\mathbf{PSh}(X)}(F,G)\]
 is an isomorphism. In other words, $F\mapsto F^+$ is the left adjoint functor of the inclusion functor $Sh(X)\to PSh(X)$. Moreover, $(F^+,\theta)$ is unique up to isomorphism, and for any $x\in X$, $\theta_x:F_x\to F^+_x$ is an isomorphism. The sheaf $F^+$ is called \emph{the sheaf associated to $F$} or \emph{sheafification of $F$.}

Given an abelian group $A$, we denote by $A_X$ the sheaf associated to the presheaf $U\mapsto A$, where $U$ is open in $X$, and we say $A_X$ is the \emph{constant sheaf} on $X$ with stalk $A$.

Let $F$ be a sheaf on  
$X$. We say $F$ is \emph{flabby} if the map $F(U\subset X):F(X)\to F(U)$ is surjective for all open $U\subset X$.

Let $\mathcal{R}$ be a sheaf of rings on $X$. The pair $(X,\mathcal{R})$ is called a \emph{ringed space}. A left $\mathcal{R}$-module $M$ is a sheaf of abelian groups $M$ such that for every open $U\subset X$, $M(U)$ is a left $\mathcal{R}(U)$-module, and for any inclusion $V\subset U$, $V$ and $U$ open, the restriction morphism is compatible with the structure of the module, that is, $M(V\subset U)(sm)=\mathcal{R}(V\subset U)(s)\cdot M(V\subset U)(m)$ for every $s\in \mathcal{R}(U)$ and $m\in M(U)$. Define right $\mathcal{R}$-modules in the obvious way and morphisms between left(right) modules is a natural transformation compatible with the structure of the module. Denote these sets of natural transformations by $\text{Hom}_{\mathcal{R}}(M,N)$. We denote the category of right $\mathcal{R}$-modules by $\shModR{\mathcal{R}}$, and the category of left $\mathcal{R}$-modules by $\shRMod{\mathcal{R}}$.

  Denote by $\mathbb{Z}_X$ the sheaf associated to the constant presheaf $U\mapsto \mathbb{Z}$ for every open $U\subset X$. Then $\mathbb{Z}_{X}$-modules are precisely sheaves with values in abelian groups, i.e, $\mathbf{Mod}(\mathbb{Z}_X)=\mathbf{Sh}(X)$. More generally, define $R_X$ to be the sheaf associated to the constant presheaf $U\mapsto R$ for every open $U\subset X$.
  For example, we have the constant sheaf $\mathbf{k}_{\R^n}$.

\label{def:sheaf_hom}
Let $\mathcal{R}$ be a sheaf of rings and let $F$ and $G$ be two left $\mathcal{R}$-modules. Then the presheaf $\scHom(F,G)$ defined by $\scHom(F,G)(U):=\text{Hom}_{\mathcal{R}|_U}(F|_U,G|_U)$ is a sheaf of abelian groups, in particular a left $\mathcal{R}$-module. 

Let $F$ be a right $\mathcal{R}$-module and $G$ be a left $\mathcal{R}$-module. Define $F\otimes_{\mathcal{R}}G$ to be the sheaf associated to the presheaf of abelian groups $U\mapsto F(U)\otimes_{\mathcal{R}(U)}G(U)$, and call $F\otimes_{\mathcal{R}}G$ the tensor product of $F$ and $G$ over $\mathcal{R}$.

\begin{proposition}[{\cite[Proposition 2.2.9]{MR1074006}}] \label{prop:sheaf-tensor-hom-adjunction}
Let $X$ be a topological space. Let $\mathcal{R}$ be a sheaf of rings on $X$, $\mathcal{S}$ a sheaf of commutative rings and $\mathcal{S}\to\mathcal{R}$ a morphism of sheaves of rings such that its image is contained in the center of $\mathcal{R}$. Let $F$ and $G$ be two $\mathcal{R}$-modules and $H$ and $\mathcal{S}$-module. Then one has canonical isomorphisms:
\[\scHom_{\mathcal{R}}(H\otimes_{\mathcal{S}}F,G)\cong \scHom_{\mathcal{R}}(F,\scHom_{\mathcal{S}}(H,G))\cong \scHom_{\mathcal{S}}(H,\scHom_{\mathcal{R}}(F,G))\]
Note that by taking global sections of each of the sheaves above, we get:
\[
\Hom_{\mathcal{R}}(H\otimes_{\mathcal{S}}F,G)\cong \Hom_{\mathcal{S}}(H,\scHom_{\mathcal{R}}(F,G))
\]
In other words, $-\otimes_{\mathcal{S}}F$ is the left adjoint of $\scHom_{\mathcal{R}}(F,-)$.
\end{proposition}

\begin{proposition}[{\cite[Section 2.2]{MR1074006}}]
\label{prop:sheaf_tensor_is_right_sheaf_hom_is_left_exact}
The functor $\scHom(-,-)$ is left exact with respect to each of its arguments and the functor $-\otimes_{\mathcal{R}}-$ is right exact with respect to each of its arguments.
\end{proposition}

\section{Homological algebra}
\label{sec:homological-algebra}

We introduce some homological algebra. For more details, see \cite{rotman2008introduction,MR1269324,MR1074006,MR2050072}.

The \emph{homotopy category} of an abelian category $\mathbf{A}$, $K(\mathbf{A})$, is obtained from the category of chain complexes valued in $\mathbf{A}$, $C(\mathbf{A})$ by identifying all morphisms that are chain homotopic to $0$. Furthermore, by formally inverting quasi isomorphisms we obtain the \emph{derived category} of $\mathbf{A}$, $D(\mathbf{A})$.

Assuming that $\mathbf{A}$ has enough projectives/injectives we are able to construct projective/injective resolutions which are used to compute derived functors. Given an object $A$ of $\mathbf{A}$, we consider it to be a chain complex concentrated in degree $0$. 
For such a chain complex there is a quasi-isomorphism to a chain complex of injective objects of $\mathbf{A}$ concentrated in non-negative degrees given by an injective resolution of $A$:
$\cdots\to 0\to E^0\to E^1\to \cdots.$
In the derived category, $D(\mathbf{A})$, $A$ is isomorphic to this injective resolution.
  Given a left-exact functor $F:\cat{A} \to \cat{B}$ we compute the $i$-th right derived functor of $F$ by calculating
the $i$-th cohomology group of the chain complex:
$\cdots \to 0\to F(E^0)\to F(E^1)\to \cdots.$

Similarly, we use projective resolutions to compute left derived functors of right-exact functors.

\label{def:monoidal_prod_of_chain_complexes}
Let $\mathbf{A}$ be an abelian category and consider $\mathcal{C}(\mathbf{A})$ the category of chain complexes valued in $\mathbf{A}$. Suppose $\mathbf{A}$ comes equipped with a monoidal product, say $\otimes_*$ and an adjoint for the monoidal product, say $\text{Hom}^*$.

Consider $(A,d_A)$ and $(B,d_B)$ in $\mathcal{C}(\mathbf{A})$.
  The tensor product 
  $A \otimes_* B$ is the chain complex given by
  $(A\otimes_* B)_n := \bigoplus\limits_{p+q=n}(A_p\otimes_* B_q)$ and the differential given on elements $x \tensor_* y$ in homogeneous degree by
$d(x \tensor_* y) = d_Ax\otimes_* y+(-1)^{|x|}x\otimes_* d_By$.
The hom chain complex 
  $\Hom^*(A,B)$ given by
  $\Hom^*(A,B)_n=\prod\limits_{p+q=n}\Hom^*(A_{-p},B_q)$, with differential
  $d$ 
  defined on homogeneous  $f \in \Hom^*(A,B)_n$ by
  $df := d_B \circ f - (-)^n f \circ d_A$.

\subsection*{Acknowledgments}

We would like to thank the anonymous referee for many helpful comments and corrections that greatly improved our paper. We would like to thank Michael Catanzaro for his insights during the start of this project involving the graded module point of view of persistence modules. We would like to thank Ezra Miller for help with examples involving multi-parameter injective persistence modules (for example, Example~\ref{ex:ezra}), for pointing us to the literature on Matlis duality, and for many helpful comments on an earlier version of this paper. We would also like to thank the organizers of the Conference: Bridging Statistics and Sheaves at the IMA in May 2018 for helpful discussions on the sheaf point of view of persistence modules.
This research was partially supported by the Southeast Center for Mathematics and Biology, an NSF-Simons Research Center for Mathematics of Complex Biological Systems, under National Science Foundation Grant No. DMS-1764406 and Simons Foundation Grant No. 594594.
This material is based upon work supported by, or in part by, the Army Research Laboratory and the Army Research Office under contract/grant number
W911NF-18-1-0307.

 \printbibliography

\end{document}